\documentclass[a4paper,10pt,leqno]{amsart}
\usepackage[margin=3cm]{geometry}
\usepackage{enumerate, amsmath, amsfonts, amssymb, amsthm, mathtools, thmtools,xparse,array,colortbl,hyperref,url,hypcap,tikz,xspace,accents,enumitem,letltxmacro,comment,bbm,ytableau}
\hypersetup{colorlinks=true, citecolor=darkblue, linkcolor=darkblue,urlcolor=darkblue}
\usepackage[procnames]{listings}
\usepackage[colorinlistoftodos]{todonotes}
\usetikzlibrary{calc,through,backgrounds,shapes,matrix,math,trees}
\definecolor{darkblue}{rgb}{0.0,0,0.7} 
\newcommand{\darkblue}{\color{darkblue}} 

\numberwithin{equation}{section}
\newtheorem{theorem}{Theorem}[section]
\newtheorem{proposition}[theorem]{Proposition}
\newtheorem{corollary}[theorem]{Corollary}

\theoremstyle{definition}
\newtheorem{definition}[theorem]{Definition}
\newtheorem{example}[theorem]{Example}
\theoremstyle{remark}
\newtheorem{remark}[theorem]{Remark}

\usepackage[capitalize]{cleveref}
\crefformat{footnote}{#2\footnotemark[#1]#3}
\crefformat{conjecture}{Conjecture~#2#1#3}
\crefname{figure}{Figure}{Figures}
\crefformat{openproblem}{Problem~#2#1#3}

\newcommand{\len}{\mathrm{len}}

\newlist{thmlist}{enumerate}{1}
\setlist[thmlist]{label=(\arabic{thmlisti}), ref=\thetheorem(\arabic{thmlisti}),noitemsep}
\addtotheorempostheadhook[theorem]{\crefalias{thmlisti}{theorem}}

\usepackage{todonotes}

\newcommand{\defn}[1]{\emph{\darkblue #1}}
\usepackage[all]{xy}
\usepackage[T1]{fontenc}

\renewcommand{\d}{d}
\renewcommand{\SS}{Q}
\newcommand{\M}{\mathcal{M}}
\newcommand{\E}{\mathbb{E}}
\newcommand{\Mm}{\boldsymbol{\M}}
\newcommand{\Nn}{\boldsymbol{\mathcal{N}}}
\newcommand{\Vv}{\boldsymbol{\mathcal{V}}}
\newcommand{\Mb}{\mathcal{W}}
\newcommand{\MMM}{\mathbb{M}}
\newcommand{\WWW}{\mathbb{W}}
\newcommand{\Mbm}{\boldsymbol{\Mb}}
\newcommand{\bija}{A}

\renewcommand{\P}{P}
\newcommand{\Pt}{\P^\times}
\newcommand{\Pl}{\P^\leq}
\newcommand{\W}{\mathcal{W}}
\newcommand{\Ww}{\boldsymbol{\W}}

\newcommand{\NC}{\mathrm{NC}}
\newcommand{\Var}{\mathrm{Var}}

\title[Wiener Indices of Minuscule Lattices]{Wiener Indices of Minuscule Lattices}

\author[C.~Defant]{Colin Defant}
\address[C.~Defant]{Massachusetts Institute of Technology}
\email{colindefant@gmail.com}

\author[V.~F\'eray]{Valentin F\'eray}
\address[V.~F\'eray]{Universit\'e de Lorraine, CNRS, IECL, F-54000, Nancy, France}
\email{valentin.feray@univ-lorraine.fr}

\author[P.~Nadeau]{Philippe Nadeau}
\address[P.~Nadeau]{Univ Lyon, CNRS, Université Claude Bernard Lyon 1, Institut Camille Jordan, F-69622 Villeurbanne Cedex, France}
\email{nadeau@math.univ-lyon1.fr}

\author[N.~Williams]{Nathan Williams}
\address[N.~Williams]{University of Texas at Dallas}
\email{nathan.williams1@utdallas.edu}

\subjclass[2000]{}

\begin{document}

\begin{abstract}
The \emph{Wiener index} of a finite graph $G$ is the sum over all pairs $(p,q)$ of vertices of $G$ of the distance between $p$ and $q$.  When $P$ is a finite poset, we define its \emph{Wiener index} as the Wiener index of the graph of its Hasse diagram.  In this paper, we find exact expressions for the Wiener indices of the distributive lattices of order ideals in minuscule posets.
For infinite families of such posets, we also provide results on the asymptotic distribution of the distance between two random order ideals.
\end{abstract}

\maketitle

\section{Introduction}

\subsection{Background: the Wiener index of the noncrossing partition lattice}

Let $\NC(n)$ be the lattice of noncrossing partitions
of $n$.
In the paper~\cite{goulden2020asymptotics},
motivated by problems about meanders and meandric systems, Goulden, Nica, and Puder raised the following question: what is the average distance between two (uniform)
random partitions in $\NC(n)$?
The question was answered for large $n$ by Th\'evenin
and the second author in~\cite{feray2022components},
where it was proved that this average distance behaves as $\kappa\, n$ for some constant $\kappa$.

It is natural to ask similar questions for other families of posets, looking either for an exact nice formula or for an asymptotic answer.
When the number of elements is known (which is the case for $\NC(n)$),
one can equivalently ask for the sum of distances between all pairs of elements.

In general, let $G=(V,E)$ be a finite connected graph, and for $p,q\in V$, write $\d(p,q)$ for the distance in $G$ from $p$ to $q$.  The \defn{Wiener index} of $G$ is defined to be
      \begin{equation}\label{eq:dist}
            \d(G):=\sum_{(p,q) \in V \times V} \d(p,q).
      \end{equation}
This definition has its origin as the \emph{Wiener index} predicting the boiling point of certain organic compounds~\cite{wiener1947structural,wikiwiener,rouvray2002rich}, and it has also been called the \emph{distance} of the graph $G$~\cite{entringer1976distance}.  When $P$ is a poset, we define $\d(P):=\d(G(P))$ for the Wiener index of the Hasse diagram $G(P)$ of $P$ (that is, the vertices of $G(P)$ are the elements of $P$, and there is an edge in $G(P)$ between $p$ and $q$ when there is a cover relation in $P$ between $p$ and $q$). 

In the case of the noncrossing partition lattice, the results in~\cite{feray2022components} imply that
\[d(\NC(n)) \sim |\NC(n)|^2 \kappa \, n \sim \frac{\kappa\, 8^n}{\pi n^2},\]
but no exact enumeration appears possible.
(There does not even seem
to be a simple formula for $\kappa$;
see the discussion in~\cite{feray2022components}
and the related open problem in~\cite{ober2022}.)

\subsection{Wiener indices of other lattices}
Computer experiments suggest that there are few nontrivial families of graphs $\{G_n\}_{n\geq 1}$ of combinatorial objects for which it is possible to find \emph{exact} formulas for the Wiener index.  
First, there are families with elementary exact solutions, for which $|G_n|$ is a relatively small polynomial in $n$ (for example: path graphs, grid graphs, etc.), 
or in which each graph $G_n$ has a transitive underlying symmetry group (for example: the weak order on a finite Coxeter group, a boolean lattice or hypercube, etc.).  A short list of examples is given in~\cite{weissteinwiener}. 

There is, however, one class of posets in algebraic combinatorics that demonstrates consistently exceptional enumerative behavior: the minuscule lattices~\cite{proctor1984bruhat}.  For example, both the number of elements and the number of maximal chains in a minuscule lattice have simple (uniformly stated and proven) product formulas, and the minuscule lattices are well understood from the perspective of dynamical algebraic combinatorics (promotion, rowmotion, etc.)~\cite{striker2012promotion,hopkins2020order}.

It turns out that the Wiener indices of minuscule lattices also admit simple exact formulas,
and one of the goal of the present paper is to establish such formulas.

For completeness, we first recall the definition of minuscule lattices
(note that we will only use here their classification, and not the algebraic definition).  Let $\mathfrak{g}$ be a complex simple Lie group with Weyl group $W$. Fix a set $\Phi^+$ of positive roots of $\mathfrak g$, and let $\Lambda^+$ be the set of dominant weights.  The finite-dimensional irreducible complex representations $V_\lambda$ of $\mathfrak{g}$ are indexed by dominant weights $\lambda \in \Lambda^+$; $\lambda$ is called \defn{minuscule} if the $W$-orbit of $\lambda$ is the set of \emph{all} weights in $V_\lambda$.  The minuscule weights are exactly those fundamental weights whose corresponding simple roots appear exactly once in the simple root expansion of the highest root.  For more information, we refer the reader to~\cite{stanley1980weyl,proctor1984bruhat}.

For $\lambda$ minuscule, define a poset on the weights in $V_\lambda$ by introducing a cover relation $\mu \lessdot \nu$ whenever $\mu + \alpha = \nu$ for some simple root $\alpha \in \Phi^+$.  This poset on $V_\lambda$ is a distributive lattice, which we call a \defn{minuscule lattice}~\cite{proctor1984bruhat}.

There are three infinite families of minuscule lattices---the order ideals in:
\begin{itemize}
\item a rectangle (type $A$; in type $B$, minuscule lattices are chains and thus particular cases of rectangles),
\item a shifted staircase (types $C$ and $D$), and
\item a ``double tailed diamond'' (type $D$)
\end{itemize}
---as well as two exceptional minuscule lattices (of types $E_6$ and $E_7$).
\medskip

In this paper, we show that the Wiener indices of the infinite families of minuscule lattices  admit simple product formulas, although we regrettably have been unable to find a unifying expression for these formulas.  We also provide information about the asymptotic distribution
of the distance between random elements in these posets (in the rectangle and shifted staircase cases),
and we give a method for computing the higher moments exactly.


\subsection{Rectangles}
Write $J(P)$ for the lattice of order ideals in a finite poset $P$, ordered by inclusion. By Birhoff's representation theorem, any distributive lattice is of this form. In~\Cref{sec:rect_bijections,sec:proof_rect_thm,sec:proof_rect_cors}, we consider the Wiener index of $\P_{m,k}=J([m]\times[k])$, the lattice of order ideals in an $m \times k$ rectangle.  
In this case (and, more generally, for 
any distributive lattice), it is easy to see that $\d(p,q)=|p \vartriangle q|$, where $p \vartriangle q = (p \backslash q) \cup (q \backslash p)$ is the symmetric difference of the order ideals $p$ and $q$.

It will be convenient to draw the elements of $\P_{m,k}$ as lattice paths from $(0,0)$ to $(m+k,m-k)$ using steps of the form $U=(1,1)$ and $D=(1,-1)$.  
Writing $p_i,q_i$ for the heights of $p$ and $q$ after the ends of their $i$th steps, the number of squares in column $i$ between the lattice paths $p$ and $q$ is given as $\left|\frac{q_i-p_i}{2}\right|$, so that 
\begin{equation}\label{eq:distance}
      \d(p,q) = \left|\frac{q_1-p_1}{2}\right|+\left|\frac{q_2-p_2}{2}\right|+\cdots+\left|\frac{q_{m+k}-p_{m+k}}{2}\right|.
\end{equation}

\begin{example}
    The graph $G(\P_{2,2})$ is drawn in \Cref{fig:2x2}. Its Wiener index is
      \begin{align*}
            56 = \frac{4}{18}\binom{10}{5}=&(0{+}1{+}2{+}2{+}3{+}4){+}(1{+}0{+}1{+}1{+}2{+}3){+}(2{+}1{+}0{+}2{+}1{+}2){+}\\&{+}(2{+}1{+}2{+}0{+}1{+}2){+}(3{+}2{+}1{+}1{+}0{+}1){+}(4{+}3{+}2{+}2{+}1{+}0).
      \end{align*}
\end{example}

\begin{figure}[htbp]

\begin{tikzpicture}
      \node (5) at (1,4.5) {
            \begin{tikzpicture}[scale=.25]
                  \node[inner sep=0pt] (a) at (0,0) {};
                  \node[inner sep=0pt] (b) at (1,1) {};
                  \node[inner sep=0pt] (c) at (2,2) {};
                  \node[inner sep=0pt] (d) at (1,-1) {};
                  \node[inner sep=0pt] (e) at (2,0) {};
                  \node[inner sep=0pt] (f) at (3,1) {};
                  \node[inner sep=0pt] (g) at (2,-2) {};
                  \node[inner sep=0pt] (h) at (3,-1) {};
                  \node[inner sep=0pt] (i) at (4,0) {};
                  \filldraw[lightgray,fill=lightgray] (a.center) to (c.center) to (i.center) to (g.center) to (a.center);
                  \draw[gray,line cap=round] (a.center) to (b.center) to (c.center);
                  \draw[gray,line cap=round] (d.center) to (e.center) to (f.center);
                  \draw[gray,line cap=round] (g.center) to (h.center) to (i.center);
                  \draw[gray,line cap=round] (a.center) to (d.center) to (g.center);
                  \draw[gray,line cap=round] (b.center) to (e.center) to (h.center);
                  \draw[gray,line cap=round] (c.center) to (f.center) to (i.center);
                  \draw[black,line cap=round,very thick] (a.center) to (c.center) to (i.center);
            \end{tikzpicture}};
      \node (4) at (0,3.5) {
            \begin{tikzpicture}[scale=.25]
                  \node[inner sep=0pt] (a) at (0,0) {};
                  \node[inner sep=0pt] (b) at (1,1) {};
                  \node[inner sep=0pt] (c) at (2,2) {};
                  \node[inner sep=0pt] (d) at (1,-1) {};
                  \node[inner sep=0pt] (e) at (2,0) {};
                  \node[inner sep=0pt] (f) at (3,1) {};
                  \node[inner sep=0pt] (g) at (2,-2) {};
                  \node[inner sep=0pt] (h) at (3,-1) {};
                  \node[inner sep=0pt] (i) at (4,0) {};
                  \filldraw[lightgray,fill=lightgray] (a.center) to (b.center) to (e.center) to (f.center) to (i.center) to (g.center) to (a.center);
                  \draw[gray,line cap=round] (a.center) to (b.center) to (c.center);
                  \draw[gray,line cap=round] (d.center) to (e.center) to (f.center);
                  \draw[gray,line cap=round] (g.center) to (h.center) to (i.center);
                  \draw[gray,line cap=round] (a.center) to (d.center) to (g.center);
                  \draw[gray,line cap=round] (b.center) to (e.center) to (h.center);
                  \draw[gray,line cap=round] (c.center) to (f.center) to (i.center);
                  \draw[black,line cap=round,very thick] (a.center) to (b.center) to (e.center) to (f.center) to (i.center);
            \end{tikzpicture}};
      \node (3) at (1,2.5) {
            \begin{tikzpicture}[scale=.25]
                  \node[inner sep=0pt] (a) at (0,0) {};
                  \node[inner sep=0pt] (b) at (1,1) {};
                  \node[inner sep=0pt] (c) at (2,2) {};
                  \node[inner sep=0pt] (d) at (1,-1) {};
                  \node[inner sep=0pt] (e) at (2,0) {};
                  \node[inner sep=0pt] (f) at (3,1) {};
                  \node[inner sep=0pt] (g) at (2,-2) {};
                  \node[inner sep=0pt] (h) at (3,-1) {};
                  \node[inner sep=0pt] (i) at (4,0) {};
                  \filldraw[lightgray,fill=lightgray] (a.center) to (d.center) to (f.center) to (i.center) to (g.center) to (a.center);
                  \draw[gray,line cap=round] (a.center) to (b.center) to (c.center);
                  \draw[gray,line cap=round] (d.center) to (e.center) to (f.center);
                  \draw[gray,line cap=round] (g.center) to (h.center) to (i.center);
                  \draw[gray,line cap=round] (a.center) to (d.center) to (g.center);
                  \draw[gray,line cap=round] (b.center) to (e.center) to (h.center);
                  \draw[gray,line cap=round] (c.center) to (f.center) to (i.center);
                  \draw[black,line cap=round,very thick] (a.center) to (d.center) to (f.center) to (i.center);
            \end{tikzpicture}};
      \node (2) at (-1,2.5) {
            \begin{tikzpicture}[scale=.25]
                  \node[inner sep=0pt] (a) at (0,0) {};
                  \node[inner sep=0pt] (b) at (1,1) {};
                  \node[inner sep=0pt] (c) at (2,2) {};
                  \node[inner sep=0pt] (d) at (1,-1) {};
                  \node[inner sep=0pt] (e) at (2,0) {};
                  \node[inner sep=0pt] (f) at (3,1) {};
                  \node[inner sep=0pt] (g) at (2,-2) {};
                  \node[inner sep=0pt] (h) at (3,-1) {};
                  \node[inner sep=0pt] (i) at (4,0) {};
                  \filldraw[lightgray,fill=lightgray] (a.center) to (b.center) to (h.center) to (i.center) to (g.center) to (a.center);
                  \draw[gray,line cap=round] (a.center) to (b.center) to (c.center);
                  \draw[gray,line cap=round] (d.center) to (e.center) to (f.center);
                  \draw[gray,line cap=round] (g.center) to (h.center) to (i.center);
                  \draw[gray,line cap=round] (a.center) to (d.center) to (g.center);
                  \draw[gray,line cap=round] (b.center) to (e.center) to (h.center);
                  \draw[gray,line cap=round] (c.center) to (f.center) to (i.center);
                  \draw[black,line cap=round,very thick] (a.center) to (b.center) to (h.center) to (i.center);
            \end{tikzpicture}};
      \node (1) at (0,1.5) {
            \begin{tikzpicture}[scale=.25]
                  \node[inner sep=0pt] (a) at (0,0) {};
                  \node[inner sep=0pt] (b) at (1,1) {};
                  \node[inner sep=0pt] (c) at (2,2) {};
                  \node[inner sep=0pt] (d) at (1,-1) {};
                  \node[inner sep=0pt] (e) at (2,0) {};
                  \node[inner sep=0pt] (f) at (3,1) {};
                  \node[inner sep=0pt] (g) at (2,-2) {};
                  \node[inner sep=0pt] (h) at (3,-1) {};
                  \node[inner sep=0pt] (i) at (4,0) {};
                  \filldraw[lightgray,fill=lightgray] (a.center) to (d.center) to (e.center) to (h.center) to (i.center) to (g.center) to (a.center);
                  \draw[gray,line cap=round] (a.center) to (b.center) to (c.center);
                  \draw[gray,line cap=round] (d.center) to (e.center) to (f.center);
                  \draw[gray,line cap=round] (g.center) to (h.center) to (i.center);
                  \draw[gray,line cap=round] (a.center) to (d.center) to (g.center);
                  \draw[gray,line cap=round] (b.center) to (e.center) to (h.center);
                  \draw[gray,line cap=round] (c.center) to (f.center) to (i.center);
                  \draw[black,line cap=round,very thick] (a.center) to (d.center) to (e.center) to (h.center) to (i.center);
            \end{tikzpicture}};
      \node (0) at (1,.5) {
            \begin{tikzpicture}[scale=.25]
                  \node[inner sep=0pt] (a) at (0,0) {};
                  \node[inner sep=0pt] (b) at (1,1) {};
                  \node[inner sep=0pt] (c) at (2,2) {};
                  \node[inner sep=0pt] (d) at (1,-1) {};
                  \node[inner sep=0pt] (e) at (2,0) {};
                  \node[inner sep=0pt] (f) at (3,1) {};
                  \node[inner sep=0pt] (g) at (2,-2) {};
                  \node[inner sep=0pt] (h) at (3,-1) {};
                  \node[inner sep=0pt] (i) at (4,0) {};
                  \draw[gray,line cap=round] (a.center) to (b.center) to (c.center);
                  \draw[gray,line cap=round] (d.center) to (e.center) to (f.center);
                  \draw[gray,line cap=round] (g.center) to (h.center) to (i.center);
                  \draw[gray,line cap=round] (a.center) to (d.center) to (g.center);
                  \draw[gray,line cap=round] (b.center) to (e.center) to (h.center);
                  \draw[gray,line cap=round] (c.center) to (f.center) to (i.center);
                  \draw[black,line cap=round,very thick] (a.center) to (g.center) to (i.center);
            \end{tikzpicture}};
      \draw[thick] (0) to (1) to (2) to (4) to (5);
      \draw[thick] (1) to (3) to (4);
\end{tikzpicture}

\caption{The Hasse diagram $G(\P_{2,2})$ of order ideals in a $2 \times 2$ square.}
\label{fig:2x2}
\end{figure}
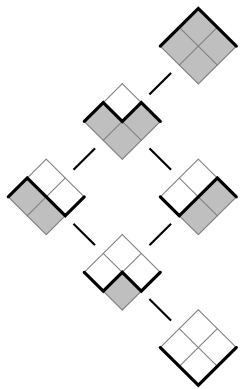

We have three results that completely describe the Wiener index of the lattice of order ideals of a rectangle.
\begin{theorem}\label{thm:gf}
      The generating function for the Wiener index of all posets $\P_{m,k}=J([m]\times[k])$ is given by
      \begin{equation}\label{eq:total_geo}
            \sum_{m=0}^\infty \sum_{k=0}^\infty \d(P_{m,k})x^m y^k = \frac{2xy}{(x^2 - 2xy + y^2 - 2x - 2y + 1)^2}.
      \end{equation} 
\end{theorem}
This theorem is obtained via classical first return decomposition for lattice paths.
The fact that this generating series is rational comes as a surprise,
since several intermediate computation steps involve algebraic non-rational functions.
Extracting the coefficient of $x^m y^k$ from~\Cref{eq:total_geo}, we obtain a formula for $\d(\P_{m,k})$.
\begin{corollary}\label{cor:coeff}
      The Wiener index of $\P_{m,k}$ is
      \[\d(\P_{m,k})=\frac{mk}{4m+4k+2}\binom{2m+2k+2}{2k+1}.\]
\end{corollary}
 
For fixed $\alpha$, we obtain the asymptotic expected value of $\d(p,q)$ in a $(\alpha n) \times n$ rectangle as $n \to \infty$.
To keep notation simple, we assume throughout the paper that $\alpha n$ is an integer
(otherwise, it suffices to replace $\alpha n$
by its integer value).
\begin{corollary}\label{cor:asymptotic}
     We have \[\frac{\d(\P_{\alpha n,n})}{|\P_{\alpha n,n}|^2} \sim \frac{\sqrt{\pi \alpha (1+\alpha)}}{4} n^{3/2} \text{ as } n \to \infty.\]
\end{corollary}

In~\Cref{sec:asymptotics}, we also describe in this regime the asymptotic distribution of the distance $D_{\alpha,n}$ between
two independent uniform random elements of $P_{\alpha n,n}$.
\begin{proposition}
\label{prop:cv_law_distance}
The random variable $n^{-3/2} D_{\alpha, n}$ converges in distribution
and in moments to $\sqrt{2 \alpha (1+\alpha)} \cdot \int_0^1 |B_0(t)| dt$,
where $B_0(t)$ is a Brownian bridge on $[0,1]$.
\end{proposition}
Informally, a Brownian bridge on $[0,1]$ is a Brownian motion conditioned to have value $0$ at time $1$.
Alternatively,
if $B$ is a Brownian motion, then $B_0(t):=B(t)-tB(1)$ is a Brownian bridge.
Brownian bridges have been extensively studied in the probabilistic litterature.
In particular, much is known on the random variable $\int_0^1 |B_0(t)| dt$;
see \cite[Section 20]{janson2007area} for a survey of results including numerous references. 
In particular, a table of the first few moments can be found in  \cite[Table 2]{janson2007area}.
We copy here the first three:
\[ \mathbb E\left[\int_0^1 |B_0(t)| dt\right] = \frac14 \sqrt{\frac{\pi}2},
\quad \mathbb E\left[\left(\int_0^1 |B_0(t)| dt\right)^2\right] = \frac7{60},
\quad \mathbb E\left[\left(\int_0^1 |B_0(t)| dt\right)^3\right] = \frac{21}{512} \sqrt{\frac{\pi}2}.\]
Together with \cref{prop:cv_law_distance}, this implies
\begin{align*}
\frac1{|\P_{\alpha n,n}|^2} \sum_{p,q \in P_{\alpha n,n}} d(p,q) =\mathbb E[D_{\alpha,n}] &\sim \frac{\sqrt{\pi \alpha (1+\alpha)}}{4} n^{3/2},\\
\frac1{|\P_{\alpha n,n}|^2} \sum_{p,q \in P_{\alpha n,n}} d(p,q)^2 =\mathbb E[D_{\alpha,n}^2]&\sim \frac{7}{30} \alpha (1+\alpha) n^{3},\\
\frac1{|\P_{\alpha n,n}|^2} \sum_{p,q \in P_{\alpha n,n}} d(p,q)^3 =\mathbb E[D_{\alpha,n}^3]&\sim \frac{21}{256} \sqrt{\pi} \alpha^{3/2} (1+\alpha)^{3/2} n^{9/2}.
\end{align*}
Note that the first estimate is nothing but \cref{cor:asymptotic}. This gives a second derivation
of this asymptotic result, which does not go through the exact expression.
Exact expressions for such higher moments can also be obtained through combinatorial means,
see \cref{sec:higher_moments} for a derivation of the second moment.

\subsection{Shifted staircases}
In~\Cref{sec:bij_stair,sec:proof_stair}, we consider the Wiener index of  $\SS_n$, the distributive lattice of order ideals in the $n$th \defn{shifted staircase} poset. Explicitly, $\SS_n$ is the set of order ideals in the poset $\{(i,j):1\leq i\leq j\leq n\}$ under componentwise ordering. The following results determine $\d(\SS_n)$ exactly.

The elements of $\SS_n$ can be represented
as lattice paths starting at $(0,0)$,
ending somewhere on the line $x=n$, and using steps 
of the form $U=(1,1)$ and $D=(1,-1)$. 
In particular, $|\SS_n|=2^n$.
Similarly as for rectangles,
if $p$ and $q$ are elements in $\SS_n$,
writing $p_i,q_i$ for the heights of $p$ and $q$ after the ends of their $i$th steps, we have
\begin{equation}
\label{eq:distance_shifted}
d(p,q) = \frac12 \sum_{i=1}^n \left|\frac{q_i-p_i}{2}\right|.
\end{equation}

\begin{example}
    The graph  $G(\SS_{3})$ is plotted in \Cref{fig:ss3}. Its Wiener index is
      \begin{align*}
            140 = \frac{6\cdot 7}{3}\binom{5}{3}=&24+18+14+14+14+14+18+24.
      \end{align*}
\end{example}

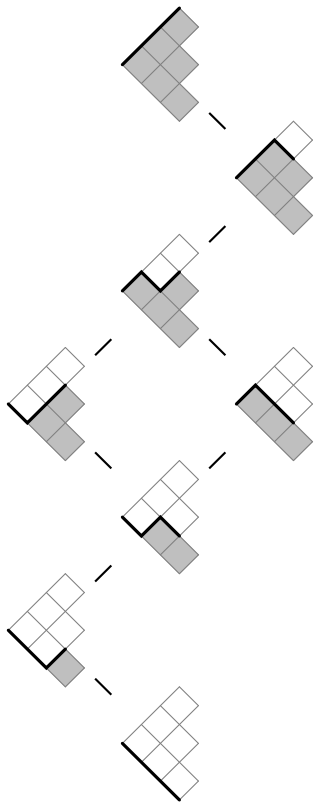
\begin{figure}[htbp]
\begin{tikzpicture}[scale=1.5]
      \node (0) at (0,0) {
            \begin{tikzpicture}[scale=.25]
                  \node[inner sep=0pt] (a) at (0,0) {};
                  \node[inner sep=0pt] (b) at (1,1) {};
                  \node[inner sep=0pt] (c) at (2,2) {};
                  \node[inner sep=0pt] (d) at (3,3) {};
                  \node[inner sep=0pt] (e) at (4,4) {};
                  \node[inner sep=0pt] (f) at (1,-1) {};
                  \node[inner sep=0pt] (g) at (2,0) {};
                  \node[inner sep=0pt] (h) at (3,1) {};
                  \node[inner sep=0pt] (i) at (4,2) {};
                  \node[inner sep=0pt] (j) at (2,-2) {};
                  \node[inner sep=0pt] (k) at (3,-1) {};
                  \node[inner sep=0pt] (l) at (4,0) {};
                  \node[inner sep=0pt] (m) at (3,-3) {};
                  \node[inner sep=0pt] (n) at (4,-2) {};
                  \node[inner sep=0pt] (o) at (4,-4) {};
                  \draw[gray,line cap=round] (a.center) to (b.center) to (c.center) to (d.center);
                  \draw[gray,line cap=round] (f.center) to (g.center) to (h.center) to (i.center);
                  \draw[gray,line cap=round] (j.center) to (k.center) to (l.center);
                  \draw[gray,line cap=round] (m.center) to (n.center);
                  \draw[gray,line cap=round] (a.center) to (f.center) to (j.center) to (m.center);
                  \draw[gray,line cap=round] (b.center) to (g.center) to (k.center) to (n.center);
                  \draw[gray,line cap=round] (c.center) to (h.center) to (l.center);
                  \draw[gray,line cap=round] (d.center) to (i.center);
                  \draw[black,line cap=round,very thick] (a.center) to (f.center) to (j.center) to (m.center);
            \end{tikzpicture}};
    \node (1) at (-1,1) {
            \begin{tikzpicture}[scale=.25]
                  \node[inner sep=0pt] (a) at (0,0) {};
                  \node[inner sep=0pt] (b) at (1,1) {};
                  \node[inner sep=0pt] (c) at (2,2) {};
                  \node[inner sep=0pt] (d) at (3,3) {};
                  \node[inner sep=0pt] (e) at (4,4) {};
                  \node[inner sep=0pt] (f) at (1,-1) {};
                  \node[inner sep=0pt] (g) at (2,0) {};
                  \node[inner sep=0pt] (h) at (3,1) {};
                  \node[inner sep=0pt] (i) at (4,2) {};
                  \node[inner sep=0pt] (j) at (2,-2) {};
                  \node[inner sep=0pt] (k) at (3,-1) {};
                  \node[inner sep=0pt] (l) at (4,0) {};
                  \node[inner sep=0pt] (m) at (3,-3) {};
                  \node[inner sep=0pt] (n) at (4,-2) {};
                  \node[inner sep=0pt] (o) at (4,-4) {};
                  \filldraw[lightgray,fill=lightgray] (j.center) to (k.center) to (n.center) to (m.center) to (j.center);
                  \draw[gray,line cap=round] (a.center) to (b.center) to (c.center) to (d.center);
                  \draw[gray,line cap=round] (f.center) to (g.center) to (h.center) to (i.center);
                  \draw[gray,line cap=round] (j.center) to (k.center) to (l.center);
                  \draw[gray,line cap=round] (m.center) to (n.center);
                  \draw[gray,line cap=round] (a.center) to (f.center) to (j.center) to (m.center);
                  \draw[gray,line cap=round] (b.center) to (g.center) to (k.center) to (n.center);
                  \draw[gray,line cap=round] (c.center) to (h.center) to (l.center);
                  \draw[gray,line cap=round] (d.center) to (i.center);
                  \draw[black,line cap=round,very thick] (a.center) to (f.center) to (j.center) to (k.center);
            \end{tikzpicture}};
    \node (2) at (0,2) {
            \begin{tikzpicture}[scale=.25]
                  \node[inner sep=0pt] (a) at (0,0) {};
                  \node[inner sep=0pt] (b) at (1,1) {};
                  \node[inner sep=0pt] (c) at (2,2) {};
                  \node[inner sep=0pt] (d) at (3,3) {};
                  \node[inner sep=0pt] (e) at (4,4) {};
                  \node[inner sep=0pt] (f) at (1,-1) {};
                  \node[inner sep=0pt] (g) at (2,0) {};
                  \node[inner sep=0pt] (h) at (3,1) {};
                  \node[inner sep=0pt] (i) at (4,2) {};
                  \node[inner sep=0pt] (j) at (2,-2) {};
                  \node[inner sep=0pt] (k) at (3,-1) {};
                  \node[inner sep=0pt] (l) at (4,0) {};
                  \node[inner sep=0pt] (m) at (3,-3) {};
                  \node[inner sep=0pt] (n) at (4,-2) {};
                  \node[inner sep=0pt] (o) at (4,-4) {};
                  \filldraw[lightgray,fill=lightgray] (f.center) to (g.center) to (n.center) to (m.center) to (f.center);
                  \draw[gray,line cap=round] (a.center) to (b.center) to (c.center) to (d.center);
                  \draw[gray,line cap=round] (f.center) to (g.center) to (h.center) to (i.center);
                  \draw[gray,line cap=round] (j.center) to (k.center) to (l.center);
                  \draw[gray,line cap=round] (m.center) to (n.center);
                  \draw[gray,line cap=round] (a.center) to (f.center) to (j.center) to (m.center);
                  \draw[gray,line cap=round] (b.center) to (g.center) to (k.center) to (n.center);
                  \draw[gray,line cap=round] (c.center) to (h.center) to (l.center);
                  \draw[gray,line cap=round] (d.center) to (i.center);
                  \draw[black,line cap=round,very thick] (a.center) to (f.center) to (g.center) to (k.center);
            \end{tikzpicture}};
    \node (3) at (-1,3) {
            \begin{tikzpicture}[scale=.25]
                  \node[inner sep=0pt] (a) at (0,0) {};
                  \node[inner sep=0pt] (b) at (1,1) {};
                  \node[inner sep=0pt] (c) at (2,2) {};
                  \node[inner sep=0pt] (d) at (3,3) {};
                  \node[inner sep=0pt] (e) at (4,4) {};
                  \node[inner sep=0pt] (f) at (1,-1) {};
                  \node[inner sep=0pt] (g) at (2,0) {};
                  \node[inner sep=0pt] (h) at (3,1) {};
                  \node[inner sep=0pt] (i) at (4,2) {};
                  \node[inner sep=0pt] (j) at (2,-2) {};
                  \node[inner sep=0pt] (k) at (3,-1) {};
                  \node[inner sep=0pt] (l) at (4,0) {};
                  \node[inner sep=0pt] (m) at (3,-3) {};
                  \node[inner sep=0pt] (n) at (4,-2) {};
                  \node[inner sep=0pt] (o) at (4,-4) {};
                  \filldraw[lightgray,fill=lightgray] (f.center) to (h.center) to (l.center) to (k.center) to (n.center) to (m.center) to (f.center);
                  \draw[gray,line cap=round] (a.center) to (b.center) to (c.center) to (d.center);
                  \draw[gray,line cap=round] (f.center) to (g.center) to (h.center) to (i.center);
                  \draw[gray,line cap=round] (j.center) to (k.center) to (l.center);
                  \draw[gray,line cap=round] (m.center) to (n.center);
                  \draw[gray,line cap=round] (a.center) to (f.center) to (j.center) to (m.center);
                  \draw[gray,line cap=round] (b.center) to (g.center) to (k.center) to (n.center);
                  \draw[gray,line cap=round] (c.center) to (h.center) to (l.center);
                  \draw[gray,line cap=round] (d.center) to (i.center);
                  \draw[black,line cap=round,very thick] (a.center) to (f.center) to (g.center) to (h.center);
            \end{tikzpicture}};
        \node (4) at (1,3) {
            \begin{tikzpicture}[scale=.25]
                  \node[inner sep=0pt] (a) at (0,0) {};
                  \node[inner sep=0pt] (b) at (1,1) {};
                  \node[inner sep=0pt] (c) at (2,2) {};
                  \node[inner sep=0pt] (d) at (3,3) {};
                  \node[inner sep=0pt] (e) at (4,4) {};
                  \node[inner sep=0pt] (f) at (1,-1) {};
                  \node[inner sep=0pt] (g) at (2,0) {};
                  \node[inner sep=0pt] (h) at (3,1) {};
                  \node[inner sep=0pt] (i) at (4,2) {};
                  \node[inner sep=0pt] (j) at (2,-2) {};
                  \node[inner sep=0pt] (k) at (3,-1) {};
                  \node[inner sep=0pt] (l) at (4,0) {};
                  \node[inner sep=0pt] (m) at (3,-3) {};
                  \node[inner sep=0pt] (n) at (4,-2) {};
                  \node[inner sep=0pt] (o) at (4,-4) {};
                  \filldraw[lightgray,fill=lightgray] (a.center) to (b.center) to (n.center) to (m.center) to (a.center);
                  \draw[gray,line cap=round] (a.center) to (b.center) to (c.center) to (d.center);
                  \draw[gray,line cap=round] (f.center) to (g.center) to (h.center) to (i.center);
                  \draw[gray,line cap=round] (j.center) to (k.center) to (l.center);
                  \draw[gray,line cap=round] (m.center) to (n.center);
                  \draw[gray,line cap=round] (a.center) to (f.center) to (j.center) to (m.center);
                  \draw[gray,line cap=round] (b.center) to (g.center) to (k.center) to (n.center);
                  \draw[gray,line cap=round] (c.center) to (h.center) to (l.center);
                  \draw[gray,line cap=round] (d.center) to (i.center);
                  \draw[black,line cap=round,very thick] (a.center) to (b.center) to (g.center) to (k.center);
            \end{tikzpicture}};
        \node (5) at (0,4) {
            \begin{tikzpicture}[scale=.25]
                  \node[inner sep=0pt] (a) at (0,0) {};
                  \node[inner sep=0pt] (b) at (1,1) {};
                  \node[inner sep=0pt] (c) at (2,2) {};
                  \node[inner sep=0pt] (d) at (3,3) {};
                  \node[inner sep=0pt] (e) at (4,4) {};
                  \node[inner sep=0pt] (f) at (1,-1) {};
                  \node[inner sep=0pt] (g) at (2,0) {};
                  \node[inner sep=0pt] (h) at (3,1) {};
                  \node[inner sep=0pt] (i) at (4,2) {};
                  \node[inner sep=0pt] (j) at (2,-2) {};
                  \node[inner sep=0pt] (k) at (3,-1) {};
                  \node[inner sep=0pt] (l) at (4,0) {};
                  \node[inner sep=0pt] (m) at (3,-3) {};
                  \node[inner sep=0pt] (n) at (4,-2) {};
                  \node[inner sep=0pt] (o) at (4,-4) {};
                  \filldraw[lightgray,fill=lightgray] (a.center) to (b.center) to (g.center) to (h.center) to (l.center) to (k.center) to (n.center) to (m.center) to (a.center);
                  \draw[gray,line cap=round] (a.center) to (b.center) to (c.center) to (d.center);
                  \draw[gray,line cap=round] (f.center) to (g.center) to (h.center) to (i.center);
                  \draw[gray,line cap=round] (j.center) to (k.center) to (l.center);
                  \draw[gray,line cap=round] (m.center) to (n.center);
                  \draw[gray,line cap=round] (a.center) to (f.center) to (j.center) to (m.center);
                  \draw[gray,line cap=round] (b.center) to (g.center) to (k.center) to (n.center);
                  \draw[gray,line cap=round] (c.center) to (h.center) to (l.center);
                  \draw[gray,line cap=round] (d.center) to (i.center);
                  \draw[black,line cap=round,very thick] (a.center) to (b.center) to (g.center) to (h.center);
            \end{tikzpicture}};
        \node (6) at (1,5) {
            \begin{tikzpicture}[scale=.25]
                  \node[inner sep=0pt] (a) at (0,0) {};
                  \node[inner sep=0pt] (b) at (1,1) {};
                  \node[inner sep=0pt] (c) at (2,2) {};
                  \node[inner sep=0pt] (d) at (3,3) {};
                  \node[inner sep=0pt] (e) at (4,4) {};
                  \node[inner sep=0pt] (f) at (1,-1) {};
                  \node[inner sep=0pt] (g) at (2,0) {};
                  \node[inner sep=0pt] (h) at (3,1) {};
                  \node[inner sep=0pt] (i) at (4,2) {};
                  \node[inner sep=0pt] (j) at (2,-2) {};
                  \node[inner sep=0pt] (k) at (3,-1) {};
                  \node[inner sep=0pt] (l) at (4,0) {};
                  \node[inner sep=0pt] (m) at (3,-3) {};ate is nothing but Corollary 1.4. This gives a second derivation of this
asymptotic result, which 
                  \node[inner sep=0pt] (n) at (4,-2) {};
                  \node[inner sep=0pt] (o) at (4,-4) {};
                  \filldraw[lightgray,fill=lightgray] (a.center) to (c.center) to (l.center) to (k.center) to (n.center) to (m.center) to (a.center);
                  \draw[gray,line cap=round] (a.center) to (b.center) to (c.center) to (d.center);
                  \draw[gray,line cap=round] (f.center) to (g.center) to (h.center) to (i.center);
                  \draw[gray,line cap=round] (j.center) to (k.center) to (l.center);
                  \draw[gray,line cap=round] (m.center) to (n.center);
                  \draw[gray,line cap=round] (a.center) to (f.center) to (j.center) to (m.center);
                  \draw[gray,line cap=round] (b.center) to (g.center) to (k.center) to (n.center);
                  \draw[gray,line cap=round] (c.center) to (h.center) to (l.center);
                  \draw[gray,line cap=round] (d.center) to (i.center);
                  \draw[black,line cap=round,very thick] (a.center) to (b.center) to (c.center) to (h.center);
            \end{tikzpicture}};
        \node (7) at (0,6) {
            \begin{tikzpicture}[scale=.25]
                  \node[inner sep=0pt] (a) at (0,0) {};
                  \node[inner sep=0pt] (b) at (1,1) {};
                  \node[inner sep=0pt] (c) at (2,2) {};
                  \node[inner sep=0pt] (d) at (3,3) {};
                  \node[inner sep=0pt] (e) at (4,4) {};
                  \node[inner sep=0pt] (f) at (1,-1) {};
                  \node[inner sep=0pt] (g) at (2,0) {};
                  \node[inner sep=0pt] (h) at (3,1) {};
                  \node[inner sep=0pt] (i) at (4,2) {};
                  \node[inner sep=0pt] (j) at (2,-2) {};
                  \node[inner sep=0pt] (k) at (3,-1) {};
                  \node[inner sep=0pt] (l) at (4,0) {};
                  \node[inner sep=0pt] (m) at (3,-3) {};
                  \node[inner sep=0pt] (n) at (4,-2) {};
                  \node[inner sep=0pt] (o) at (4,-4) {};
                  \filldraw[lightgray,fill=lightgray] (a.center) to (d.center) to (i.center) to (h.center) to (l.center) to (k.center) to (n.center) to (m.center) to (a.center);
                  \draw[gray,line cap=round] (a.center) to (b.center) to (c.center) to (d.center);
                  \draw[gray,line cap=round] (f.center) to (g.center) to (h.center) to (i.center);
                  \draw[gray,line cap=round] (j.center) to (k.center) to (l.center);
                  \draw[gray,line cap=round] (m.center) to (n.center);
                  \draw[gray,line cap=round] (a.center) to (f.center) to (j.center) to (m.center);
                  \draw[gray,line cap=round] (b.center) to (g.center) to (k.center) to (n.center);
                  \draw[gray,line cap=round] (c.center) to (h.center) to (l.center);
                  \draw[gray,line cap=round] (d.center) to (i.center);
                  \draw[black,line cap=round,very thick] (a.center) to (b.center) to (c.center) to (d.center);
            \end{tikzpicture}};
      \draw[thick] (0) to (1) to (2) to (4) to (5) to (6) to (7);
      \draw[thick] (2) to (3) to (5);
\end{tikzpicture}
\caption{The Hasse diagram $G(\SS_{3})$ of order ideals in the third shifted staircase.}
\label{fig:ss3}
\end{figure}

\begin{theorem}\label{thm:gf_SS}
The generating function for the Wiener index of all lattices $\SS_n$ is given by 
 \begin{equation}\label{eq:total_geo_SS}
            \sum_{n=0}^\infty \d(\SS_n)x^n = \frac{8x\left(1+\sqrt{1-4x}-x(3+\sqrt{1-4x})\right)}{(1-4x)(1-4x+\sqrt{1-4x})^3}.
      \end{equation} 
\end{theorem}

\begin{corollary}
\label{cor:Wiener_JSn}
The Wiener index of $\SS_n$ is \[\d(\SS_n)=\frac{2n(2n+1)}{3}\binom{2n-1}{n}.\]
Consequently, as $n$ tends to $+\infty$, we have
$d(\SS_n) \sim \frac{2}{3\sqrt \pi} 4^n n^{3/2}$.
\end{corollary}

In~\Cref{sec:asymptotics}, we turn to the asymptotic distribution of the distance $E_n$ between two random order ideals of 
$\SS_n$.

\begin{proposition}
\label{prop:cv_law_distance_shifted}
The random variable $n^{-3/2} E_n$ converges in distribution
and in moments to $\frac{1}{\sqrt{2}} \cdot \int_0^1 |B(t)| dt$,
where $B(t)$ is a Brownian motion on $[0,1]$.
\end{proposition}

Again, much is known about the random variable 
$\int_0^1 |B(t)| dt$, and a comprehensive literature review appears in~\cite[Section 21]{janson2007area}. In particular, the first few moments are given in~\cite[Table 3]{janson2007area}:
\[ \mathbb E\left[\int_0^1 |B(t)| dt\right] = \frac23 \sqrt{\frac2{\pi}},
\quad \mathbb E\left[\left(\int_0^1 |B(t)| dt\right)^2\right] = \frac3{8},
\quad \mathbb E\left[\left(\int_0^1 |B(t)| dt\right)^3\right] = \frac{263}{630} \sqrt{\frac2{\pi}}.\]
Together with \cref{prop:cv_law_distance_shifted}, this implies
\begin{align*}
\frac1{|\SS_n|^2} \sum_{p,q \in \SS_n} d(p,q) =\mathbb E[E_n]&\sim \frac{2}{3 \sqrt \pi} n^{3/2},\\
\frac1{|\SS_n|^2} \sum_{p,q \in \SS_n} d(p,q)^2 =\mathbb E[E_n^2] &\sim \frac{3}{16} n^{3},\\
\frac1{|\SS_n|^2} \sum_{p,q \in \SS_n} d(p,q)^3  =\mathbb E[E_n^3]&\sim \frac{263}{1260 \sqrt{\pi}}
n^{9/2}.
\end{align*}
Again, recalling that $|\SS_n|=2^n$,
this allows us 
to recover the asymptotic
behaviour of $d(\SS_n)$ given in \cref{cor:Wiener_JSn}
without going through its exact expression.

\subsection{The remaining minuscule lattices}
The Wiener indices of the remaining minuscule lattices are simple calculations.

Let $R_n$ be the $n$th ``double tailed diamond''---that is, the distributive lattice of order ideals in the minuscule poset of type $D_n$ corresponding to the first fundamental weight.  
\begin{theorem}\label{thm:other_mins}
We have 
\[\d(R_n)=\frac{2}{3} (n+3) \left(4 n^2+9 n+8\right).\]  The minuscule lattices of types $E_6$ and $E_7$ have Wiener indices $3584$ and $24048$, respectively.
\end{theorem}
The proof in the case of $R_n$ is elementary and left to the reader.
The cases of $E_6$ and $E_7$ are treated by computer.
The expressions for Wiener indices given in \Cref{cor:coeff}, \Cref{cor:Wiener_JSn}, and \Cref{thm:other_mins} suggest that there may be a uniform formula for $\d(P)$ for $P$ a minuscule lattice---but we regrettably have been unable to find such an expression.

\section{Lattice path bijections}\label{sec:rect_bijections}

We continue to use the steps $U=(1,1)$ and $D=(1,-1)$, but we will also make use of two versions (or colors) of the step $(1,0)$, denoted $O_1$ and $O_2$.  For $(p,q) \in \P_{k,n-k} \times \P_{k,n-k}$, define a lattice path $\bija(p,q)$ using the following dictionary between the $i$th pair of steps in $(p, q)$ and the $i$th step in $\bija(p,q)$:
      \begin{equation}\label{eq:bija}
            \begin{array}{|c|c|c|c|} \hline
                  (p,q) & A(p,q) & \left|\frac{q_{i+1}-p_{i+1}}{2}\right|-\left|\frac{q_i-p_i}{2}\right| & r_{i+1}-r_i\\ \hline
                  (D,U) & U & +1 & +1\\
                  (U,D) & D & -1 & -1\\
                  (U,U) & O_1 & 0 & 0\\
                  (D,D) & O_2 & 0 & 0 \\\hline
            \end{array}.
      \end{equation}
      
Two examples of this bijection are illustrated in~\Cref{fig:bija}.  Given a lattice path $r$ of length $n$ with steps from the set $\{U,D,O_1,O_2\}$, write $r_i$ for the height (i.e.\, the $y$-coordinate) of $r$ at the end of its $i$th step. The unsigned area between $r$ and the $x$-axis is $\d(r)=|r_0|+|r_1|+|r_2|+\cdots+|r_{n-1}|+\frac{1}{2}|r_n|$. Let us also write $\overline\d(r)=|r_0|+|r_1|+|r_2|+\cdots+|r_{n}|=\d(r)+\frac{1}{2}|r_n|$.  Then, comparing with~\Cref{eq:distance}, it is clear that $\d(p,q) = \d(\bija(p,q))$: certainly $\frac{q_0-p_0}{2}=0=r_0$, so suppose that $\frac{q_i-p_i}{2} = r_i$; then the difference in height at the $(i+1)$st step in $\bija(p,q)$ matches the difference in height at the $(i+1)$st steps of $p$ and $q$, as shown in the rightmost two columns of~\eqref{eq:bija}.

\begin{figure}[htbp]
\[\raisebox{-.5\height}{
\begin{tikzpicture}[scale=.25]
\draw[gray] (0,0) to (6,6) to (12,0) to (6,-6) to (0,0);
\filldraw[lightgray] (0,0) to (1,1) to (2,2) to (3,1) to (4,0) to (5,1) to (6,2) to (7,1) to (8,2) to (9,1) to (10,2) to (11,1) to (12,0) to (11,-1) to (10,-2) to (9,-1) to (8,-2) to (7,-3) to (6,-4) to (5,-3) to (4,-2) to (3,-1) to (2,-2) to (1,-1) to (0,0);
\draw[very thick] (0,0) to (1,1) to (2,2) to (3,1) to (4,0) to (5,1) to (6,2) to (7,1) to (8,2) to (9,1) to (10,2) to (11,1) to (12,0);
\draw (0,0) to (1,-1) to (2,-2) to (3,-1) to (4,-2) to (5,-3) to (6,-4) to (7,-3) to (8,-2) to (9,-1) to (10,-2) to (11,-1) to (12,0);
\end{tikzpicture}} \xmapsto{\bija} 
\raisebox{-.5\height}{\begin{tikzpicture}[scale=.5]
            \filldraw (0,0) circle (2pt);
            \filldraw (1,1) circle (2pt);
            \filldraw (2,2) circle (2pt);
            \filldraw (3,1) circle (2pt);
            \filldraw (4,1) circle (2pt);
            \filldraw (5,2) circle (2pt);
            \filldraw (1,0) circle (2pt);
            \filldraw (2,0) circle (2pt);
            \filldraw (3,0) circle (2pt);
            \filldraw (4,0) circle (2pt);
            \filldraw (5,0) circle (2pt);
            \filldraw (6,3) circle (2pt);
            \filldraw (7,2) circle (2pt);
            \filldraw (8,2) circle (2pt);
            \filldraw (9,1) circle (2pt);
            \filldraw (10,2) circle (2pt);
            \filldraw (11,1) circle (2pt);
            \filldraw (12,0) circle (2pt);
            \filldraw (11,0) circle (2pt);
            \filldraw (10,0) circle (2pt);
            \filldraw (9,0) circle (2pt);
            \filldraw (8,0) circle (2pt);
            \filldraw (7,0) circle (2pt);
            \filldraw (6,0) circle (2pt);
            \draw[very thick] (3,1) to (4,1);
            \draw (0,0) to (1,1) to (2,2) to (3,1) to (4,1) to (5,2) to (6,3) to (7,2) to (8,2) to (9,1) to (10,2) to (11,1) to (12,0);
            \draw (0,0) to (12,0);
      \end{tikzpicture}}\]

\[
\raisebox{-.5\height}{
\begin{tikzpicture}[scale=.25]
\draw[gray] (0,0) to (6,6) to (12,0) to (6,-6) to (0,0);
\filldraw[lightgray] (0,0) to (2,2) to (3,1) to (1,-1) to (0,0);
\filldraw[lightgray] (3,1) to (6,4) to (10,0) to (8,-2) to (7,-1) to (6,-2) to (3,1);
\filldraw[lightgray] (10,0) to (11,1) to (12,0) to (11,-1) to (10,0);
\draw[very thick] (0,0) to (1,1) to (2,2) to (3,1) to (4,0) to (5,-1) to (6,-2) to (7,-1) to (8,-2) to (9,-1) to (10,0) to (11,1) to (12,0);
\draw (0,0) to (1,-1) to (2,0) to (3,1) to (4,2) to (5,3) to (6,4) to (7,3) to (8,2) to (9,1) to (10,0) to (11,-1) to (12,0);
\end{tikzpicture}} \xmapsto{\bija} 
\raisebox{-.5\height}{
\begin{tikzpicture}[scale=.5]
            \filldraw (0,0) circle (2pt);
            \filldraw (1,1) circle (2pt);
            \filldraw (2,1) circle (2pt);
            \filldraw (3,0) circle (2pt);
            \filldraw (4,-1) circle (2pt);
            \filldraw (5,-2) circle (2pt);
            \filldraw (1,0) circle (2pt);
            \filldraw (2,0) circle (2pt);
            \filldraw (3,0) circle (2pt);
            \filldraw (4,0) circle (2pt);
            \filldraw (5,0) circle (2pt);
            \filldraw (6,-3) circle (2pt);
            \filldraw (7,-2) circle (2pt);
            \filldraw (8,-2) circle (2pt);
            \filldraw (9,-1) circle (2pt);
            \filldraw (10,0) circle (2pt);
            \filldraw (11,1) circle (2pt);
            \filldraw (12,0) circle (2pt);
            \filldraw (11,1) circle (2pt);
            \filldraw (10,0) circle (2pt);
            \filldraw (9,0) circle (2pt);
            \filldraw (8,0) circle (2pt);
            \filldraw (7,0) circle (2pt);
            \filldraw (6,0) circle (2pt);
            \draw[very thick] (7,-2) to (8,-2);
            \draw (0,0) to (1,1) to (2,1) to (3,0) to (4,-1) to (5,-2) to (6,-3) to (7,-2) to (8,-2) to (9,-1) to (10,0) to (11,1) to (12,0);
            \draw (0,0) to (12,0);
\end{tikzpicture}}
\]

\caption{Illustration of the bijection $\bija$ from the table in~\eqref{eq:bija}.  The horizontal steps $O_1$ and $O_2$ are indicated with lines of different thickness.}
\label{fig:bija}
\end{figure}
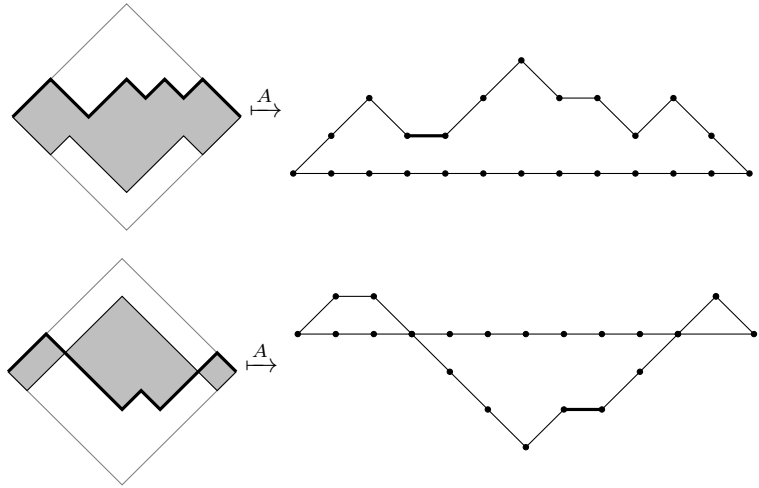


Write the set of all ordered pairs of paths in $\P_{k.n-k}$ as \[\Pt_{k,n-k}:=\P_{k,n-k} \times \P_{k,n-k},\] and denote the restriction of $\Pt_{k,n-k}$ to those pairs $(p,q)$ with $p\leq q$ as \[\Pl_{k,n-k} := \{(p,q) \in \Pt_{k,n-k} : p \leq q\}.\]  We begin by converting the total area between pairs of paths in $\Pt_{k,n-k}$ and $\Pl_{k,n-k}$ into the (unsigned) area under a single Motzkin path.

\begin{definition}
      Write $\Mb$ for the set of \defn{bilateral Motzkin paths}---that is, lattice paths from $(0,0)$ to $(n,0)$ for some $n \in \mathbb{Z}_{\geq 0}$ that use step set $\{U,D,O_1,O_2\}$.  We write $\Mb_{n,k}$ for the set of bilateral Motzkin paths that end at $(n,0)$ and use exactly $k$ steps of the form $U$ or $O_1$.  A \defn{bicolored Motzkin path} is a bilateral Motzkin path that stays weakly above the $x$-axis. Write $\M$ (resp. $\M_{n,k}$) for the set of bicolored Motzkin paths in $\Mb$ (resp. $\Mb_{n,k}$).
\end{definition}

\begin{proposition}\label{prop:bija}
The map $\bija\colon\Pt_{n,n-k}\to\Mb_{n,k}$ is a bijection satisfying $\d(p,q) = \d(\bija(p,q))$, and it restricts to a bijection from $\Pl_{n,n-k}$ to $\M_{n,k}$.
\end{proposition}

\begin{proof}
      The dictionary in~\eqref{eq:bija} shows that $\bija$ is a bijection. When $p \leq q$, $\bija(p,q)$ never goes below the $x$-axis, so $\bija$ maps $\Pl_{n,n-k}$ onto $\M_{n,k}$.  The claim about $\d$ was proven immediately after~\eqref{eq:bija}.
\end{proof}




      






\section{Proof of~\Cref{thm:gf}}\label{sec:proof_rect_thm}

In this section, we use recurrences on Motzkin paths and their associated generating functions to deduce~\Cref{thm:gf}.  

\subsection{Bicolored Motzkin paths}

Define the generating function for bicolored Motzkin paths to be
      \[\M(x,u) = \sum_{n \geq 0 } \sum_{k \geq 0} |\M_{n,k}| x^{n} u^k,\]
so that the coefficient of $x^n u^k$ counts bicolored Motzkin paths of total length $n$ with exactly $k$ steps of the form $U$ and $O_1$.

\begin{proposition} \label{prop:m_decomp}
      The generating function $\M(x,u)$ satisfies the functional equation
            \[\M(x,u)=1+x(1+u)\M(x,u)+ux^2 \M(x,u)^2,\]
      with the explicit solution
            \[\M(x,u) = \frac{1-(u + 1)x - \sqrt{(u^2 - 2u + 1)x^2 - 2(u + 1)x + 1}}{2ux^2}.\]
\end{proposition}

\begin{proof}
      The functional equation comes from decomposing a lattice path $r\in \M$ by first return to the $x$-axis: $r$ is empty; or $r$ starts with an $O_1$ step; or $r$ starts with an $O_2$ step; or $r$ starts with a $U$ step.  This is illustrated in~\Cref{fig:m_decomp}.  The explicit solution is easily obtained by solving this quadratic equation in $\M(x,u)$.
\end{proof}

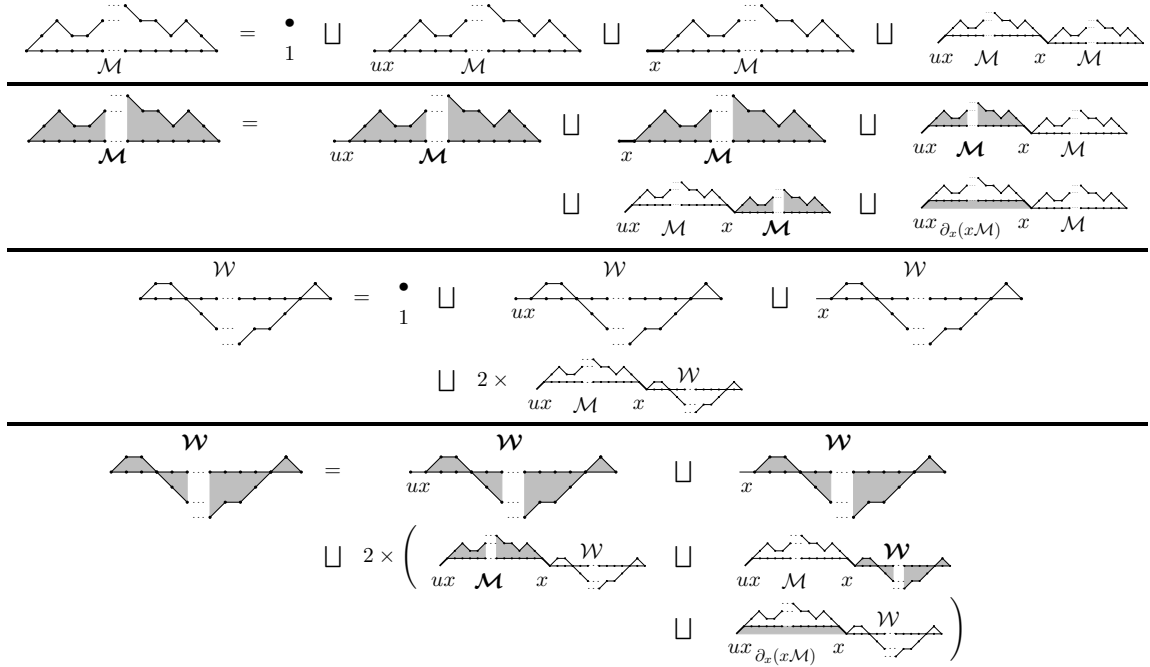
\begin{figure}[htbp]

\scalebox{.8}{
$\begin{array}{ccccccccc}
 \raisebox{-.5\height}{\begin{tikzpicture}[scale=.25]
            \node (1) at (5.5,-1) {$\M$};
            \filldraw (0,0) circle (2pt);
            \filldraw (1,1) circle (2pt);
            \filldraw (2,2) circle (2pt);
            \filldraw (3,1) circle (2pt);
            \filldraw (4,1) circle (2pt);
            \filldraw (5,2) circle (2pt);
            \filldraw (1,0) circle (2pt);
            \filldraw (2,0) circle (2pt);
            \filldraw (3,0) circle (2pt);
            \filldraw (4,0) circle (2pt);
            \filldraw (5,0) circle (2pt);
            \filldraw (6.5,3) circle (2pt);
            \filldraw (7.5,2) circle (2pt);
            \filldraw (8.5,2) circle (2pt);
            \filldraw (9.5,1) circle (2pt);
            \filldraw (10.5,2) circle (2pt);
            \filldraw (11.5,1) circle (2pt);
            \filldraw (12.5,0) circle (2pt);
            \filldraw (11.5,0) circle (2pt);
            \filldraw (10.5,0) circle (2pt);
            \filldraw (9.5,0) circle (2pt);
            \filldraw (8.5,0) circle (2pt);
            \filldraw (7.5,0) circle (2pt);
            \filldraw (6.5,0) circle (2pt);
            \draw (0,0) to (1,1) to (2,2) to (3,1) to (4,1) to (5,2);
            \draw (0,0) to (5,0);
            \draw (6.5,0) to (12.5,0);
            \draw (6.5,3) to (7.5,2) to (8.5,2) to (9.5,1) to (10.5,2) to (11.5,1) to (12.5,0);
            \node (a) at (5.75,0) {\scalebox{.5}{$\ldots$}};
            \node (b) at (5.75,2) {\scalebox{.5}{$\ldots$}};
            \node (c) at (5.75,3) {\scalebox{.5}{$\ldots$}};
      \end{tikzpicture}} & = & \raisebox{-.5\height}{\begin{tikzpicture} \node (0) at (0,0) {$\bullet$}; \node (1) at (0,-.5) {$1$}; \end{tikzpicture}} & \bigsqcup &  
\raisebox{-.5\height}{\begin{tikzpicture}[scale=.25]
            \filldraw (-1,0) circle (2pt);
            \node (0) at (-.5,-1) {$ux$};
            \node (1) at (5.5,-1) {$\M$};
            \filldraw (0,0) circle (2pt);
            \filldraw (1,1) circle (2pt);
            \filldraw (2,2) circle (2pt);
            \filldraw (3,1) circle (2pt);
            \filldraw (4,1) circle (2pt);
            \filldraw (5,2) circle (2pt);
            \filldraw (1,0) circle (2pt);
            \filldraw (2,0) circle (2pt);
            \filldraw (3,0) circle (2pt);
            \filldraw (4,0) circle (2pt);
            \filldraw (5,0) circle (2pt);
            \filldraw (6.5,3) circle (2pt);
            \filldraw (7.5,2) circle (2pt);
            \filldraw (8.5,2) circle (2pt);
            \filldraw (9.5,1) circle (2pt);
            \filldraw (10.5,2) circle (2pt);
            \filldraw (11.5,1) circle (2pt);
            \filldraw (12.5,0) circle (2pt);
            \filldraw (11.5,0) circle (2pt);
            \filldraw (10.5,0) circle (2pt);
            \filldraw (9.5,0) circle (2pt);
            \filldraw (8.5,0) circle (2pt);
            \filldraw (7.5,0) circle (2pt);
            \filldraw (6.5,0) circle (2pt);
            \draw (0,0) to (1,1) to (2,2) to (3,1) to (4,1) to (5,2);
            \draw (-1,0) to (5,0);
            \draw (6.5,0) to (12.5,0);
            \draw (6.5,3) to (7.5,2) to (8.5,2) to (9.5,1) to (10.5,2) to (11.5,1) to (12.5,0);
            \node (a) at (5.75,0) {\scalebox{.5}{$\ldots$}};
            \node (b) at (5.75,2) {\scalebox{.5}{$\ldots$}};
            \node (c) at (5.75,3) {\scalebox{.5}{$\ldots$}};
      \end{tikzpicture}} & \bigsqcup &
      \raisebox{-.5\height}{\begin{tikzpicture}[scale=.25]
            \node (0) at (-.5,-1) {$x$};
            \node (1) at (5.5,-1) {$\M$};
            \filldraw (0,0) circle (2pt);
            \filldraw (-1,0) circle (2pt);
            \filldraw (0,0) circle (2pt);
            \filldraw (1,1) circle (2pt);
            \filldraw (2,2) circle (2pt);
            \filldraw (3,1) circle (2pt);
            \filldraw (4,1) circle (2pt);
            \filldraw (5,2) circle (2pt);
            \filldraw (1,0) circle (2pt);
            \filldraw (2,0) circle (2pt);
            \filldraw (3,0) circle (2pt);
            \filldraw (4,0) circle (2pt);
            \filldraw (5,0) circle (2pt);
            \filldraw (6.5,3) circle (2pt);
            \filldraw (7.5,2) circle (2pt);
            \filldraw (8.5,2) circle (2pt);
            \filldraw (9.5,1) circle (2pt);
            \filldraw (10.5,2) circle (2pt);
            \filldraw (11.5,1) circle (2pt);
            \filldraw (12.5,0) circle (2pt);
            \filldraw (11.5,0) circle (2pt);
            \filldraw (10.5,0) circle (2pt);
            \filldraw (9.5,0) circle (2pt);
            \filldraw (8.5,0) circle (2pt);
            \filldraw (7.5,0) circle (2pt);
            \filldraw (6.5,0) circle (2pt);
            \draw[very thick] (-1,0) to (0,0);
            \draw (0,0) to (1,1) to (2,2) to (3,1) to (4,1) to (5,2);
            \draw (0,0) to (5,0);
            \draw (6.5,0) to (12.5,0);
            \draw (6.5,3) to (7.5,2) to (8.5,2) to (9.5,1) to (10.5,2) to (11.5,1) to (12.5,0);
            \node (a) at (5.75,0) {\scalebox{.5}{$\ldots$}};
            \node (b) at (5.75,2) {\scalebox{.5}{$\ldots$}};
            \node (c) at (5.75,3) {\scalebox{.5}{$\ldots$}};
      \end{tikzpicture}} & \bigsqcup &
      \raisebox{-.5\height}{
      \begin{tikzpicture}[scale=.125]
            \draw[thick] (-1,-1) to (0,0);
            \draw[thick] (12.5,0) to (13.5,-1);
            \node (1) at (5.5,-3) {$\M$};
            \node (2) at (19,-3) {$\M$};
            \node (3) at (-.5,-3) {$ux$};
            \node (4) at (12.5,-3) {$x$};
            \filldraw (0,0) circle (2pt);
            \filldraw (1,1) circle (2pt);
            \filldraw (2,2) circle (2pt);
            \filldraw (3,1) circle (2pt);
            \filldraw (4,1) circle (2pt);
            \filldraw (5,2) circle (2pt);
            \filldraw (1,0) circle (2pt);
            \filldraw (2,0) circle (2pt);
            \filldraw (3,0) circle (2pt);
            \filldraw (4,0) circle (2pt);
            \filldraw (5,0) circle (2pt);
            \filldraw (6.5,3) circle (2pt);
            \filldraw (7.5,2) circle (2pt);
            \filldraw (8.5,2) circle (2pt);
            \filldraw (9.5,1) circle (2pt);
            \filldraw (10.5,2) circle (2pt);
            \filldraw (11.5,1) circle (2pt);
            \filldraw (12.5,0) circle (2pt);
            \filldraw (11.5,0) circle (2pt);
            \filldraw (10.5,0) circle (2pt);
            \filldraw (9.5,0) circle (2pt);
            \filldraw (8.5,0) circle (2pt);
            \filldraw (7.5,0) circle (2pt);
            \filldraw (6.5,0) circle (2pt);
            \draw (0,0) to (1,1) to (2,2) to (3,1) to (4,1) to (5,2);
            \draw (0,0) to (5,0);
            \draw (6.5,0) to (12.5,0);
            \draw (6.5,3) to (7.5,2) to (8.5,2) to (9.5,1) to (10.5,2) to (11.5,1) to (12.5,0);
            \node (a) at (5.75,0) {\scalebox{.25}{$\ldots$}};
            \node (b) at (5.75,2) {\scalebox{.25}{$\ldots$}};
            \node (c) at (5.75,3) {\scalebox{.25}{$\ldots$}};
            \filldraw (13.5,-1) circle (2pt);
            \filldraw (14.5,0) circle (2pt);
            \filldraw (15.5,1) circle (2pt);
            \filldraw (16.5,0) circle (2pt);
            \filldraw (17.5,0) circle (2pt);
            \filldraw (18.5,1) circle (2pt);
            \filldraw (14.5,-1) circle (2pt);
            \filldraw (15.5,-1) circle (2pt);
            \filldraw (16.5,-1) circle (2pt);
            \filldraw (17.5,-1) circle (2pt);
            \filldraw (18.5,-1) circle (2pt);
            \filldraw (20,2) circle (2pt);
            \filldraw (21,1) circle (2pt);
            \filldraw (22,1) circle (2pt);
            \filldraw (23,0) circle (2pt);
            \filldraw (24,1) circle (2pt);
            \filldraw (25,0) circle (2pt);
            \filldraw (26,-1) circle (2pt);
            \filldraw (25,-1) circle (2pt);
            \filldraw (24,-1) circle (2pt);
            \filldraw (23,-1) circle (2pt);
            \filldraw (22,-1) circle (2pt);
            \filldraw (21,-1) circle (2pt);
            \filldraw (20,-1) circle (2pt);
            \draw (13.5,-1) to (14.5,0) to (15.5,1) to (16.5,0) to (17.5,0) to (18.5,1);
            \draw (13.5,-1) to (18.5,-1);
            \draw (20,-1) to (26,-1);
            \draw (20,2) to (21,1) to (22,1) to (23,0) to (24,1) to (25,0) to (26,-1);
            \node (d) at (19.25,-1) {\scalebox{.25}{$\ldots$}};
            \node (e) at (19.25,1) {\scalebox{.25}{$\ldots$}};
            \node (f) at (19.25,2) {\scalebox{.25}{$\ldots$}};
      \end{tikzpicture}
      }
\end{array}$}

\noindent\rule{\textwidth}{1pt}

\scalebox{.8}{
$\begin{array}{ccccccccc}
 \raisebox{-.5\height}{\begin{tikzpicture}[scale=.25]
             \filldraw[lightgray,fill=lightgray] (0,0) to (1,1) to (2,2) to (3,1) to (4,1) to (5,2) to (5,0) to (0,0);
            \filldraw[lightgray,fill=lightgray] (6.5,3) to (7.5,2) to (8.5,2) to (9.5,1) to (10.5,2) to (11.5,1) to (12.5,0) to (6.5,0) to (6.5,3);
            \node (1) at (5.5,-1) {$\Mm$};
            \filldraw (0,0) circle (2pt);
            \filldraw (1,1) circle (2pt);
            \filldraw (2,2) circle (2pt);
            \filldraw (3,1) circle (2pt);
            \filldraw (4,1) circle (2pt);
            \filldraw (5,2) circle (2pt);
            \filldraw (1,0) circle (2pt);
            \filldraw (2,0) circle (2pt);
            \filldraw (3,0) circle (2pt);
            \filldraw (4,0) circle (2pt);
            \filldraw (5,0) circle (2pt);
            \filldraw (6.5,3) circle (2pt);
            \filldraw (7.5,2) circle (2pt);
            \filldraw (8.5,2) circle (2pt);
            \filldraw (9.5,1) circle (2pt);
            \filldraw (10.5,2) circle (2pt);
            \filldraw (11.5,1) circle (2pt);
            \filldraw (12.5,0) circle (2pt);
            \filldraw (11.5,0) circle (2pt);
            \filldraw (10.5,0) circle (2pt);
            \filldraw (9.5,0) circle (2pt);
            \filldraw (8.5,0) circle (2pt);
            \filldraw (7.5,0) circle (2pt);
            \filldraw (6.5,0) circle (2pt);
            \draw (0,0) to (1,1) to (2,2) to (3,1) to (4,1) to (5,2);
            \draw (0,0) to (5,0);
            \draw (6.5,0) to (12.5,0);
            \draw (6.5,3) to (7.5,2) to (8.5,2) to (9.5,1) to (10.5,2) to (11.5,1) to (12.5,0);
            \node (a) at (5.75,0) {\scalebox{.5}{$\ldots$}};
            \node (b) at (5.75,2) {\scalebox{.5}{$\ldots$}};
            \node (c) at (5.75,3) {\scalebox{.5}{$\ldots$}};
      \end{tikzpicture}} & = &  &  &  
\raisebox{-.5\height}{\begin{tikzpicture}[scale=.25]
            \filldraw[lightgray,fill=lightgray] (0,0) to (1,1) to (2,2) to (3,1) to (4,1) to (5,2) to (5,0) to (0,0);
            \filldraw[lightgray,fill=lightgray] (6.5,3) to (7.5,2) to (8.5,2) to (9.5,1) to (10.5,2) to (11.5,1) to (12.5,0) to (6.5,0) to (6.5,3);
            \filldraw (-1,0) circle (2pt);
            \node (0) at (-.5,-1) {$ux$};
            \node (1) at (5.5,-1) {$\Mm$};
            \filldraw (0,0) circle (2pt);
            \filldraw (1,1) circle (2pt);
            \filldraw (2,2) circle (2pt);
            \filldraw (3,1) circle (2pt);
            \filldraw (4,1) circle (2pt);
            \filldraw (5,2) circle (2pt);
            \filldraw (1,0) circle (2pt);
            \filldraw (2,0) circle (2pt);
            \filldraw (3,0) circle (2pt);
            \filldraw (4,0) circle (2pt);
            \filldraw (5,0) circle (2pt);
            \filldraw (6.5,3) circle (2pt);
            \filldraw (7.5,2) circle (2pt);
            \filldraw (8.5,2) circle (2pt);
            \filldraw (9.5,1) circle (2pt);
            \filldraw (10.5,2) circle (2pt);
            \filldraw (11.5,1) circle (2pt);
            \filldraw (12.5,0) circle (2pt);
            \filldraw (11.5,0) circle (2pt);
            \filldraw (10.5,0) circle (2pt);
            \filldraw (9.5,0) circle (2pt);
            \filldraw (8.5,0) circle (2pt);
            \filldraw (7.5,0) circle (2pt);
            \filldraw (6.5,0) circle (2pt);
            \draw (0,0) to (1,1) to (2,2) to (3,1) to (4,1) to (5,2);
            \draw (-1,0) to (5,0);
            \draw (6.5,0) to (12.5,0);
            \draw (6.5,3) to (7.5,2) to (8.5,2) to (9.5,1) to (10.5,2) to (11.5,1) to (12.5,0);
            \node (a) at (5.75,0) {\scalebox{.5}{$\ldots$}};
            \node (b) at (5.75,2) {\scalebox{.5}{$\ldots$}};
            \node (c) at (5.75,3) {\scalebox{.5}{$\ldots$}};
      \end{tikzpicture}} & \bigsqcup &
      \raisebox{-.5\height}{\begin{tikzpicture}[scale=.25]
            \filldraw[lightgray,fill=lightgray] (0,0) to (1,1) to (2,2) to (3,1) to (4,1) to (5,2) to (5,0) to (0,0);
            \filldraw[lightgray,fill=lightgray] (6.5,3) to (7.5,2) to (8.5,2) to (9.5,1) to (10.5,2) to (11.5,1) to (12.5,0) to (6.5,0) to (6.5,3);
            \node (0) at (-.5,-1) {$x$};
            \node (1) at (5.5,-1) {$\Mm$};
            \filldraw (0,0) circle (2pt);
            \filldraw (-1,0) circle (2pt);
            \filldraw (0,0) circle (2pt);
            \filldraw (1,1) circle (2pt);
            \filldraw (2,2) circle (2pt);
            \filldraw (3,1) circle (2pt);
            \filldraw (4,1) circle (2pt);
            \filldraw (5,2) circle (2pt);
            \filldraw (1,0) circle (2pt);
            \filldraw (2,0) circle (2pt);
            \filldraw (3,0) circle (2pt);
            \filldraw (4,0) circle (2pt);
            \filldraw (5,0) circle (2pt);
            \filldraw (6.5,3) circle (2pt);
            \filldraw (7.5,2) circle (2pt);
            \filldraw (8.5,2) circle (2pt);
            \filldraw (9.5,1) circle (2pt);
            \filldraw (10.5,2) circle (2pt);
            \filldraw (11.5,1) circle (2pt);
            \filldraw (12.5,0) circle (2pt);
            \filldraw (11.5,0) circle (2pt);
            \filldraw (10.5,0) circle (2pt);
            \filldraw (9.5,0) circle (2pt);
            \filldraw (8.5,0) circle (2pt);
            \filldraw (7.5,0) circle (2pt);
            \filldraw (6.5,0) circle (2pt);
            \draw[very thick] (-1,0) to (0,0);
            \draw (0,0) to (1,1) to (2,2) to (3,1) to (4,1) to (5,2);
            \draw (0,0) to (5,0);
            \draw (6.5,0) to (12.5,0);
            \draw (6.5,3) to (7.5,2) to (8.5,2) to (9.5,1) to (10.5,2) to (11.5,1) to (12.5,0);
            \node (a) at (5.75,0) {\scalebox{.5}{$\ldots$}};
            \node (b) at (5.75,2) {\scalebox{.5}{$\ldots$}};
            \node (c) at (5.75,3) {\scalebox{.5}{$\ldots$}};
      \end{tikzpicture}} & \bigsqcup &
      \raisebox{-.5\height}{
      \begin{tikzpicture}[scale=.125]
            \filldraw[lightgray,fill=lightgray] (0,0) to (1,1) to (2,2) to (3,1) to (4,1) to (5,2) to (5,0) to (0,0);
            \filldraw[lightgray,fill=lightgray] (6.5,3) to (7.5,2) to (8.5,2) to (9.5,1) to (10.5,2) to (11.5,1) to (12.5,0) to (6.5,0) to (6.5,3);
            \draw[thick] (-1,-1) to (0,0);
            \draw[thick] (12.5,0) to (13.5,-1);
            \node (1) at (5.5,-3) {$\Mm$};
            \node (2) at (19,-3) {$\M$};
            \node (3) at (-.5,-3) {$ux$};
            \node (4) at (12.5,-3) {$x$};
            \filldraw (0,0) circle (2pt);
            \filldraw (1,1) circle (2pt);
            \filldraw (2,2) circle (2pt);
            \filldraw (3,1) circle (2pt);
            \filldraw (4,1) circle (2pt);
            \filldraw (5,2) circle (2pt);
            \filldraw (1,0) circle (2pt);
            \filldraw (2,0) circle (2pt);
            \filldraw (3,0) circle (2pt);
            \filldraw (4,0) circle (2pt);
            \filldraw (5,0) circle (2pt);
            \filldraw (6.5,3) circle (2pt);
            \filldraw (7.5,2) circle (2pt);
            \filldraw (8.5,2) circle (2pt);
            \filldraw (9.5,1) circle (2pt);
            \filldraw (10.5,2) circle (2pt);
            \filldraw (11.5,1) circle (2pt);
            \filldraw (12.5,0) circle (2pt);
            \filldraw (11.5,0) circle (2pt);
            \filldraw (10.5,0) circle (2pt);
            \filldraw (9.5,0) circle (2pt);
            \filldraw (8.5,0) circle (2pt);
            \filldraw (7.5,0) circle (2pt);
            \filldraw (6.5,0) circle (2pt);
            \draw (0,0) to (1,1) to (2,2) to (3,1) to (4,1) to (5,2);
            \draw (0,0) to (5,0);
            \draw (6.5,0) to (12.5,0);
            \draw (6.5,3) to (7.5,2) to (8.5,2) to (9.5,1) to (10.5,2) to (11.5,1) to (12.5,0);
            \node (a) at (5.75,0) {\scalebox{.25}{$\ldots$}};
            \node (b) at (5.75,2) {\scalebox{.25}{$\ldots$}};
            \node (c) at (5.75,3) {\scalebox{.25}{$\ldots$}};
            \filldraw (13.5,-1) circle (2pt);
            \filldraw (14.5,0) circle (2pt);
            \filldraw (15.5,1) circle (2pt);
            \filldraw (16.5,0) circle (2pt);
            \filldraw (17.5,0) circle (2pt);
            \filldraw (18.5,1) circle (2pt);
            \filldraw (14.5,-1) circle (2pt);
            \filldraw (15.5,-1) circle (2pt);
            \filldraw (16.5,-1) circle (2pt);
            \filldraw (17.5,-1) circle (2pt);
            \filldraw (18.5,-1) circle (2pt);
            \filldraw (20,2) circle (2pt);
            \filldraw (21,1) circle (2pt);
            \filldraw (22,1) circle (2pt);
            \filldraw (23,0) circle (2pt);
            \filldraw (24,1) circle (2pt);
            \filldraw (25,0) circle (2pt);
            \filldraw (26,-1) circle (2pt);
            \filldraw (25,-1) circle (2pt);
            \filldraw (24,-1) circle (2pt);
            \filldraw (23,-1) circle (2pt);
            \filldraw (22,-1) circle (2pt);
            \filldraw (21,-1) circle (2pt);
            \filldraw (20,-1) circle (2pt);
            \draw (13.5,-1) to (14.5,0) to (15.5,1) to (16.5,0) to (17.5,0) to (18.5,1);
            \draw (13.5,-1) to (18.5,-1);
            \draw (20,-1) to (26,-1);
            \draw (20,2) to (21,1) to (22,1) to (23,0) to (24,1) to (25,0) to (26,-1);
            \node (d) at (19.25,-1) {\scalebox{.25}{$\ldots$}};
            \node (e) at (19.25,1) {\scalebox{.25}{$\ldots$}};
            \node (f) at (19.25,2) {\scalebox{.25}{$\ldots$}};
      \end{tikzpicture}
      }\\
       & &  & & & \bigsqcup & \raisebox{-.5\height}{
      \begin{tikzpicture}[scale=.125]
            \draw[thick] (-1,-1) to (0,0);
            \draw[thick] (12.5,0) to (13.5,-1);
            \node (1) at (5.5,-3) {$\M$};
            \node (2) at (19,-3) {$\Mm$};
            \node (3) at (-.5,-3) {$ux$};
            \node (4) at (12.5,-3) {$x$};
            \filldraw (0,0) circle (2pt);
            \filldraw (1,1) circle (2pt);
            \filldraw (2,2) circle (2pt);
            \filldraw (3,1) circle (2pt);
            \filldraw (4,1) circle (2pt);
            \filldraw (5,2) circle (2pt);
            \filldraw (1,0) circle (2pt);
            \filldraw (2,0) circle (2pt);
            \filldraw (3,0) circle (2pt);
            \filldraw (4,0) circle (2pt);
            \filldraw (5,0) circle (2pt);
            \filldraw (6.5,3) circle (2pt);
            \filldraw (7.5,2) circle (2pt);
            \filldraw (8.5,2) circle (2pt);
            \filldraw (9.5,1) circle (2pt);
            \filldraw (10.5,2) circle (2pt);
            \filldraw (11.5,1) circle (2pt);
            \filldraw (12.5,0) circle (2pt);
            \filldraw (11.5,0) circle (2pt);
            \filldraw (10.5,0) circle (2pt);
            \filldraw (9.5,0) circle (2pt);
            \filldraw (8.5,0) circle (2pt);
            \filldraw (7.5,0) circle (2pt);
            \filldraw (6.5,0) circle (2pt);
            \draw (0,0) to (1,1) to (2,2) to (3,1) to (4,1) to (5,2);
            \draw (0,0) to (5,0);
            \draw (6.5,0) to (12.5,0);
            \draw (6.5,3) to (7.5,2) to (8.5,2) to (9.5,1) to (10.5,2) to (11.5,1) to (12.5,0);
            \node (a) at (5.75,0) {\scalebox{.25}{$\ldots$}};
            \node (b) at (5.75,2) {\scalebox{.25}{$\ldots$}};
            \node (c) at (5.75,3) {\scalebox{.25}{$\ldots$}};
            \filldraw[lightgray,fill=lightgray] (13.5,-1) to (14.5,0) to (15.5,1) to (16.5,0) to (17.5,0) to (18.5,1) to (18.5,-1) to (13.5,-1);
            \filldraw[lightgray,fill=lightgray] (20,2) to (21,1) to (22,1) to (23,0) to (24,1) to (25,0) to (26,-1) to (20,-1) to (20,2);
            \filldraw (13.5,-1) circle (2pt);
            \filldraw (14.5,0) circle (2pt);
            \filldraw (15.5,1) circle (2pt);
            \filldraw (16.5,0) circle (2pt);
            \filldraw (17.5,0) circle (2pt);
            \filldraw (18.5,1) circle (2pt);
            \filldraw (14.5,-1) circle (2pt);
            \filldraw (15.5,-1) circle (2pt);
            \filldraw (16.5,-1) circle (2pt);
            \filldraw (17.5,-1) circle (2pt);
            \filldraw (18.5,-1) circle (2pt);
            \filldraw (20,2) circle (2pt);
            \filldraw (21,1) circle (2pt);
            \filldraw (22,1) circle (2pt);
            \filldraw (23,0) circle (2pt);
            \filldraw (24,1) circle (2pt);
            \filldraw (25,0) circle (2pt);
            \filldraw (26,-1) circle (2pt);
            \filldraw (25,-1) circle (2pt);
            \filldraw (24,-1) circle (2pt);
            \filldraw (23,-1) circle (2pt);
            \filldraw (22,-1) circle (2pt);
            \filldraw (21,-1) circle (2pt);
            \filldraw (20,-1) circle (2pt);
            \draw (13.5,-1) to (14.5,0) to (15.5,1) to (16.5,0) to (17.5,0) to (18.5,1);
            \draw (13.5,-1) to (18.5,-1);
            \draw (20,-1) to (26,-1);
            \draw (20,2) to (21,1) to (22,1) to (23,0) to (24,1) to (25,0) to (26,-1);
            \node (d) at (19.25,-1) {\scalebox{.25}{$\ldots$}};
            \node (e) at (19.25,1) {\scalebox{.25}{$\ldots$}};
            \node (f) at (19.25,2) {\scalebox{.25}{$\ldots$}};
      \end{tikzpicture}
      }  & \bigsqcup &
      \raisebox{-.5\height}{
      \begin{tikzpicture}[scale=.125]
            \filldraw[lightgray,fill=lightgray] (-1,-1) to (0,0) to (12.5,0) to (13.5,-1) to (-1,-1);
            \draw[thick] (-1,-1) to (0,0);
            \draw[thick] (12.5,0) to (13.5,-1);
            \node (1) at (5.5,-4) {\scalebox{0.8}{$\partial_x(x\M)$}};
            \node (2) at (19,-3) {$\M$};
            \node (3) at (-.5,-3) {$ux$};
            \node (4) at (12.5,-3) {$x$};
            \filldraw (0,0) circle (2pt);
            \filldraw (1,1) circle (2pt);
            \filldraw (2,2) circle (2pt);
            \filldraw (3,1) circle (2pt);
            \filldraw (4,1) circle (2pt);
            \filldraw (5,2) circle (2pt);
            \filldraw (1,0) circle (2pt);
            \filldraw (2,0) circle (2pt);
            \filldraw (3,0) circle (2pt);
            \filldraw (4,0) circle (2pt);
            \filldraw (5,0) circle (2pt);
            \filldraw (6.5,3) circle (2pt);
            \filldraw (7.5,2) circle (2pt);
            \filldraw (8.5,2) circle (2pt);
            \filldraw (9.5,1) circle (2pt);
            \filldraw (10.5,2) circle (2pt);
            \filldraw (11.5,1) circle (2pt);
            \filldraw (12.5,0) circle (2pt);
            \filldraw (11.5,0) circle (2pt);
            \filldraw (10.5,0) circle (2pt);
            \filldraw (9.5,0) circle (2pt);
            \filldraw (8.5,0) circle (2pt);
            \filldraw (7.5,0) circle (2pt);
            \filldraw (6.5,0) circle (2pt);
            \draw (0,0) to (1,1) to (2,2) to (3,1) to (4,1) to (5,2);
            \draw (0,0) to (5,0);
            \draw (6.5,0) to (12.5,0);
            \draw (6.5,3) to (7.5,2) to (8.5,2) to (9.5,1) to (10.5,2) to (11.5,1) to (12.5,0);
            \node (a) at (5.75,0) {\scalebox{.25}{$\ldots$}};
            \node (b) at (5.75,2) {\scalebox{.25}{$\ldots$}};
            \node (c) at (5.75,3) {\scalebox{.25}{$\ldots$}};
            \filldraw (13.5,-1) circle (2pt);
            \filldraw (14.5,0) circle (2pt);
            \filldraw (15.5,1) circle (2pt);
            \filldraw (16.5,0) circle (2pt);
            \filldraw (17.5,0) circle (2pt);
            \filldraw (18.5,1) circle (2pt);
            \filldraw (14.5,-1) circle (2pt);
            \filldraw (15.5,-1) circle (2pt);
            \filldraw (16.5,-1) circle (2pt);
            \filldraw (17.5,-1) circle (2pt);
            \filldraw (18.5,-1) circle (2pt);
            \filldraw (20,2) circle (2pt);
            \filldraw (21,1) circle (2pt);
            \filldraw (22,1) circle (2pt);
            \filldraw (23,0) circle (2pt);
            \filldraw (24,1) circle (2pt);
            \filldraw (25,0) circle (2pt);
            \filldraw (26,-1) circle (2pt);
            \filldraw (25,-1) circle (2pt);
            \filldraw (24,-1) circle (2pt);
            \filldraw (23,-1) circle (2pt);
            \filldraw (22,-1) circle (2pt);
            \filldraw (21,-1) circle (2pt);
            \filldraw (20,-1) circle (2pt);
            \draw (13.5,-1) to (14.5,0) to (15.5,1) to (16.5,0) to (17.5,0) to (18.5,1);
            \draw (13.5,-1) to (18.5,-1);
            \draw (20,-1) to (26,-1);
            \draw (20,2) to (21,1) to (22,1) to (23,0) to (24,1) to (25,0) to (26,-1);
            \node (d) at (19.25,-1) {\scalebox{.25}{$\ldots$}};
            \node (e) at (19.25,1) {\scalebox{.25}{$\ldots$}};
            \node (f) at (19.25,2) {\scalebox{.25}{$\ldots$}};
      \end{tikzpicture}
      } 
\end{array}$}

\noindent\rule{\textwidth}{1pt}

\scalebox{.8}{
$\begin{array}{ccccccccc}
 \raisebox{-.5\height}{\begin{tikzpicture}[scale=.25]
            \node (1) at (5.5,2) {$\Mb$};
            \filldraw (0,0) circle (2pt);
            \filldraw (1,1) circle (2pt);
            \filldraw (2,1) circle (2pt);
            \filldraw (3,0) circle (2pt);
            \filldraw (4,-1) circle (2pt);
            \filldraw (5,-2) circle (2pt);
            \filldraw (1,0) circle (2pt);
            \filldraw (2,0) circle (2pt);
            \filldraw (3,0) circle (2pt);
            \filldraw (4,0) circle (2pt);
            \filldraw (5,0) circle (2pt);
            \filldraw (6.5,-3) circle (2pt);
            \filldraw (7.5,-2) circle (2pt);
            \filldraw (8.5,-2) circle (2pt);
            \filldraw (9.5,-1) circle (2pt);
            \filldraw (10.5,0) circle (2pt);
            \filldraw (11.5,1) circle (2pt);
            \filldraw (12.5,0) circle (2pt);
            \filldraw (11.5,1) circle (2pt);
            \filldraw (10.5,0) circle (2pt);
            \filldraw (9.5,0) circle (2pt);
            \filldraw (8.5,0) circle (2pt);
            \filldraw (7.5,0) circle (2pt);
            \filldraw (6.5,0) circle (2pt);
            \draw (0,0) to (1,1) to (2,1) to (3,0) to (4,-1) to (5,-2);
            \draw (0,0) to (5,0);
            \draw (6.5,0) to (12.5,0);
            \draw (6.5,-3) to (7.5,-2) to (8.5,-2) to (9.5,-1) to (10.5,0) to (11.5,1) to (12.5,0);
            \node (a) at (5.75,0) {\scalebox{.5}{$\ldots$}};
            \node (b) at (5.75,-2) {\scalebox{.5}{$\ldots$}};
            \node (c) at (5.75,-3) {\scalebox{.5}{$\ldots$}};
      \end{tikzpicture}} & = & \raisebox{-.5\height}{\begin{tikzpicture} \node (0) at (0,0) {$\bullet$}; \node (1) at (0,-.5) {$1$}; \end{tikzpicture}} & \bigsqcup &  
\raisebox{-.5\height}{\begin{tikzpicture}[scale=.25]
            \filldraw (-1,0) circle (2pt);
            \node (0) at (-.5,-1) {$ux$};
            \node (1) at (5.5,2) {$\Mb$};
            \filldraw (0,0) circle (2pt);
            \filldraw (1,1) circle (2pt);
            \filldraw (2,1) circle (2pt);
            \filldraw (3,0) circle (2pt);
            \filldraw (4,-1) circle (2pt);
            \filldraw (5,-2) circle (2pt);
            \filldraw (1,0) circle (2pt);
            \filldraw (2,0) circle (2pt);
            \filldraw (3,0) circle (2pt);
            \filldraw (4,0) circle (2pt);
            \filldraw (5,0) circle (2pt);
            \filldraw (6.5,-3) circle (2pt);
            \filldraw (7.5,-2) circle (2pt);
            \filldraw (8.5,-2) circle (2pt);
            \filldraw (9.5,-1) circle (2pt);
            \filldraw (10.5,0) circle (2pt);
            \filldraw (11.5,1) circle (2pt);
            \filldraw (12.5,0) circle (2pt);
            \filldraw (11.5,1) circle (2pt);
            \filldraw (10.5,0) circle (2pt);
            \filldraw (9.5,0) circle (2pt);
            \filldraw (8.5,0) circle (2pt);
            \filldraw (7.5,0) circle (2pt);
            \filldraw (6.5,0) circle (2pt);
            \draw (0,0) to (1,1) to (2,1) to (3,0) to (4,-1) to (5,-2);
            \draw (-1,0) to (5,0);
            \draw (6.5,0) to (12.5,0);
            \draw (6.5,-3) to (7.5,-2) to (8.5,-2) to (9.5,-1) to (10.5,0) to (11.5,1) to (12.5,0);
            \node (a) at (5.75,0) {\scalebox{.5}{$\ldots$}};
            \node (b) at (5.75,-2) {\scalebox{.5}{$\ldots$}};
            \node (c) at (5.75,-3) {\scalebox{.5}{$\ldots$}};
      \end{tikzpicture}} & \bigsqcup &
      \raisebox{-.5\height}{\begin{tikzpicture}[scale=.25]
            \node (0) at (-.5,-1) {$x$};
            \node (1) at (5.5,2) {$\Mb$};
            \filldraw (0,0) circle (2pt);
            \filldraw (1,1) circle (2pt);
            \filldraw (2,1) circle (2pt);
            \filldraw (3,0) circle (2pt);
            \filldraw (4,-1) circle (2pt);
            \filldraw (5,-2) circle (2pt);
            \filldraw (1,0) circle (2pt);
            \filldraw (2,0) circle (2pt);
            \filldraw (3,0) circle (2pt);
            \filldraw (4,0) circle (2pt);
            \filldraw (5,0) circle (2pt);
            \filldraw (6.5,-3) circle (2pt);
            \filldraw (7.5,-2) circle (2pt);
            \filldraw (8.5,-2) circle (2pt);
            \filldraw (9.5,-1) circle (2pt);
            \filldraw (10.5,0) circle (2pt);
            \filldraw (11.5,1) circle (2pt);
            \filldraw (12.5,0) circle (2pt);
            \filldraw (11.5,1) circle (2pt);
            \filldraw (10.5,0) circle (2pt);
            \filldraw (9.5,0) circle (2pt);
            \filldraw (8.5,0) circle (2pt);
            \filldraw (7.5,0) circle (2pt);
            \filldraw (6.5,0) circle (2pt);
            \draw (0,0) to (1,1) to (2,1) to (3,0) to (4,-1) to (5,-2);
            \draw (-1,0) to (5,0);
            \draw (6.5,0) to (12.5,0);
            \draw (6.5,-3) to (7.5,-2) to (8.5,-2) to (9.5,-1) to (10.5,0) to (11.5,1) to (12.5,0);
            \node (a) at (5.75,0) {\scalebox{.5}{$\ldots$}};
            \node (b) at (5.75,-2) {\scalebox{.5}{$\ldots$}};
            \node (c) at (5.75,-3) {\scalebox{.5}{$\ldots$}};
      \end{tikzpicture}} \\ & & & \bigsqcup &
      2\times \raisebox{-.5\height}{
      \begin{tikzpicture}[scale=.125]
            \draw[thick] (-1,-1) to (0,0);
            \draw[thick] (12.5,0) to (13.5,-1);
            \node (1) at (5.5,-3) {$\M$};
            \node (2) at (19,1) {$\Mb$};
            \node (3) at (-.5,-3) {$ux$};
            \node (4) at (12.5,-3) {$x$};
            \filldraw (0,0) circle (2pt);
            \filldraw (1,1) circle (2pt);
            \filldraw (2,2) circle (2pt);
            \filldraw (3,1) circle (2pt);
            \filldraw (4,1) circle (2pt);
            \filldraw (5,2) circle (2pt);
            \filldraw (1,0) circle (2pt);
            \filldraw (2,0) circle (2pt);
            \filldraw (3,0) circle (2pt);
            \filldraw (4,0) circle (2pt);
            \filldraw (5,0) circle (2pt);
            \filldraw (6.5,3) circle (2pt);
            \filldraw (7.5,2) circle (2pt);
            \filldraw (8.5,2) circle (2pt);
            \filldraw (9.5,1) circle (2pt);
            \filldraw (10.5,2) circle (2pt);
            \filldraw (11.5,1) circle (2pt);
            \filldraw (12.5,0) circle (2pt);
            \filldraw (11.5,0) circle (2pt);
            \filldraw (10.5,0) circle (2pt);
            \filldraw (9.5,0) circle (2pt);
            \filldraw (8.5,0) circle (2pt);
            \filldraw (7.5,0) circle (2pt);
            \filldraw (6.5,0) circle (2pt);
            \draw (0,0) to (1,1) to (2,2) to (3,1) to (4,1) to (5,2);
            \draw (0,0) to (5,0);
            \draw (6.5,0) to (12.5,0);
            \draw (6.5,3) to (7.5,2) to (8.5,2) to (9.5,1) to (10.5,2) to (11.5,1) to (12.5,0);
            \node (a) at (5.75,0) {\scalebox{.5}{$\ldots$}};
            \node (b) at (5.75,2) {\scalebox{.5}{$\ldots$}};
            \node (c) at (5.75,3) {\scalebox{.5}{$\ldots$}};
            \filldraw (13.5,-1) circle (2pt);
            \filldraw (14.5,0) circle (2pt);
            \filldraw (15.5,0) circle (2pt);
            \filldraw (16.5,-1) circle (2pt);
            \filldraw (17.5,-2) circle (2pt);
            \filldraw (18.5,-3) circle (2pt);
            \filldraw (14.5,-1) circle (2pt);
            \filldraw (15.5,-1) circle (2pt);
            \filldraw (16.5,-1) circle (2pt);
            \filldraw (17.5,-1) circle (2pt);
            \filldraw (18.5,-1) circle (2pt);
            \filldraw (20,-4) circle (2pt);
            \filldraw (21,-3) circle (2pt);
            \filldraw (22,-3) circle (2pt);
            \filldraw (23,-2) circle (2pt);
            \filldraw (24,-1) circle (2pt);
            \filldraw (25,0) circle (2pt);
            \filldraw (26,-1) circle (2pt);
            \filldraw (25,0) circle (2pt);
            \filldraw (25,-1) circle (2pt);
            \filldraw (23,-1) circle (2pt);
            \filldraw (22,-1) circle (2pt);
            \filldraw (21,-1) circle (2pt);
            \filldraw (20,-1) circle (2pt);
            \draw (13.5,-1) to (14.5,0) to (15.5,0) to (16.5,-1) to (17.5,-2) to (18.5,-3);
            \draw (13.5,-1) to (18.5,-1);
            \draw (20,-1) to (26,-1);
            \draw (20,-4) to (21,-3) to (22,-3) to (23,-2) to (24,-1) to (25,0) to (26,-1);
            \node (a) at (19.25,-1) {\scalebox{.5}{$\ldots$}};
            \node (b) at (19.25,-3) {\scalebox{.5}{$\ldots$}};
            \node (c) at (19.25,-4) {\scalebox{.5}{$\ldots$}};
      \end{tikzpicture}
      }
\end{array}$}

\noindent\rule{\textwidth}{1pt}

\scalebox{.8}{
$\begin{array}{ccccccccc}
 \raisebox{-.5\height}{\begin{tikzpicture}[scale=.25]
 \filldraw[lightgray,fill=lightgray] (0,0) to (1,1) to (2,1) to (3,0) to (0,0);
            \filldraw[lightgray,fill=lightgray] (3,0) to (5,0) to (5,-2) to (3,0);
            \filldraw[lightgray,fill=lightgray] (6.5,0) to (10.5,0) to (8.5,-2) to (7.5,-2) to (6.5,-3) to (6.5,0);
            \filldraw[lightgray,fill=lightgray] (10.5,0) to (11.5,1) to (12.5,0) to (10.5,0);
            \node (1) at (5.5,2) {$\Mbm$};
            \filldraw (0,0) circle (2pt);
            \filldraw (1,1) circle (2pt);
            \filldraw (2,1) circle (2pt);
            \filldraw (3,0) circle (2pt);
            \filldraw (4,-1) circle (2pt);
            \filldraw (5,-2) circle (2pt);
            \filldraw (1,0) circle (2pt);
            \filldraw (2,0) circle (2pt);
            \filldraw (3,0) circle (2pt);
            \filldraw (4,0) circle (2pt);
            \filldraw (5,0) circle (2pt);
            \filldraw (6.5,-3) circle (2pt);
            \filldraw (7.5,-2) circle (2pt);
            \filldraw (8.5,-2) circle (2pt);
            \filldraw (9.5,-1) circle (2pt);
            \filldraw (10.5,0) circle (2pt);
            \filldraw (11.5,1) circle (2pt);
            \filldraw (12.5,0) circle (2pt);
            \filldraw (11.5,1) circle (2pt);
            \filldraw (10.5,0) circle (2pt);
            \filldraw (9.5,0) circle (2pt);
            \filldraw (8.5,0) circle (2pt);
            \filldraw (7.5,0) circle (2pt);
            \filldraw (6.5,0) circle (2pt);
            \draw (0,0) to (1,1) to (2,1) to (3,0) to (4,-1) to (5,-2);
            \draw (0,0) to (5,0);
            \draw (6.5,0) to (12.5,0);
            \draw (6.5,-3) to (7.5,-2) to (8.5,-2) to (9.5,-1) to (10.5,0) to (11.5,1) to (12.5,0);
            \node (a) at (5.75,0) {\scalebox{.5}{$\ldots$}};
            \node (b) at (5.75,-2) {\scalebox{.5}{$\ldots$}};
            \node (c) at (5.75,-3) {\scalebox{.5}{$\ldots$}};
      \end{tikzpicture}} & = &
\raisebox{-.5\height}{\begin{tikzpicture}[scale=.25]
\filldraw[lightgray,fill=lightgray] (0,0) to (1,1) to (2,1) to (3,0) to (0,0);
            \filldraw[lightgray,fill=lightgray] (3,0) to (5,0) to (5,-2) to (3,0);
            \filldraw[lightgray,fill=lightgray] (6.5,0) to (10.5,0) to (8.5,-2) to (7.5,-2) to (6.5,-3) to (6.5,0);
            \filldraw[lightgray,fill=lightgray] (10.5,0) to (11.5,1) to (12.5,0) to (10.5,0);
            \filldraw (-1,0) circle (2pt);
            \node (0) at (-.5,-1) {$ux$};
            \node (1) at (5.5,2) {$\Mbm$};
            \filldraw (0,0) circle (2pt);
            \filldraw (1,1) circle (2pt);
            \filldraw (2,1) circle (2pt);
            \filldraw (3,0) circle (2pt);
            \filldraw (4,-1) circle (2pt);
            \filldraw (5,-2) circle (2pt);
            \filldraw (1,0) circle (2pt);
            \filldraw (2,0) circle (2pt);
            \filldraw (3,0) circle (2pt);
            \filldraw (4,0) circle (2pt);
            \filldraw (5,0) circle (2pt);
            \filldraw (6.5,-3) circle (2pt);
            \filldraw (7.5,-2) circle (2pt);
            \filldraw (8.5,-2) circle (2pt);
            \filldraw (9.5,-1) circle (2pt);
            \filldraw (10.5,0) circle (2pt);
            \filldraw (11.5,1) circle (2pt);
            \filldraw (12.5,0) circle (2pt);
            \filldraw (11.5,1) circle (2pt);
            \filldraw (10.5,0) circle (2pt);
            \filldraw (9.5,0) circle (2pt);
            \filldraw (8.5,0) circle (2pt);
            \filldraw (7.5,0) circle (2pt);
            \filldraw (6.5,0) circle (2pt);
            \draw (0,0) to (1,1) to (2,1) to (3,0) to (4,-1) to (5,-2);
            \draw (-1,0) to (5,0);
            \draw (6.5,0) to (12.5,0);
            \draw (6.5,-3) to (7.5,-2) to (8.5,-2) to (9.5,-1) to (10.5,0) to (11.5,1) to (12.5,0);
            \node (a) at (5.75,0) {\scalebox{.5}{$\ldots$}};
            \node (b) at (5.75,-2) {\scalebox{.5}{$\ldots$}};
            \node (c) at (5.75,-3) {\scalebox{.5}{$\ldots$}};
      \end{tikzpicture}} & \bigsqcup &
      \raisebox{-.5\height}{\begin{tikzpicture}[scale=.25]
      \filldraw[lightgray,fill=lightgray] (0,0) to (1,1) to (2,1) to (3,0) to (0,0);
            \filldraw[lightgray,fill=lightgray] (3,0) to (5,0) to (5,-2) to (3,0);
            \filldraw[lightgray,fill=lightgray] (6.5,0) to (10.5,0) to (8.5,-2) to (7.5,-2) to (6.5,-3) to (6.5,0);
            \filldraw[lightgray,fill=lightgray] (10.5,0) to (11.5,1) to (12.5,0) to (10.5,0);
            \node (0) at (-.5,-1) {$x$};
            \node (1) at (5.5,2) {$\Mbm$};
            \filldraw (0,0) circle (2pt);
            \filldraw (1,1) circle (2pt);
            \filldraw (2,1) circle (2pt);
            \filldraw (3,0) circle (2pt);
            \filldraw (4,-1) circle (2pt);
            \filldraw (5,-2) circle (2pt);
            \filldraw (1,0) circle (2pt);
            \filldraw (2,0) circle (2pt);
            \filldraw (3,0) circle (2pt);
            \filldraw (4,0) circle (2pt);
            \filldraw (5,0) circle (2pt);
            \filldraw (6.5,-3) circle (2pt);
            \filldraw (7.5,-2) circle (2pt);
            \filldraw (8.5,-2) circle (2pt);
            \filldraw (9.5,-1) circle (2pt);
            \filldraw (10.5,0) circle (2pt);
            \filldraw (11.5,1) circle (2pt);
            \filldraw (12.5,0) circle (2pt);
            \filldraw (11.5,1) circle (2pt);
            \filldraw (10.5,0) circle (2pt);
            \filldraw (9.5,0) circle (2pt);
            \filldraw (8.5,0) circle (2pt);
            \filldraw (7.5,0) circle (2pt);
            \filldraw (6.5,0) circle (2pt);
            \draw (0,0) to (1,1) to (2,1) to (3,0) to (4,-1) to (5,-2);
            \draw (-1,0) to (5,0);
            \draw (6.5,0) to (12.5,0);
            \draw (6.5,-3) to (7.5,-2) to (8.5,-2) to (9.5,-1) to (10.5,0) to (11.5,1) to (12.5,0);
            \node (a) at (5.75,0) {\scalebox{.5}{$\ldots$}};
            \node (b) at (5.75,-2) {\scalebox{.5}{$\ldots$}};
            \node (c) at (5.75,-3) {\scalebox{.5}{$\ldots$}};
      \end{tikzpicture}} & & & &\\
      & \bigsqcup &
      2\times \Bigg( \raisebox{-.5\height}{
      \begin{tikzpicture}[scale=.125]
      \filldraw[lightgray,fill=lightgray] (0,0) to (1,1) to (2,2) to (3,1) to (4,1) to (5,2) to (5,0) to (0,0);
            \filldraw[lightgray,fill=lightgray] (6.5,3) to (7.5,2) to (8.5,2) to (9.5,1) to (10.5,2) to (11.5,1) to (12.5,0) to (6.5,0) to (6.5,3);
            \draw[thick] (-1,-1) to (0,0);
            \draw[thick] (12.5,0) to (13.5,-1);
            \node (1) at (5.5,-3) {$\Mm$};
            \node (2) at (19,1) {$\Mb$};
            \node (3) at (-.5,-3) {$ux$};
            \node (4) at (12.5,-3) {$x$};
            \filldraw (0,0) circle (2pt);
            \filldraw (1,1) circle (2pt);
            \filldraw (2,2) circle (2pt);
            \filldraw (3,1) circle (2pt);
            \filldraw (4,1) circle (2pt);
            \filldraw (5,2) circle (2pt);
            \filldraw (1,0) circle (2pt);
            \filldraw (2,0) circle (2pt);
            \filldraw (3,0) circle (2pt);
            \filldraw (4,0) circle (2pt);
            \filldraw (5,0) circle (2pt);
            \filldraw (6.5,3) circle (2pt);
            \filldraw (7.5,2) circle (2pt);
            \filldraw (8.5,2) circle (2pt);
            \filldraw (9.5,1) circle (2pt);
            \filldraw (10.5,2) circle (2pt);
            \filldraw (11.5,1) circle (2pt);
            \filldraw (12.5,0) circle (2pt);
            \filldraw (11.5,0) circle (2pt);
            \filldraw (10.5,0) circle (2pt);
            \filldraw (9.5,0) circle (2pt);
            \filldraw (8.5,0) circle (2pt);
            \filldraw (7.5,0) circle (2pt);
            \filldraw (6.5,0) circle (2pt);
            \draw (0,0) to (1,1) to (2,2) to (3,1) to (4,1) to (5,2);
            \draw (0,0) to (5,0);
            \draw (6.5,0) to (12.5,0);
            \draw (6.5,3) to (7.5,2) to (8.5,2) to (9.5,1) to (10.5,2) to (11.5,1) to (12.5,0);
            \node (a) at (5.75,0) {\scalebox{.5}{$\ldots$}};
            \node (b) at (5.75,2) {\scalebox{.5}{$\ldots$}};
            \node (c) at (5.75,3) {\scalebox{.5}{$\ldots$}};
            \filldraw (13.5,-1) circle (2pt);
            \filldraw (14.5,0) circle (2pt);
            \filldraw (15.5,0) circle (2pt);
            \filldraw (16.5,-1) circle (2pt);
            \filldraw (17.5,-2) circle (2pt);
            \filldraw (18.5,-3) circle (2pt);
            \filldraw (14.5,-1) circle (2pt);
            \filldraw (15.5,-1) circle (2pt);
            \filldraw (16.5,-1) circle (2pt);
            \filldraw (17.5,-1) circle (2pt);
            \filldraw (18.5,-1) circle (2pt);
            \filldraw (20,-4) circle (2pt);
            \filldraw (21,-3) circle (2pt);
            \filldraw (22,-3) circle (2pt);
            \filldraw (23,-2) circle (2pt);
            \filldraw (24,-1) circle (2pt);
            \filldraw (25,0) circle (2pt);
            \filldraw (26,-1) circle (2pt);
            \filldraw (25,0) circle (2pt);
            \filldraw (25,-1) circle (2pt);
            \filldraw (23,-1) circle (2pt);
            \filldraw (22,-1) circle (2pt);
            \filldraw (21,-1) circle (2pt);
            \filldraw (20,-1) circle (2pt);
            \draw (13.5,-1) to (14.5,0) to (15.5,0) to (16.5,-1) to (17.5,-2) to (18.5,-3);
            \draw (13.5,-1) to (18.5,-1);
            \draw (20,-1) to (26,-1);
            \draw (20,-4) to (21,-3) to (22,-3) to (23,-2) to (24,-1) to (25,0) to (26,-1);
            \node (a) at (19.25,-1) {\scalebox{.5}{$\ldots$}};
            \node (b) at (19.25,-3) {\scalebox{.5}{$\ldots$}};
            \node (c) at (19.25,-4) {\scalebox{.5}{$\ldots$}};
      \end{tikzpicture}
      } & \bigsqcup & \raisebox{-.5\height}{
      \begin{tikzpicture}[scale=.125]
      \filldraw[lightgray,fill=lightgray] (13.5,-1) to (14.5,0) to (15.5,0) to (16.5,-1) to (13.5,-1);
            \filldraw[lightgray,fill=lightgray] (16.5,-1) to (18.5,-1) to (18.5,-3) to (16.5,-1);
            \filldraw[lightgray,fill=lightgray] (20,-1) to (24,-1) to (22,-3) to (21,-3) to (20,-4) to (20,-1);
            \filldraw[lightgray,fill=lightgray] (24,-1) to (25,0) to (26,-1) to (24,-1);
            \draw[thick] (-1,-1) to (0,0);
            \draw[thick] (12.5,0) to (13.5,-1);
            \node (1) at (5.5,-3) {$\M$};
            \node (2) at (19,1) {$\Mbm$};
            \node (3) at (-.5,-3) {$ux$};
            \node (4) at (12.5,-3) {$x$};
            \filldraw (0,0) circle (2pt);
            \filldraw (1,1) circle (2pt);
            \filldraw (2,2) circle (2pt);
            \filldraw (3,1) circle (2pt);
            \filldraw (4,1) circle (2pt);
            \filldraw (5,2) circle (2pt);
            \filldraw (1,0) circle (2pt);
            \filldraw (2,0) circle (2pt);
            \filldraw (3,0) circle (2pt);
            \filldraw (4,0) circle (2pt);
            \filldraw (5,0) circle (2pt);
            \filldraw (6.5,3) circle (2pt);
            \filldraw (7.5,2) circle (2pt);
            \filldraw (8.5,2) circle (2pt);
            \filldraw (9.5,1) circle (2pt);
            \filldraw (10.5,2) circle (2pt);
            \filldraw (11.5,1) circle (2pt);
            \filldraw (12.5,0) circle (2pt);
            \filldraw (11.5,0) circle (2pt);
            \filldraw (10.5,0) circle (2pt);
            \filldraw (9.5,0) circle (2pt);
            \filldraw (8.5,0) circle (2pt);
            \filldraw (7.5,0) circle (2pt);
            \filldraw (6.5,0) circle (2pt);
            \draw (0,0) to (1,1) to (2,2) to (3,1) to (4,1) to (5,2);
            \draw (0,0) to (5,0);
            \draw (6.5,0) to (12.5,0);
            \draw (6.5,3) to (7.5,2) to (8.5,2) to (9.5,1) to (10.5,2) to (11.5,1) to (12.5,0);
            \node (a) at (5.75,0) {\scalebox{.5}{$\ldots$}};
            \node (b) at (5.75,2) {\scalebox{.5}{$\ldots$}};
            \node (c) at (5.75,3) {\scalebox{.5}{$\ldots$}};
            \filldraw (13.5,-1) circle (2pt);
            \filldraw (14.5,0) circle (2pt);
            \filldraw (15.5,0) circle (2pt);
            \filldraw (16.5,-1) circle (2pt);
            \filldraw (17.5,-2) circle (2pt);
            \filldraw (18.5,-3) circle (2pt);
            \filldraw (14.5,-1) circle (2pt);
            \filldraw (15.5,-1) circle (2pt);
            \filldraw (16.5,-1) circle (2pt);
            \filldraw (17.5,-1) circle (2pt);
            \filldraw (18.5,-1) circle (2pt);
            \filldraw (20,-4) circle (2pt);
            \filldraw (21,-3) circle (2pt);
            \filldraw (22,-3) circle (2pt);
            \filldraw (23,-2) circle (2pt);
            \filldraw (24,-1) circle (2pt);
            \filldraw (25,0) circle (2pt);
            \filldraw (26,-1) circle (2pt);
            \filldraw (25,0) circle (2pt);
            \filldraw (25,-1) circle (2pt);
            \filldraw (23,-1) circle (2pt);
            \filldraw (22,-1) circle (2pt);
            \filldraw (21,-1) circle (2pt);
            \filldraw (20,-1) circle (2pt);
            \draw (13.5,-1) to (14.5,0) to (15.5,0) to (16.5,-1) to (17.5,-2) to (18.5,-3);
            \draw (13.5,-1) to (18.5,-1);
            \draw (20,-1) to (26,-1);
            \draw (20,-4) to (21,-3) to (22,-3) to (23,-2) to (24,-1) to (25,0) to (26,-1);
            \node (a) at (19.25,-1) {\scalebox{.5}{$\ldots$}};
            \node (b) at (19.25,-3) {\scalebox{.5}{$\ldots$}};
            \node (c) at (19.25,-4) {\scalebox{.5}{$\ldots$}};
      \end{tikzpicture}
      }\\ & &  & \bigsqcup & \raisebox{-.5\height}{
      \begin{tikzpicture}[scale=.125]
            \filldraw[lightgray,fill=lightgray] (-1,-1) to (0,0) to (12.5,0) to (13.5,-1) to (-1,-1);
            \draw[thick] (-1,-1) to (0,0);
            \draw[thick] (12.5,0) to (13.5,-1);
            \node (1) at (5.5,-4) {\scalebox{0.8}{$\partial_x(x\M)$}};
            \node (3) at (-.5,-3) {$ux$};
            \node (4) at (12.5,-3) {$x$};
            \node (2) at (19,1) {$\Mb$};
            \filldraw (0,0) circle (2pt);
            \filldraw (1,1) circle (2pt);
            \filldraw (2,2) circle (2pt);
            \filldraw (3,1) circle (2pt);
            \filldraw (4,1) circle (2pt);
            \filldraw (5,2) circle (2pt);
            \filldraw (1,0) circle (2pt);
            \filldraw (2,0) circle (2pt);
            \filldraw (3,0) circle (2pt);
            \filldraw (4,0) circle (2pt);
            \filldraw (5,0) circle (2pt);
            \filldraw (6.5,3) circle (2pt);
            \filldraw (7.5,2) circle (2pt);
            \filldraw (8.5,2) circle (2pt);
            \filldraw (9.5,1) circle (2pt);
            \filldraw (10.5,2) circle (2pt);
            \filldraw (11.5,1) circle (2pt);
            \filldraw (12.5,0) circle (2pt);
            \filldraw (11.5,0) circle (2pt);
            \filldraw (10.5,0) circle (2pt);
            \filldraw (9.5,0) circle (2pt);
            \filldraw (8.5,0) circle (2pt);
            \filldraw (7.5,0) circle (2pt);
            \filldraw (6.5,0) circle (2pt);
            \draw (0,0) to (1,1) to (2,2) to (3,1) to (4,1) to (5,2);
            \draw (0,0) to (5,0);
            \draw (6.5,0) to (12.5,0);
            \draw (6.5,3) to (7.5,2) to (8.5,2) to (9.5,1) to (10.5,2) to (11.5,1) to (12.5,0);
            \node (a) at (5.75,0) {\scalebox{.5}{$\ldots$}};
            \node (b) at (5.75,2) {\scalebox{.5}{$\ldots$}};
            \node (c) at (5.75,3) {\scalebox{.5}{$\ldots$}};
            \filldraw (13.5,-1) circle (2pt);
            \filldraw (14.5,0) circle (2pt);
            \filldraw (15.5,0) circle (2pt);
            \filldraw (16.5,-1) circle (2pt);
            \filldraw (17.5,-2) circle (2pt);
            \filldraw (18.5,-3) circle (2pt);
            \filldraw (14.5,-1) circle (2pt);
            \filldraw (15.5,-1) circle (2pt);
            \filldraw (16.5,-1) circle (2pt);
            \filldraw (17.5,-1) circle (2pt);
            \filldraw (18.5,-1) circle (2pt);
            \filldraw (20,-4) circle (2pt);
            \filldraw (21,-3) circle (2pt);
            \filldraw (22,-3) circle (2pt);
            \filldraw (23,-2) circle (2pt);
            \filldraw (24,-1) circle (2pt);
            \filldraw (25,0) circle (2pt);
            \filldraw (26,-1) circle (2pt);
            \filldraw (25,0) circle (2pt);
            \filldraw (25,-1) circle (2pt);
            \filldraw (23,-1) circle (2pt);
            \filldraw (22,-1) circle (2pt);
            \filldraw (21,-1) circle (2pt);
            \filldraw (20,-1) circle (2pt);
            \draw (13.5,-1) to (14.5,0) to (15.5,0) to (16.5,-1) to (17.5,-2) to (18.5,-3);
            \draw (13.5,-1) to (18.5,-1);
            \draw (20,-1) to (26,-1);
            \draw (20,-4) to (21,-3) to (22,-3) to (23,-2) to (24,-1) to (25,0) to (26,-1);
            \node (a) at (19.25,-1) {\scalebox{.5}{$\ldots$}};
            \node (b) at (19.25,-3) {\scalebox{.5}{$\ldots$}};
            \node (c) at (19.25,-4) {\scalebox{.5}{$\ldots$}};
      \end{tikzpicture}
      }  \Bigg)
\end{array}$}

      \caption{The decompositions of lattice paths in $\M$, $\Mm$, $\Mb$, and $\Mbm$ from~\Cref{prop:m_decomp},~\Cref{prop:mm_decomp},~\Cref{prop:mb_decomp}, and~\Cref{prop:mbm_decomp}.  The multiplication by 2 comes from the symmetry of reflecting the first factor across the $x$-axis.}
      \label{fig:m_decomp}
\end{figure}

Define the generating function for the total area of paths in $\M$ by \[\Mm(x,u) := \sum_{n \geq 0 } \sum_{k \geq 0} x^n u^k \sum_{p \in \M_{n,k}} \d(p).\]

\begin{proposition} \label{prop:mm_decomp}
      The generating function $\Mm(x,u)$ satisfies the functional equation
            \[\Mm(x,u)=x(1+u)\Mm(x,u)+ux^2(2\M(x,u)\Mm(x,u)+\M(x,u) \frac{d}{dx}(x\M(x,u)),\]
      with the explicit solution
            \[\Mm(x,u) = \frac{\left(u^2+1\right) x^2-(u+1) x+((u +1)x-1) \left(\sqrt{(u-1)^2 x^2-2 (u+1) x+1}-1\right)}{2 u x^2 \left((u-1)^2 x^2-2 (u+1) x+1\right)}.\]
\end{proposition}

\begin{proof}
      As in~\Cref{prop:m_decomp}, the functional equation comes from decomposing a lattice path $r$ in $\Mm$ by first return to the $x$-axis: either $r$ is empty (in which case it contributes zero area); or $r$ starts with an $O_1$ step; or $r$ starts with an $O_2$ step; or $r$ starts with a $U$ step.  This is illustrated in~\Cref{fig:m_decomp}.  The explicit solution is easily obtained by using the explicit form of $\M(x,u)$ from~\Cref{prop:m_decomp} and solving the linear equation in $\Mm(x,u)$.
\end{proof}





\subsection{Bilateral Motzkin paths}

Define the generating function for bilateral Motzkin paths to be
      \[\Mb(x,u) = \sum_{n \geq 0 } \sum_{k \geq 0} |\Mb_{n,k}| x^{n} u^k,\]
so that the coefficient of $x^n u^k$ counts bilateral Motzkin paths of total length $n$ with exactly $k$ steps of the form $U$ or $O_1$.

\begin{proposition}\label{prop:mb_decomp}
      The generating function $\Mb(x,u)$ satisfies the functional equation
            \[\Mb(x,u)=1+x(1+u)\Mb(x,u)+2u x^2 \M(x,u) \Mb(x,u),\]
      with the explicit solution
      \[\Mb(x,u) = (u^2x^2 - 2ux^2 - 2ux + x^2 - 2x + 1)^{-1/2}.\]
\end{proposition}

\begin{proof}
      As in~\Cref{prop:m_decomp}, the functional equation comes from decomposing a lattice path $r$ in $\Mb$ by first return to the $x$-axis: $r$ is empty; or $r$ starts with an $O_1$ step; or $r$ starts with an $O_2$ step; or $r$ starts with a $U$ or $D$ step.  This is illustrated in~\Cref{fig:m_decomp}.  The explicit solution is easily obtained by using the explicit form of $\M(x,u)$ from~\Cref{prop:m_decomp} and solving the linear equation in $\Mb(x,u)$.
\end{proof}

\begin{remark}
Because $\Mb_{n,k}$ encodes $\Pt_{k,n-k}$ by Proposition~\ref{prop:bija}, we have
\begin{equation}
\label{eq:expansion_bilateral}
\Mb(x,u)=\sum_{n,k\geq 0}\binom{n}{k}^2x^nu^k.
\end{equation}
\end{remark}

Define the generating function for the total area of paths in $\Mb$ to be
      \[\Mbm(x,u) := \sum_{n \geq 0 } \sum_{k \geq 0} x^n u^k \sum_{p \in \Mb_{n,k}} \d(p).\]

\begin{proposition}\label{prop:mbm_decomp}
      The generating function $\Mbm(x,u)$ satisfies the functional equation
            \[\Mbm(x,u)=x(1+u)\Mbm(x,u)+2ux^2(\M(x,u)\Mbm(x,u)+\Mm(x,u)\Mb(x,u)+\Mb(x,u)\frac{d}{dx}(x\M(x,u)),\]
      with the explicit solution
      \begin{equation}\label{eq:total_geo_sub}
            \Mbm(x,u) = \frac{2 u x^2}{\left((u-1)^2 x^2-2 (u+1) x+1\right)^2}.
      \end{equation}
\end{proposition}

\begin{proof}
      As in~\Cref{prop:mm_decomp}, the functional equation comes from decomposing a lattice path $r$ in $\Mbm$ by first return to the $x$-axis: if $r$ is empty, then it counts for zero area; otherwise, $r$ starts with an $O_1$ step; or $r$ starts with an $O_2$ step; or $r$ starts with a $U$ or $D$ step.  This is illustrated in~\Cref{fig:m_decomp}.  The explicit solution is easily obtained by using the explicit forms of $\M(x,u)$, $\Mb(x,u)$, and $\Mm(x,u)$ from~\Cref{prop:m_decomp,prop:mb_decomp,prop:mm_decomp} and solving the linear equation in $\Mbm(x,u)$.
\end{proof}

Substituting $u=y/x$ into~\Cref{eq:total_geo_sub}, we obtain~\Cref{eq:total_geo} and thus complete the proof of~\Cref{thm:gf} for the generating function for the Wiener indices of the lattices $\P_{m,k}=J([m]\times[k])$.





\section{Proofs of~\Cref{cor:coeff,cor:asymptotic}}\label{sec:proof_rect_cors}

We note that
\[x^2 - 2xy + y^2 - 2x - 2y + 1 = (q - t - 1) (q - t + 1) (q + t - 1) (q + t + 1),\]
where $q^2=x$ and $t^2=y$.  Then, by performing a partial fraction decomposition with a computer algebra system, we get
\begin{align} \label{eq:qt}
\nonumber \frac{2xy}{(x^2 - 2xy + y^2 - 2x - 2y + 1)^2} &=
      \frac{1}{32} \left(
            \frac{1}{(-1-t+q)^2}
            +\frac{1}{(1-t+q)^2}
            +\frac{1}{(-1+t+q)^2}
            +\frac{1}{(1+t+q)^2}\right) \\
 \nonumber     &+\frac{1}{32t(t+1)}\left(
            \frac{1+t+t^2}{(1+t-q)}
            +\frac{1+t+t^2}{(1+t+q)}
      \right)\\
      &+\frac{1}{32t(t-1)}\left(
            \frac{1-t+t^2}{(1-t+q)}
            +\frac{-1+t-t^2}{(-1+t+q)}
      \right).
\end{align}

Taking the $n$th coefficient in $q$ from the right-hand side of~\Cref{eq:qt} gives

\begin{align}\label{eq:qt2}
\nonumber \frac{n+1}{32}&\Big((1+t)^{-2-n}+(-1+t)^{-2-n}+(1-t)^{-2-n}+(-1-t)^{-2-n}\Big)\\
\nonumber +
\frac{1}{32t}& \Big(\left(t^2+t+1\right) (t+1)^{-n-2}+\left(t^2+t+1\right) (-t-1)^{-n-2}\Big)\\
+
\frac{1}{32t}& \Big(\left(-t^2+t-1\right) (1-t)^{-n-2}+\left(-t^2+t-1\right) (t-1)^{-n-2}\Big).
\end{align}

Taking the $j$th coefficient in $t$ from~\eqref{eq:qt2} and simplifying gives

\begin{align}\label{eq:qt3}
\frac{\left((-1)^n+1\right) \left((-1)^j+1\right)}{32} \left(
 \frac{(n+j+1)!}{n!j!}-
  \frac{(n+j)! ((n+1)^2+j^2+j (n+2))}{(n+1)! (j+1)!}\right).
\end{align}

Restricting~\eqref{eq:qt3} to $n=2m$ and $j=2k$ even, we obtain the desired expression in~\Cref{cor:coeff} for the coefficient of $q^{n} t^{m}$, giving the coefficient for $x^{m} y^{k}$:

\begin{align*}
&\frac{1}{8} \left(\frac{(n+j+1)!}{n!j!}-
\frac{ (n+j)! ((n+1)^2+j^2+j(n+2))}{(n+1)! (j+1)!}\right)\\
&=\frac{(n+j)!}{8 n!j!}\left(n+j+1-\frac{j^2+(n+1)^2+j (n+2)}{(n+1)(j+1)}\right)\\
&=\frac{nj}{8(n+j+1)}\binom{n+j+2}{j+1}\\
&=\frac{mk}{4m+4k+2}\binom{2m+2k+2}{2k+1}.
\end{align*}

Given the exact expression for $\d(\P_{m,k})$,~\Cref{cor:asymptotic} is routine using Stirling's asymptotic equivalent for factorials.

\section{Shifted staircases}\label{sec:bij_stair}

As for rectangles, we can view elements of $\SS_n$ as lattice paths of length $n$ that start at $(0,0)$ and use steps of the form $U=(1,1)$ and $D=(1,-1)$. The main difference is that paths representing different order ideals can have different endpoints. Let \[\SS_n^\times:=\SS_n \times \SS_n \quad\text{and}\quad \SS_n^{\leq}:=\{(p,q)\in \SS_n^\times:p\leq q\}.\] 

\begin{definition}
Define a \defn{bilateral Motzkin prefix} to be a lattice path that starts at $(0,0)$ and uses the steps of the form $U,D,O_1,O_2$. Let $\mathcal V$ denote the set of bilateral Motzkin prefixes, and let $\mathcal V_n$ be the set of bilateral Motzkin prefixes that use exactly $n$ steps. A \defn{bicolored Motzkin prefix} is a bilateral Motzkin prefix that stays weakly above the $x$-axis. Write $\mathcal N$ for the set of bicolored Motzkin prefixes, and let $\mathcal N_n=\mathcal N\cap\mathcal V_n$. 
\end{definition}

Throughout this section, we write $d(p,q)$ for the length of a geodesic between bilateral Motzkin prefixes $p$ and $q$ in the Hasse diagram of $\SS_n$. 

By applying the same rules as in the table in \eqref{eq:bija}, we can transform a pair $(p,q)\in \SS_n\times \SS_n$ into a path $A(p,q)$ that uses steps $U,D,O_1,O_2$. The path $A(p,q)$ is similar to a bilateral Motzkin path, except it does not necessarily end on the $x$-axis. 

\begin{proposition}\label{prop:bija_SS}
The map $\bija\colon\SS_n^\times\to\mathcal V_n$ is a bijection satisfying $\d(p,q) = \overline\d(\bija(p,q))$, and it restricts to a bijection from $\SS_n^\leq$ to $\mathcal N_n$.
\end{proposition}

\begin{proof}
      The proof is essentially the same as that of \Cref{prop:bija}.
\end{proof}

\section{Proof of \Cref{thm:gf_SS} and \Cref{cor:Wiener_JSn}}\label{sec:proof_stair}
\subsection{Bicolored Motzkin prefixes}

Let \[\mathcal N(x):=\sum_{n\geq 0}|\mathcal N_n|x^n\] be the generating function for bicolored Motzkin prefixes. 

\begin{proposition}\label{prop:N}
The generating function $\mathcal N(x)$ satisfies the functional equation \[\mathcal N(x)=1+3x\mathcal N(x)+x^2\mathcal M(x,1)\mathcal N(x).\] Thus, \[\mathcal N(x)=\frac{2}{1-4x+\sqrt{1-4x}}.\]
\end{proposition}
\begin{proof}
The expression $1+x\mathcal N(x)$ counts (possibly empty) bicolored Motzkin prefixes that only touch the $x$-axis at $(0,0)$, while the expression $2x\mathcal N(x)+x^2\mathcal M(x,1)\mathcal N(x)$ counts bicolored Motzkin prefixes that touch the $x$-axis at some point other than $(0,0)$. This is illustrated on the first line of~\Cref{fig:n_decomp}. It is routine to derive the explicit solution from the functional equation and \Cref{prop:m_decomp}. 
\end{proof}


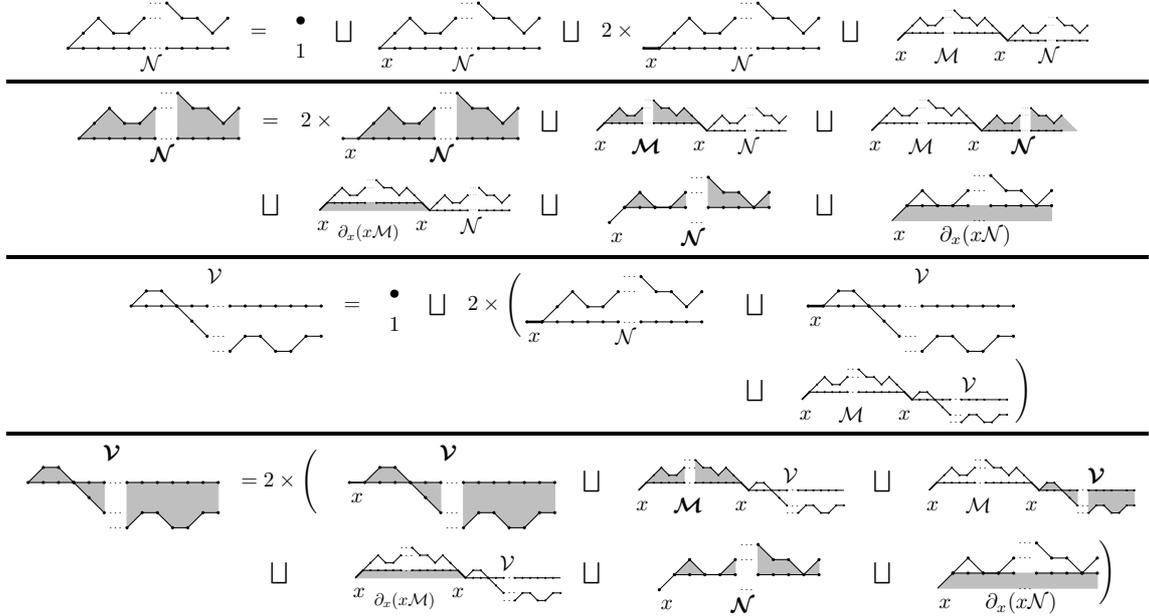
\begin{figure}[htbp]
\scalebox{.8}{
$\begin{array}{ccccccccc}
 \raisebox{-.5\height}{\begin{tikzpicture}[scale=.25]
            \node (1) at (5.5,-1) {$\mathcal N$};
            \filldraw (0,0) circle (2pt);
            \filldraw (1,1) circle (2pt);
            \filldraw (2,2) circle (2pt);
            \filldraw (3,1) circle (2pt);
            \filldraw (4,1) circle (2pt);
            \filldraw (5,2) circle (2pt);
            \filldraw (1,0) circle (2pt);
            \filldraw (2,0) circle (2pt);
            \filldraw (3,0) circle (2pt);
            \filldraw (4,0) circle (2pt);
            \filldraw (5,0) circle (2pt);
            \filldraw (6.5,3) circle (2pt);
            \filldraw (7.5,2) circle (2pt);
            \filldraw (8.5,2) circle (2pt);
            \filldraw (9.5,1) circle (2pt);
            \filldraw (10.5,2) circle (2pt);
            \filldraw (10.5,0) circle (2pt);
            \filldraw (9.5,0) circle (2pt);
            \filldraw (8.5,0) circle (2pt);
            \filldraw (7.5,0) circle (2pt);
            \filldraw (6.5,0) circle (2pt);
            \draw (0,0) to (1,1) to (2,2) to (3,1) to (4,1) to (5,2);
            \draw (0,0) to (5,0);
            \draw (6.5,0) to (10.5,0);
            \draw (6.5,3) to (7.5,2) to (8.5,2) to (9.5,1) to (10.5,2);
            \node (a) at (5.75,0) {\scalebox{.5}{$\ldots$}};
            \node (b) at (5.75,2) {\scalebox{.5}{$\ldots$}};
            \node (c) at (5.75,3) {\scalebox{.5}{$\ldots$}};
      \end{tikzpicture}} & = & \raisebox{-.5\height}{\begin{tikzpicture} \node (0) at (0,0) {$\bullet$}; \node (1) at (0,-.5) {$1$}; \end{tikzpicture}} & \bigsqcup &  
\raisebox{-.5\height}{\begin{tikzpicture}[scale=.25]
            \node (0) at (.5,-1) {$x$};
            \node (1) at (5.5,-1) {$\mathcal N$};
            \filldraw (0,0) circle (2pt);
            \filldraw (1,1) circle (2pt);
            \filldraw (2,2) circle (2pt);
            \filldraw (3,1) circle (2pt);
            \filldraw (4,1) circle (2pt);
            \filldraw (5,2) circle (2pt);
            \filldraw (1,0) circle (2pt);
            \filldraw (2,0) circle (2pt);
            \filldraw (3,0) circle (2pt);
            \filldraw (4,0) circle (2pt);
            \filldraw (5,0) circle (2pt);
            \filldraw (6.5,3) circle (2pt);
            \filldraw (7.5,2) circle (2pt);
            \filldraw (8.5,2) circle (2pt);
            \filldraw (9.5,1) circle (2pt);
            \filldraw (10.5,2) circle (2pt);
            \filldraw (10.5,0) circle (2pt);
            \filldraw (9.5,0) circle (2pt);
            \filldraw (8.5,0) circle (2pt);
            \filldraw (7.5,0) circle (2pt);
            \filldraw (6.5,0) circle (2pt);
            \draw (0,0) to (1,1) to (2,2) to (3,1) to (4,1) to (5,2);
            \draw (0,0) to (5,0);
            \draw (6.5,0) to (10.5,0);
            \draw (6.5,3) to (7.5,2) to (8.5,2) to (9.5,1) to (10.5,2);
            \node (a) at (5.75,0) {\scalebox{.5}{$\ldots$}};
            \node (b) at (5.75,2) {\scalebox{.5}{$\ldots$}};
            \node (c) at (5.75,3) {\scalebox{.5}{$\ldots$}};
      \end{tikzpicture}} & \bigsqcup & 2\times
      \raisebox{-.5\height}{\begin{tikzpicture}[scale=.25]
            \node (0) at (-.5,-1) {$x$};
            \node (1) at (5.5,-1) {$\mathcal N$};
            \filldraw (0,0) circle (2pt);
            \filldraw (-1,0) circle (2pt);
            \filldraw (0,0) circle (2pt);
            \filldraw (1,1) circle (2pt);
            \filldraw (2,2) circle (2pt);
            \filldraw (3,1) circle (2pt);
            \filldraw (4,1) circle (2pt);
            \filldraw (5,2) circle (2pt);
            \filldraw (1,0) circle (2pt);
            \filldraw (2,0) circle (2pt);
            \filldraw (3,0) circle (2pt);
            \filldraw (4,0) circle (2pt);
            \filldraw (5,0) circle (2pt);
            \filldraw (6.5,3) circle (2pt);
            \filldraw (7.5,2) circle (2pt);
            \filldraw (8.5,2) circle (2pt);
            \filldraw (9.5,1) circle (2pt);
            \filldraw (10.5,2) circle (2pt);
            \filldraw (10.5,0) circle (2pt);
            \filldraw (9.5,0) circle (2pt);
            \filldraw (8.5,0) circle (2pt);
            \filldraw (7.5,0) circle (2pt);
            \filldraw (6.5,0) circle (2pt);
            \draw[very thick] (-1,0) to (0,0);
            \draw (0,0) to (1,1) to (2,2) to (3,1) to (4,1) to (5,2);
            \draw (0,0) to (5,0);
            \draw (6.5,0) to (10.5,0);
            \draw (6.5,3) to (7.5,2) to (8.5,2) to (9.5,1) to (10.5,2);
            \node (a) at (5.75,0) {\scalebox{.5}{$\ldots$}};
            \node (b) at (5.75,2) {\scalebox{.5}{$\ldots$}};
            \node (c) at (5.75,3) {\scalebox{.5}{$\ldots$}};
      \end{tikzpicture}} & \bigsqcup &
      \raisebox{-.5\height}{
      \begin{tikzpicture}[scale=.125]
            \draw[thick] (-1,-1) to (0,0);
            \draw[thick] (12.5,0) to (13.5,-1);
            \node (1) at (5.5,-3) {$\M$};
            \node (2) at (19,-3) {$\mathcal N$};
            \node (3) at (-.5,-3) {$x$};
            \node (4) at (12.5,-3) {$x$};
            \filldraw (0,0) circle (2pt);
            \filldraw (1,1) circle (2pt);
            \filldraw (2,2) circle (2pt);
            \filldraw (3,1) circle (2pt);
            \filldraw (4,1) circle (2pt);
            \filldraw (5,2) circle (2pt);
            \filldraw (1,0) circle (2pt);
            \filldraw (2,0) circle (2pt);
            \filldraw (3,0) circle (2pt);
            \filldraw (4,0) circle (2pt);
            \filldraw (5,0) circle (2pt);
            \filldraw (6.5,3) circle (2pt);
            \filldraw (7.5,2) circle (2pt);
            \filldraw (8.5,2) circle (2pt);
            \filldraw (9.5,1) circle (2pt);
            \filldraw (10.5,2) circle (2pt);
            \filldraw (11.5,1) circle (2pt);
            \filldraw (12.5,0) circle (2pt);
            \filldraw (11.5,0) circle (2pt);
            \filldraw (10.5,0) circle (2pt);
            \filldraw (9.5,0) circle (2pt);
            \filldraw (8.5,0) circle (2pt);
            \filldraw (7.5,0) circle (2pt);
            \filldraw (6.5,0) circle (2pt);
            \draw (0,0) to (1,1) to (2,2) to (3,1) to (4,1) to (5,2);
            \draw (0,0) to (5,0);
            \draw (6.5,0) to (12.5,0);
            \draw (6.5,3) to (7.5,2) to (8.5,2) to (9.5,1) to (10.5,2) to (11.5,1) to (12.5,0);
            \node (a) at (5.75,0) {\scalebox{.25}{$\ldots$}};
            \node (b) at (5.75,2) {\scalebox{.25}{$\ldots$}};
            \node (c) at (5.75,3) {\scalebox{.25}{$\ldots$}};
            \filldraw (13.5,-1) circle (2pt);
            \filldraw (14.5,0) circle (2pt);
            \filldraw (15.5,1) circle (2pt);
            \filldraw (16.5,0) circle (2pt);
            \filldraw (17.5,0) circle (2pt);
            \filldraw (18.5,1) circle (2pt);
            \filldraw (14.5,-1) circle (2pt);
            \filldraw (15.5,-1) circle (2pt);
            \filldraw (16.5,-1) circle (2pt);
            \filldraw (17.5,-1) circle (2pt);
            \filldraw (18.5,-1) circle (2pt);
            \filldraw (20,2) circle (2pt);
            \filldraw (21,1) circle (2pt);
            \filldraw (22,1) circle (2pt);
            \filldraw (23,0) circle (2pt);
            \filldraw (24,1) circle (2pt);
            \filldraw (24,-1) circle (2pt);
            \filldraw (23,-1) circle (2pt);
            \filldraw (22,-1) circle (2pt);
            \filldraw (21,-1) circle (2pt);
            \filldraw (20,-1) circle (2pt);
            \draw (13.5,-1) to (14.5,0) to (15.5,1) to (16.5,0) to (17.5,0) to (18.5,1);
            \draw (13.5,-1) to (18.5,-1);
            \draw (20,-1) to (24,-1);
            \draw (20,2) to (21,1) to (22,1) to (23,0) to (24,1);
            \node (d) at (19.25,-1) {\scalebox{.25}{$\ldots$}};
            \node (e) at (19.25,1) {\scalebox{.25}{$\ldots$}};
            \node (f) at (19.25,2) {\scalebox{.25}{$\ldots$}};
      \end{tikzpicture}
      }
\end{array}$}

\noindent\rule{\textwidth}{1pt}

\scalebox{.8}{
$\begin{array}{ccccccc}
 \raisebox{-.5\height}{\begin{tikzpicture}[scale=.25]
             \filldraw[lightgray,fill=lightgray] (0,0) to (1,1) to (2,2) to (3,1) to (4,1) to (5,2) to (5,0) to (0,0);
            \filldraw[lightgray,fill=lightgray] (6.5,3) to (7.5,2) to (8.5,2) to (9.5,1) to (10.5,2) to (10.5,0) to (6.5,0) to (6.5,3);
            \node (1) at (5.5,-1) {$\Nn$};
            \filldraw (0,0) circle (2pt);
            \filldraw (1,1) circle (2pt);
            \filldraw (2,2) circle (2pt);
            \filldraw (3,1) circle (2pt);
            \filldraw (4,1) circle (2pt);
            \filldraw (5,2) circle (2pt);
            \filldraw (1,0) circle (2pt);
            \filldraw (2,0) circle (2pt);
            \filldraw (3,0) circle (2pt);
            \filldraw (4,0) circle (2pt);
            \filldraw (5,0) circle (2pt);
            \filldraw (6.5,3) circle (2pt);
            \filldraw (7.5,2) circle (2pt);
            \filldraw (8.5,2) circle (2pt);
            \filldraw (9.5,1) circle (2pt);
            \filldraw (10.5,2) circle (2pt);
            \filldraw (10.5,0) circle (2pt);
            \filldraw (9.5,0) circle (2pt);
            \filldraw (8.5,0) circle (2pt);
            \filldraw (7.5,0) circle (2pt);
            \filldraw (6.5,0) circle (2pt);
            \draw (0,0) to (1,1) to (2,2) to (3,1) to (4,1) to (5,2);
            \draw (0,0) to (5,0);
            \draw (6.5,0) to (10.5,0);
            \draw (6.5,3) to (7.5,2) to (8.5,2) to (9.5,1) to (10.5,2);
            \node (a) at (5.75,0) {\scalebox{.5}{$\ldots$}};
            \node (b) at (5.75,2) {\scalebox{.5}{$\ldots$}};
            \node (c) at (5.75,3) {\scalebox{.5}{$\ldots$}};
      \end{tikzpicture}} & = &   
2\times \raisebox{-.5\height}{\begin{tikzpicture}[scale=.25]
            \filldraw[lightgray,fill=lightgray] (0,0) to (1,1) to (2,2) to (3,1) to (4,1) to (5,2) to (5,0) to (0,0);
            \filldraw[lightgray,fill=lightgray] (6.5,3) to (7.5,2) to (8.5,2) to (9.5,1) to (10.5,2) to (10.5,0) to (6.5,0) to (6.5,3);
            \filldraw (-1,0) circle (2pt);
            \node (0) at (-.5,-1) {$x$};
            \node (1) at (5.5,-1) {$\Nn$};
            \filldraw (0,0) circle (2pt);
            \filldraw (1,1) circle (2pt);
            \filldraw (2,2) circle (2pt);
            \filldraw (3,1) circle (2pt);
            \filldraw (4,1) circle (2pt);
            \filldraw (5,2) circle (2pt);
            \filldraw (1,0) circle (2pt);
            \filldraw (2,0) circle (2pt);
            \filldraw (3,0) circle (2pt);
            \filldraw (4,0) circle (2pt);
            \filldraw (5,0) circle (2pt);
            \filldraw (6.5,3) circle (2pt);
            \filldraw (7.5,2) circle (2pt);
            \filldraw (8.5,2) circle (2pt);
            \filldraw (9.5,1) circle (2pt);
            \filldraw (10.5,2) circle (2pt);
            \filldraw (10.5,0) circle (2pt);
            \filldraw (9.5,0) circle (2pt);
            \filldraw (8.5,0) circle (2pt);
            \filldraw (7.5,0) circle (2pt);
            \filldraw (6.5,0) circle (2pt);
            \draw (0,0) to (1,1) to (2,2) to (3,1) to (4,1) to (5,2);
            \draw (-1,0) to (5,0);
            \draw (6.5,0) to (10.5,0);
            \draw (6.5,3) to (7.5,2) to (8.5,2) to (9.5,1) to (10.5,2);
            \node (a) at (5.75,0) {\scalebox{.5}{$\ldots$}};
            \node (b) at (5.75,2) {\scalebox{.5}{$\ldots$}};
            \node (c) at (5.75,3) {\scalebox{.5}{$\ldots$}};
      \end{tikzpicture}} & \bigsqcup &
      \raisebox{-.5\height}{
      \begin{tikzpicture}[scale=.125]
            \filldraw[lightgray,fill=lightgray] (0,0) to (1,1) to (2,2) to (3,1) to (4,1) to (5,2) to (5,0) to (0,0);
            \filldraw[lightgray,fill=lightgray] (6.5,3) to (7.5,2) to (8.5,2) to (9.5,1) to (10.5,2) to (11.5,1) to (12.5,0) to (6.5,0) to (6.5,3);
            \draw[thick] (-1,-1) to (0,0);
            \draw[thick] (12.5,0) to (13.5,-1);
            \node (1) at (5.5,-3) {$\Mm$};
            \node (2) at (19,-3) {$\mathcal N$};
            \node (3) at (-.5,-3) {$x$};
            \node (4) at (12.5,-3) {$x$};
            \filldraw (0,0) circle (2pt);
            \filldraw (1,1) circle (2pt);
            \filldraw (2,2) circle (2pt);
            \filldraw (3,1) circle (2pt);
            \filldraw (4,1) circle (2pt);
            \filldraw (5,2) circle (2pt);
            \filldraw (1,0) circle (2pt);
            \filldraw (2,0) circle (2pt);
            \filldraw (3,0) circle (2pt);
            \filldraw (4,0) circle (2pt);
            \filldraw (5,0) circle (2pt);
            \filldraw (6.5,3) circle (2pt);
            \filldraw (7.5,2) circle (2pt);
            \filldraw (8.5,2) circle (2pt);
            \filldraw (9.5,1) circle (2pt);
            \filldraw (10.5,2) circle (2pt);
            \filldraw (11.5,1) circle (2pt);
            \filldraw (12.5,0) circle (2pt);
            \filldraw (11.5,0) circle (2pt);
            \filldraw (10.5,0) circle (2pt);
            \filldraw (9.5,0) circle (2pt);
            \filldraw (8.5,0) circle (2pt);
            \filldraw (7.5,0) circle (2pt);
            \filldraw (6.5,0) circle (2pt);
            \draw (0,0) to (1,1) to (2,2) to (3,1) to (4,1) to (5,2);
            \draw (0,0) to (5,0);
            \draw (6.5,0) to (12.5,0);
            \draw (6.5,3) to (7.5,2) to (8.5,2) to (9.5,1) to (10.5,2) to (11.5,1) to (12.5,0);
            \node (a) at (5.75,0) {\scalebox{.25}{$\ldots$}};
            \node (b) at (5.75,2) {\scalebox{.25}{$\ldots$}};
            \node (c) at (5.75,3) {\scalebox{.25}{$\ldots$}};
            \filldraw (13.5,-1) circle (2pt);
            \filldraw (14.5,0) circle (2pt);
            \filldraw (15.5,1) circle (2pt);
            \filldraw (16.5,0) circle (2pt);
            \filldraw (17.5,0) circle (2pt);
            \filldraw (18.5,1) circle (2pt);
            \filldraw (14.5,-1) circle (2pt);
            \filldraw (15.5,-1) circle (2pt);
            \filldraw (16.5,-1) circle (2pt);
            \filldraw (17.5,-1) circle (2pt);
            \filldraw (18.5,-1) circle (2pt);
            \filldraw (20,2) circle (2pt);
            \filldraw (21,1) circle (2pt);
            \filldraw (22,1) circle (2pt);
            \filldraw (23,0) circle (2pt);
            \filldraw (23,-1) circle (2pt);
            \filldraw (22,-1) circle (2pt);
            \filldraw (21,-1) circle (2pt);
            \filldraw (20,-1) circle (2pt);
            \draw (13.5,-1) to (14.5,0) to (15.5,1) to (16.5,0) to (17.5,0) to (18.5,1);
            \draw (13.5,-1) to (18.5,-1);
            \draw (20,-1) to (24,-1);
            \draw (20,2) to (21,1) to (22,1) to (23,0) to (24,1);
            \node (d) at (19.25,-1) {\scalebox{.25}{$\ldots$}};
            \node (e) at (19.25,1) {\scalebox{.25}{$\ldots$}};
            \node (f) at (19.25,2) {\scalebox{.25}{$\ldots$}};
      \end{tikzpicture}
      }
       & \bigsqcup & \raisebox{-.5\height}{
      \begin{tikzpicture}[scale=.125]
            \draw[thick] (-1,-1) to (0,0);
            \draw[thick] (12.5,0) to (13.5,-1);
            \node (1) at (5.5,-3) {$\M$};
            \node (2) at (19,-3) {$\Nn$};
            \node (3) at (-.5,-3) {$x$};
            \node (4) at (12.5,-3) {$x$};
            \filldraw (0,0) circle (2pt);
            \filldraw (1,1) circle (2pt);
            \filldraw (2,2) circle (2pt);
            \filldraw (3,1) circle (2pt);
            \filldraw (4,1) circle (2pt);
            \filldraw (5,2) circle (2pt);
            \filldraw (1,0) circle (2pt);
            \filldraw (2,0) circle (2pt);
            \filldraw (3,0) circle (2pt);
            \filldraw (4,0) circle (2pt);
            \filldraw (5,0) circle (2pt);
            \filldraw (6.5,3) circle (2pt);
            \filldraw (7.5,2) circle (2pt);
            \filldraw (8.5,2) circle (2pt);
            \filldraw (9.5,1) circle (2pt);
            \filldraw (10.5,2) circle (2pt);
            \filldraw (11.5,1) circle (2pt);
            \filldraw (12.5,0) circle (2pt);
            \filldraw (11.5,0) circle (2pt);
            \filldraw (10.5,0) circle (2pt);
            \filldraw (9.5,0) circle (2pt);
            \filldraw (8.5,0) circle (2pt);
            \filldraw (7.5,0) circle (2pt);
            \filldraw (6.5,0) circle (2pt);
            \draw (0,0) to (1,1) to (2,2) to (3,1) to (4,1) to (5,2);
            \draw (0,0) to (5,0);
            \draw (6.5,0) to (12.5,0);
            \draw (6.5,3) to (7.5,2) to (8.5,2) to (9.5,1) to (10.5,2) to (11.5,1) to (12.5,0);
            \node (a) at (5.75,0) {\scalebox{.25}{$\ldots$}};
            \node (b) at (5.75,2) {\scalebox{.25}{$\ldots$}};
            \node (c) at (5.75,3) {\scalebox{.25}{$\ldots$}};
            \filldraw[lightgray,fill=lightgray] (13.5,-1) to (14.5,0) to (15.5,1) to (16.5,0) to (17.5,0) to (18.5,1) to (18.5,-1) to (13.5,-1);
            \filldraw[lightgray,fill=lightgray] (20,2) to (21,1) to (22,1) to (23,0) to (24,1) to (25,0) to (26,-1) to (20,-1) to (20,2);
            \filldraw (13.5,-1) circle (2pt);
            \filldraw (14.5,0) circle (2pt);
            \filldraw (15.5,1) circle (2pt);
            \filldraw (16.5,0) circle (2pt);
            \filldraw (17.5,0) circle (2pt);
            \filldraw (18.5,1) circle (2pt);
            \filldraw (14.5,-1) circle (2pt);
            \filldraw (15.5,-1) circle (2pt);
            \filldraw (16.5,-1) circle (2pt);
            \filldraw (17.5,-1) circle (2pt);
            \filldraw (18.5,-1) circle (2pt);
            \filldraw (20,2) circle (2pt);
            \filldraw (21,1) circle (2pt);
            \filldraw (22,1) circle (2pt);
            \filldraw (23,0) circle (2pt);
            \filldraw (24,1) circle (2pt);
            \filldraw (24,-1) circle (2pt);
            \filldraw (23,-1) circle (2pt);
            \filldraw (22,-1) circle (2pt);
            \filldraw (21,-1) circle (2pt);
            \filldraw (20,-1) circle (2pt);
            \draw (13.5,-1) to (14.5,0) to (15.5,1) to (16.5,0) to (17.5,0) to (18.5,1);
            \draw (13.5,-1) to (18.5,-1);
            \draw (20,-1) to (24,-1);
            \draw (20,2) to (21,1) to (22,1) to (23,0) to (24,1);
            \node (d) at (19.25,-1) {\scalebox{.25}{$\ldots$}};
            \node (e) at (19.25,1) {\scalebox{.25}{$\ldots$}};
            \node (f) at (19.25,2) {\scalebox{.25}{$\ldots$}};
      \end{tikzpicture}
      }\\
      & \bigsqcup &
      \raisebox{-.5\height}{
      \begin{tikzpicture}[scale=.125]
            \filldraw[lightgray,fill=lightgray] (-1,-1) to (0,0) to (12.5,0) to (13.5,-1) to (-1,-1);
            \draw[thick] (-1,-1) to (0,0);
            \draw[thick] (12.5,0) to (13.5,-1);
            \node (1) at (5.5,-4) {\scalebox{0.8}{$\partial_x(x\M)$}};
            \node (2) at (19,-3) {$\mathcal N$};
            \node (3) at (-.5,-3) {$x$};
            \node (4) at (12.5,-3) {$x$};
            \filldraw (0,0) circle (2pt);
            \filldraw (1,1) circle (2pt);
            \filldraw (2,2) circle (2pt);
            \filldraw (3,1) circle (2pt);
            \filldraw (4,1) circle (2pt);
            \filldraw (5,2) circle (2pt);
            \filldraw (1,0) circle (2pt);
            \filldraw (2,0) circle (2pt);
            \filldraw (3,0) circle (2pt);
            \filldraw (4,0) circle (2pt);
            \filldraw (5,0) circle (2pt);
            \filldraw (6.5,3) circle (2pt);
            \filldraw (7.5,2) circle (2pt);
            \filldraw (8.5,2) circle (2pt);
            \filldraw (9.5,1) circle (2pt);
            \filldraw (10.5,2) circle (2pt);
            \filldraw (11.5,1) circle (2pt);
            \filldraw (12.5,0) circle (2pt);
            \filldraw (11.5,0) circle (2pt);
            \filldraw (10.5,0) circle (2pt);
            \filldraw (9.5,0) circle (2pt);
            \filldraw (8.5,0) circle (2pt);
            \filldraw (7.5,0) circle (2pt);
            \filldraw (6.5,0) circle (2pt);
            \draw (0,0) to (1,1) to (2,2) to (3,1) to (4,1) to (5,2);
            \draw (0,0) to (5,0);
            \draw (6.5,0) to (12.5,0);
            \draw (6.5,3) to (7.5,2) to (8.5,2) to (9.5,1) to (10.5,2) to (11.5,1) to (12.5,0);
            \node (a) at (5.75,0) {\scalebox{.25}{$\ldots$}};
            \node (b) at (5.75,2) {\scalebox{.25}{$\ldots$}};
            \node (c) at (5.75,3) {\scalebox{.25}{$\ldots$}};
            \filldraw (13.5,-1) circle (2pt);
            \filldraw (14.5,0) circle (2pt);
            \filldraw (15.5,1) circle (2pt);
            \filldraw (16.5,0) circle (2pt);
            \filldraw (17.5,0) circle (2pt);
            \filldraw (18.5,1) circle (2pt);
            \filldraw (14.5,-1) circle (2pt);
            \filldraw (15.5,-1) circle (2pt);
            \filldraw (16.5,-1) circle (2pt);
            \filldraw (17.5,-1) circle (2pt);
            \filldraw (18.5,-1) circle (2pt);
            \filldraw (20,2) circle (2pt);
            \filldraw (21,1) circle (2pt);
            \filldraw (22,1) circle (2pt);
            \filldraw (23,0) circle (2pt);
            \filldraw (24,1) circle (2pt);
            \filldraw (24,-1) circle (2pt);
            \filldraw (23,-1) circle (2pt);
            \filldraw (22,-1) circle (2pt);
            \filldraw (21,-1) circle (2pt);
            \filldraw (20,-1) circle (2pt);
            \draw (13.5,-1) to (14.5,0) to (15.5,1) to (16.5,0) to (17.5,0) to (18.5,1);
            \draw (13.5,-1) to (18.5,-1);
            \draw (20,-1) to (24,-1);
            \draw (20,2) to (21,1) to (22,1) to (23,0) to (24,1);
            \node (d) at (19.25,-1) {\scalebox{.25}{$\ldots$}};
            \node (e) at (19.25,1) {\scalebox{.25}{$\ldots$}};
            \node (f) at (19.25,2) {\scalebox{.25}{$\ldots$}};
      \end{tikzpicture}
      } & \bigsqcup &  \raisebox{-.5\height}{\begin{tikzpicture}[scale=.25]
             \filldraw[lightgray,fill=lightgray] (1,1) to (2,2) to (3,1) to (4,1) to (5,2) to (5,1) to (1,1);
            \filldraw[lightgray,fill=lightgray] (6.5,3) to (7.5,2) to (8.5,2) to (9.5,1) to (10.5,2) to (10.5,1) to (6.5,1) to (6.5,3);
            \node (1) at (5.5,-1) {$\Nn$};
            \node (3) at (.5,-1) {$x$};
            \filldraw (0,0) circle (2pt);
            \filldraw (1,1) circle (2pt);
            \filldraw (2,2) circle (2pt);
            \filldraw (3,1) circle (2pt);
            \filldraw (4,1) circle (2pt);
            \filldraw (5,2) circle (2pt);
            \filldraw (1,1) circle (2pt);
            \filldraw (2,1) circle (2pt);
            \filldraw (3,1) circle (2pt);
            \filldraw (4,1) circle (2pt);
            \filldraw (5,1) circle (2pt);
            \filldraw (6.5,3) circle (2pt);
            \filldraw (7.5,2) circle (2pt);
            \filldraw (8.5,2) circle (2pt);
            \filldraw (9.5,1) circle (2pt);
            \filldraw (10.5,2) circle (2pt);
            \filldraw (10.5,1) circle (2pt);
            \filldraw (9.5,1) circle (2pt);
            \filldraw (8.5,1) circle (2pt);
            \filldraw (7.5,1) circle (2pt);
            \filldraw (6.5,1) circle (2pt);
            \draw (0,0) to (1,1) to (2,2) to (3,1) to (4,1) to (5,2);
            \draw (1,1) to (5,1);
            \draw (6.5,1) to (10.5,1);
            \draw (6.5,3) to (7.5,2) to (8.5,2) to (9.5,1) to (10.5,2);
            \node (a) at (5.75,0) {\scalebox{.5}{$\ldots$}};
            \node (b) at (5.75,2) {\scalebox{.5}{$\ldots$}};
            \node (c) at (5.75,3) {\scalebox{.5}{$\ldots$}};
      \end{tikzpicture}} & \bigsqcup &  \raisebox{-.5\height}{\begin{tikzpicture}[scale=.25]
             \filldraw[lightgray,fill=lightgray] (0,0) to (1,1) to (10.5,1) to (10.5,0) to (0,0);
            \node (1) at (5.5,-1) {$\partial_x(x\mathcal N)$};
            \node (3) at (.5,-1) {$x$};
            \filldraw (1,1) circle (2pt);
            \filldraw (2,2) circle (2pt);
            \filldraw (3,1) circle (2pt);
            \filldraw (4,1) circle (2pt);
            \filldraw (5,2) circle (2pt);
            \filldraw (1,1) circle (2pt);
            \filldraw (2,1) circle (2pt);
            \filldraw (3,1) circle (2pt);
            \filldraw (4,1) circle (2pt);
            \filldraw (5,1) circle (2pt);
            \filldraw (6.5,3) circle (2pt);
            \filldraw (7.5,2) circle (2pt);
            \filldraw (8.5,2) circle (2pt);
            \filldraw (9.5,1) circle (2pt);
            \filldraw (10.5,2) circle (2pt);
            \filldraw (10.5,1) circle (2pt);
            \filldraw (9.5,1) circle (2pt);
            \filldraw (8.5,1) circle (2pt);
            \filldraw (7.5,1) circle (2pt);
            \filldraw (6.5,1) circle (2pt);
            \draw (0,0) to (1,1) to (2,2) to (3,1) to (4,1) to (5,2);
            \draw (1,1) to (5,1);
            \draw (6.5,1) to (10.5,1);
            \draw (6.5,3) to (7.5,2) to (8.5,2) to (9.5,1) to (10.5,2);
            \node (a) at (5.75,0) {\scalebox{.5}{$\ldots$}};
            \node (b) at (5.75,2) {\scalebox{.5}{$\ldots$}};
            \node (c) at (5.75,3) {\scalebox{.5}{$\ldots$}};
      \end{tikzpicture}}
\end{array}$}

\noindent\rule{\textwidth}{1pt}

\scalebox{.8}{
$\begin{array}{ccccccccc}
 \raisebox{-.5\height}{\begin{tikzpicture}[scale=.25]
            \node (1) at (5.5,2) {$\mathcal V$};
            \filldraw (0,0) circle (2pt);
            \filldraw (1,1) circle (2pt);
            \filldraw (2,1) circle (2pt);
            \filldraw (3,0) circle (2pt);
            \filldraw (4,-1) circle (2pt);
            \filldraw (5,-2) circle (2pt);
            \filldraw (1,0) circle (2pt);
            \filldraw (2,0) circle (2pt);
            \filldraw (3,0) circle (2pt);
            \filldraw (4,0) circle (2pt);
            \filldraw (5,0) circle (2pt);
            \filldraw (6.5,-3) circle (2pt);
            \filldraw (7.5,-2) circle (2pt);
            \filldraw (8.5,-2) circle (2pt);
            \filldraw (9.5,-3) circle (2pt);
            \filldraw (10.5,-3) circle (2pt);
            \filldraw (11.5,-2) circle (2pt);
            \filldraw (12.5,-2) circle (2pt);
            \filldraw (11.5,0) circle (2pt);
            \filldraw (12.5,0) circle (2pt);
            \filldraw (10.5,0) circle (2pt);
            \filldraw (9.5,0) circle (2pt);
            \filldraw (8.5,0) circle (2pt);
            \filldraw (7.5,0) circle (2pt);
            \filldraw (6.5,0) circle (2pt);
            \draw (0,0) to (1,1) to (2,1) to (3,0) to (4,-1) to (5,-2);
            \draw (0,0) to (5,0);
            \draw (6.5,0) to (12.5,0);
            \draw (6.5,-3) to (7.5,-2) to (8.5,-2) to (9.5,-3) to (10.5,-3) to (11.5,-2) to (12.5,-2);
            \node (a) at (5.75,0) {\scalebox{.5}{$\ldots$}};
            \node (b) at (5.75,-2) {\scalebox{.5}{$\ldots$}};
            \node (c) at (5.75,-3) {\scalebox{.5}{$\ldots$}};
      \end{tikzpicture}} & = & \raisebox{-.5\height}{\begin{tikzpicture} \node (0) at (0,0) {$\bullet$}; \node (1) at (0,-.5) {$1$}; \end{tikzpicture}} & \bigsqcup &  
 2\times \Bigg(
      \raisebox{-.5\height}{\begin{tikzpicture}[scale=.25]
            \node (0) at (-.5,-1) {$x$};
            \node (1) at (5.5,-1) {$\mathcal N$};
            \filldraw (0,0) circle (2pt);
            \filldraw (-1,0) circle (2pt);
            \filldraw (0,0) circle (2pt);
            \filldraw (1,1) circle (2pt);
            \filldraw (2,2) circle (2pt);
            \filldraw (3,1) circle (2pt);
            \filldraw (4,1) circle (2pt);
            \filldraw (5,2) circle (2pt);
            \filldraw (1,0) circle (2pt);
            \filldraw (2,0) circle (2pt);
            \filldraw (3,0) circle (2pt);
            \filldraw (4,0) circle (2pt);
            \filldraw (5,0) circle (2pt);
            \filldraw (6.5,3) circle (2pt);
            \filldraw (7.5,2) circle (2pt);
            \filldraw (8.5,2) circle (2pt);
            \filldraw (9.5,1) circle (2pt);
            \filldraw (10.5,2) circle (2pt);
            \filldraw (10.5,0) circle (2pt);
            \filldraw (9.5,0) circle (2pt);
            \filldraw (8.5,0) circle (2pt);
            \filldraw (7.5,0) circle (2pt);
            \filldraw (6.5,0) circle (2pt);
            \draw[very thick] (-1,0) to (0,0);
            \draw (0,0) to (1,1) to (2,2) to (3,1) to (4,1) to (5,2);
            \draw (0,0) to (5,0);
            \draw (6.5,0) to (10.5,0);
            \draw (6.5,3) to (7.5,2) to (8.5,2) to (9.5,1) to (10.5,2);
            \node (a) at (5.75,0) {\scalebox{.5}{$\ldots$}};
            \node (b) at (5.75,2) {\scalebox{.5}{$\ldots$}};
            \node (c) at (5.75,3) {\scalebox{.5}{$\ldots$}};
      \end{tikzpicture}} & & \bigsqcup &
      \raisebox{-.5\height}{\begin{tikzpicture}[scale=.25]
            \node (1) at (6.5,2) {$\mathcal V$};
            \node (0) at (-.5,-1) {$x$};
            \filldraw (-1,0) circle (2pt);
            \filldraw (0,0) circle (2pt);
            \filldraw (1,1) circle (2pt);
            \filldraw (2,1) circle (2pt);
            \filldraw (3,0) circle (2pt);
            \filldraw (4,-1) circle (2pt);
            \filldraw (5,-2) circle (2pt);
            \filldraw (1,0) circle (2pt);
            \filldraw (2,0) circle (2pt);
            \filldraw (3,0) circle (2pt);
            \filldraw (4,0) circle (2pt);
            \filldraw (5,0) circle (2pt);
            \filldraw (6.5,-3) circle (2pt);
            \filldraw (7.5,-2) circle (2pt);
            \filldraw (8.5,-2) circle (2pt);
            \filldraw (9.5,-3) circle (2pt);
            \filldraw (10.5,-3) circle (2pt);
            \filldraw (11.5,-2) circle (2pt);
            \filldraw (12.5,-2) circle (2pt);
            \filldraw (11.5,0) circle (2pt);
            \filldraw (12.5,0) circle (2pt);
            \filldraw (10.5,0) circle (2pt);
            \filldraw (9.5,0) circle (2pt);
            \filldraw (8.5,0) circle (2pt);
            \filldraw (7.5,0) circle (2pt);
            \filldraw (6.5,0) circle (2pt);
            \draw[very thick] (-1,0) to (0,0);
            \draw (0,0) to (1,1) to (2,1) to (3,0) to (4,-1) to (5,-2);
            \draw (-1,0) to (5,0);
            \draw (6.5,0) to (12.5,0);
            \draw (6.5,-3) to (7.5,-2) to (8.5,-2) to (9.5,-3) to (10.5,-3) to (11.5,-2) to (12.5,-2);
            \node (a) at (5.75,0) {\scalebox{.5}{$\ldots$}};
            \node (b) at (5.75,-2) {\scalebox{.5}{$\ldots$}};
            \node (c) at (5.75,-3) {\scalebox{.5}{$\ldots$}};
      \end{tikzpicture}} \\ & & & & & & \bigsqcup &
       \raisebox{-.5\height}{
      \begin{tikzpicture}[scale=.125]
            \draw[thick] (-1,-1) to (0,0);
            \draw[thick] (12.5,0) to (13.5,-1);
            \node (1) at (5.5,-3) {$\M$};
            \node (2) at (21,1) {$\mathcal V$};
            \node (3) at (-.5,-3) {$x$};
            \node (4) at (12.5,-3) {$x$};
            \filldraw (0,0) circle (2pt);
            \filldraw (1,1) circle (2pt);
            \filldraw (2,2) circle (2pt);
            \filldraw (3,1) circle (2pt);
            \filldraw (4,1) circle (2pt);
            \filldraw (5,2) circle (2pt);
            \filldraw (1,0) circle (2pt);
            \filldraw (2,0) circle (2pt);
            \filldraw (3,0) circle (2pt);
            \filldraw (4,0) circle (2pt);
            \filldraw (5,0) circle (2pt);
            \filldraw (6.5,3) circle (2pt);
            \filldraw (7.5,2) circle (2pt);
            \filldraw (8.5,2) circle (2pt);
            \filldraw (9.5,1) circle (2pt);
            \filldraw (10.5,2) circle (2pt);
            \filldraw (11.5,1) circle (2pt);
            \filldraw (12.5,0) circle (2pt);
            \filldraw (11.5,0) circle (2pt);
            \filldraw (10.5,0) circle (2pt);
            \filldraw (9.5,0) circle (2pt);
            \filldraw (8.5,0) circle (2pt);
            \filldraw (7.5,0) circle (2pt);
            \filldraw (6.5,0) circle (2pt);
            \draw (0,0) to (1,1) to (2,2) to (3,1) to (4,1) to (5,2);
            \draw (0,0) to (5,0);
            \draw (6.5,0) to (12.5,0);
            \draw (6.5,3) to (7.5,2) to (8.5,2) to (9.5,1) to (10.5,2) to (11.5,1) to (12.5,0);
            \node (a) at (5.75,0) {\scalebox{.5}{$\ldots$}};
            \node (b) at (5.75,2) {\scalebox{.5}{$\ldots$}};
            \node (c) at (5.75,3) {\scalebox{.5}{$\ldots$}};
            \filldraw (13.5,-1) circle (2pt);
            \filldraw (14.5,0) circle (2pt);
            \filldraw (15.5,0) circle (2pt);
            \filldraw (16.5,-1) circle (2pt);
            \filldraw (17.5,-2) circle (2pt);
            \filldraw (18.5,-3) circle (2pt);
            \filldraw (14.5,-1) circle (2pt);
            \filldraw (15.5,-1) circle (2pt);
            \filldraw (16.5,-1) circle (2pt);
            \filldraw (17.5,-1) circle (2pt);
            \filldraw (18.5,-1) circle (2pt);
            \filldraw (20,-4) circle (2pt);
            \filldraw (21,-3) circle (2pt);
            \filldraw (22,-3) circle (2pt);
            \filldraw (23,-4) circle (2pt);
            \filldraw (24,-4) circle (2pt);
            \filldraw (25,-3) circle (2pt);
            \filldraw (26,-3) circle (2pt);
            \filldraw (26,-1) circle (2pt);
            \filldraw (25,-1) circle (2pt);
            \filldraw (23,-1) circle (2pt);
            \filldraw (22,-1) circle (2pt);
            \filldraw (21,-1) circle (2pt);
            \filldraw (20,-1) circle (2pt);
            \draw (13.5,-1) to (14.5,0) to (15.5,0) to (16.5,-1) to (17.5,-2) to (18.5,-3);
            \draw (13.5,-1) to (18.5,-1);
            \draw (20,-1) to (26,-1);
            \draw (20,-4) to (21,-3) to (22,-3) to (23,-4) to (24,-4) to (25,-3) to (26,-3);
            \node (a) at (19.25,-1) {\scalebox{.5}{$\ldots$}};
            \node (b) at (19.25,-3) {\scalebox{.5}{$\ldots$}};
            \node (c) at (19.25,-4) {\scalebox{.5}{$\ldots$}};
      \end{tikzpicture}
      }\Bigg )
\end{array}$}

\noindent\rule{\textwidth}{1pt}

\scalebox{.8}{
$\begin{array}{ccccccccc}
 \raisebox{-.5\height}{\begin{tikzpicture}[scale=.25]
 \filldraw[lightgray,fill=lightgray] (0,0) to (1,1) to (2,1) to (3,0) to (0,0);
            \filldraw[lightgray,fill=lightgray] (3,0) to (5,0) to (5,-2) to (3,0);
            \filldraw[lightgray,fill=lightgray] (6.5,0) to (12.5,0) to (12.5,-2) to (11.5,-2) to (10.5,-3) to (9.5,-3) to (8.5,-2) to (7.5,-2) to (6.5,-3) to (6.5,0);
            \node (1) at (5.5,2) {$\Vv$};
            \filldraw (0,0) circle (2pt);
            \filldraw (1,1) circle (2pt);
            \filldraw (2,1) circle (2pt);
            \filldraw (3,0) circle (2pt);
            \filldraw (4,-1) circle (2pt);
            \filldraw (5,-2) circle (2pt);
            \filldraw (1,0) circle (2pt);
            \filldraw (2,0) circle (2pt);
            \filldraw (3,0) circle (2pt);
            \filldraw (4,0) circle (2pt);
            \filldraw (5,0) circle (2pt);
            \filldraw (6.5,-3) circle (2pt);
            \filldraw (7.5,-2) circle (2pt);
            \filldraw (8.5,-2) circle (2pt);
            \filldraw (9.5,-3) circle (2pt);
            \filldraw (10.5,-3) circle (2pt);
            \filldraw (11.5,-2) circle (2pt);
            \filldraw (12.5,-2) circle (2pt);
            \filldraw (12.5,0) circle (2pt);
            \filldraw (11.5,0) circle (2pt);
            \filldraw (10.5,0) circle (2pt);
            \filldraw (9.5,0) circle (2pt);
            \filldraw (8.5,0) circle (2pt);
            \filldraw (7.5,0) circle (2pt);
            \filldraw (6.5,0) circle (2pt);
            \draw (0,0) to (1,1) to (2,1) to (3,0) to (4,-1) to (5,-2);
            \draw (0,0) to (5,0);
            \draw (6.5,0) to (12.5,0);
            \draw (6.5,-3) to (7.5,-2) to (8.5,-2) to (9.5,-3) to (10.5,-3) to (11.5,-2) to (12.5,-2);
            \node (a) at (5.75,0) {\scalebox{.5}{$\ldots$}};
            \node (b) at (5.75,-2) {\scalebox{.5}{$\ldots$}};
            \node (c) at (5.75,-3) {\scalebox{.5}{$\ldots$}};
      \end{tikzpicture}} & = 2\times \Bigg( & 
\raisebox{-.5\height}{\begin{tikzpicture}[scale=.25]
 \filldraw[lightgray,fill=lightgray] (0,0) to (1,1) to (2,1) to (3,0) to (0,0);
            \filldraw[lightgray,fill=lightgray] (3,0) to (5,0) to (5,-2) to (3,0);
            \filldraw[lightgray,fill=lightgray] (6.5,0) to (12.5,0) to (12.5,-2) to (11.5,-2) to (10.5,-3) to (9.5,-3) to (8.5,-2) to (7.5,-2) to (6.5,-3) to (6.5,0);
            \filldraw (-1,0) circle (2pt);
            \node (0) at (-.5,-1) {$x$};
            \draw[thick] (-1,0) to (0,0);
            \node (1) at (5.5,2) {$\Vv$};
            \filldraw (0,0) circle (2pt);
            \filldraw (1,1) circle (2pt);
            \filldraw (2,1) circle (2pt);
            \filldraw (3,0) circle (2pt);
            \filldraw (4,-1) circle (2pt);
            \filldraw (5,-2) circle (2pt);
            \filldraw (1,0) circle (2pt);
            \filldraw (2,0) circle (2pt);
            \filldraw (3,0) circle (2pt);
            \filldraw (4,0) circle (2pt);
            \filldraw (5,0) circle (2pt);
            \filldraw (6.5,-3) circle (2pt);
            \filldraw (7.5,-2) circle (2pt);
            \filldraw (8.5,-2) circle (2pt);
            \filldraw (9.5,-3) circle (2pt);
            \filldraw (10.5,-3) circle (2pt);
            \filldraw (11.5,-2) circle (2pt);
            \filldraw (12.5,-2) circle (2pt);
            \filldraw (12.5,0) circle (2pt);
            \filldraw (11.5,0) circle (2pt);
            \filldraw (10.5,0) circle (2pt);
            \filldraw (9.5,0) circle (2pt);
            \filldraw (8.5,0) circle (2pt);
            \filldraw (7.5,0) circle (2pt);
            \filldraw (6.5,0) circle (2pt);
            \draw (0,0) to (1,1) to (2,1) to (3,0) to (4,-1) to (5,-2);
            \draw (0,0) to (5,0);
            \draw (6.5,0) to (12.5,0);
            \draw (6.5,-3) to (7.5,-2) to (8.5,-2) to (9.5,-3) to (10.5,-3) to (11.5,-2) to (12.5,-2);
            \node (a) at (5.75,0) {\scalebox{.5}{$\ldots$}};
            \node (b) at (5.75,-2) {\scalebox{.5}{$\ldots$}};
            \node (c) at (5.75,-3) {\scalebox{.5}{$\ldots$}};
      \end{tikzpicture}} & \bigsqcup & \raisebox{-.5\height}{
      \begin{tikzpicture}[scale=.125]
      \filldraw[lightgray,fill=lightgray] (0,0) to (1,1) to (2,2) to (3,1) to (4,1) to (5,2) to (5,0) to (0,0);
            \filldraw[lightgray,fill=lightgray] (6.5,3) to (7.5,2) to (8.5,2) to (9.5,1) to (10.5,2) to (11.5,1) to (12.5,0) to (6.5,0) to (6.5,3);
            \draw[thick] (-1,-1) to (0,0);
            \draw[thick] (12.5,0) to (13.5,-1);
            \node (1) at (5.5,-3) {$\Mm$};
            \node (2) at (19,1) {$\mathcal V$};
            \node (3) at (-.5,-3) {$x$};
            \node (4) at (12.5,-3) {$x$};
            \filldraw (0,0) circle (2pt);
            \filldraw (1,1) circle (2pt);
            \filldraw (2,2) circle (2pt);
            \filldraw (3,1) circle (2pt);
            \filldraw (4,1) circle (2pt);
            \filldraw (5,2) circle (2pt);
            \filldraw (1,0) circle (2pt);
            \filldraw (2,0) circle (2pt);
            \filldraw (3,0) circle (2pt);
            \filldraw (4,0) circle (2pt);
            \filldraw (5,0) circle (2pt);
            \filldraw (6.5,3) circle (2pt);
            \filldraw (7.5,2) circle (2pt);
            \filldraw (8.5,2) circle (2pt);
            \filldraw (9.5,1) circle (2pt);
            \filldraw (10.5,2) circle (2pt);
            \filldraw (11.5,1) circle (2pt);
            \filldraw (12.5,0) circle (2pt);
            \filldraw (11.5,0) circle (2pt);
            \filldraw (10.5,0) circle (2pt);
            \filldraw (9.5,0) circle (2pt);
            \filldraw (8.5,0) circle (2pt);
            \filldraw (7.5,0) circle (2pt);
            \filldraw (6.5,0) circle (2pt);
            \draw (0,0) to (1,1) to (2,2) to (3,1) to (4,1) to (5,2);
            \draw (0,0) to (5,0);
            \draw (6.5,0) to (12.5,0);
            \draw (6.5,3) to (7.5,2) to (8.5,2) to (9.5,1) to (10.5,2) to (11.5,1) to (12.5,0);
            \node (a) at (5.75,0) {\scalebox{.5}{$\ldots$}};
            \node (b) at (5.75,2) {\scalebox{.5}{$\ldots$}};
            \node (c) at (5.75,3) {\scalebox{.5}{$\ldots$}};
            \filldraw (13.5,-1) circle (2pt);
            \filldraw (14.5,0) circle (2pt);
            \filldraw (15.5,0) circle (2pt);
            \filldraw (16.5,-1) circle (2pt);
            \filldraw (17.5,-2) circle (2pt);
            \filldraw (18.5,-3) circle (2pt);
            \filldraw (14.5,-1) circle (2pt);
            \filldraw (15.5,-1) circle (2pt);
            \filldraw (16.5,-1) circle (2pt);
            \filldraw (17.5,-1) circle (2pt);
            \filldraw (18.5,-1) circle (2pt);
            \filldraw (20,-4) circle (2pt);
            \filldraw (21,-3) circle (2pt);
            \filldraw (22,-3) circle (2pt);
            \filldraw (23,-4) circle (2pt);
            \filldraw (24,-4) circle (2pt);
            \filldraw (25,-3) circle (2pt);
            \filldraw (26,-3) circle (2pt);
            \filldraw (25,-1) circle (2pt);
            \filldraw (24,-1) circle (2pt);
            \filldraw (23,-1) circle (2pt);
            \filldraw (22,-1) circle (2pt);
            \filldraw (21,-1) circle (2pt);
            \filldraw (20,-1) circle (2pt);
            \draw (13.5,-1) to (14.5,0) to (15.5,0) to (16.5,-1) to (17.5,-2) to (18.5,-3);
            \draw (13.5,-1) to (18.5,-1);
            \draw (20,-1) to (26,-1);
            \draw (20,-4) to (21,-3) to (22,-3) to (23,-4) to (24,-4) to (25,-3) to (26,-3);
            \node (a) at (19.25,-1) {\scalebox{.5}{$\ldots$}};
            \node (b) at (19.25,-3) {\scalebox{.5}{$\ldots$}};
            \node (c) at (19.25,-4) {\scalebox{.5}{$\ldots$}};
      \end{tikzpicture}
      } & \bigsqcup & \raisebox{-.5\height}{
      \begin{tikzpicture}[scale=.125]
            \filldraw[lightgray,fill=lightgray] (13.5,-1) to (14.5,0) to (15.5,0) to (16.5,-1) to (13.5,-1);
            \filldraw[lightgray,fill=lightgray] (16.5,-1) to (18.5,-1) to (18.5,-3) to (16.5,-1);
            \filldraw[lightgray,fill=lightgray] (20,-1) to (26,-1) to (26,-3) to (25,-3) to (24,-4) to (23,-4) to (22,-3) to (21,-3) to (20,-4) to (20,-1);
            \draw[thick] (-1,-1) to (0,0);
            \draw[thick] (12.5,0) to (13.5,-1);
            \node (1) at (5.5,-3) {$\M$};
            \node (2) at (21,1) {$\Vv$};
            \node (3) at (-.5,-3) {$x$};
            \node (4) at (12.5,-3) {$x$};
            \filldraw (0,0) circle (2pt);
            \filldraw (1,1) circle (2pt);
            \filldraw (2,2) circle (2pt);
            \filldraw (3,1) circle (2pt);
            \filldraw (4,1) circle (2pt);
            \filldraw (5,2) circle (2pt);
            \filldraw (1,0) circle (2pt);
            \filldraw (2,0) circle (2pt);
            \filldraw (3,0) circle (2pt);
            \filldraw (4,0) circle (2pt);
            \filldraw (5,0) circle (2pt);
            \filldraw (6.5,3) circle (2pt);
            \filldraw (7.5,2) circle (2pt);
            \filldraw (8.5,2) circle (2pt);
            \filldraw (9.5,1) circle (2pt);
            \filldraw (10.5,2) circle (2pt);
            \filldraw (11.5,1) circle (2pt);
            \filldraw (12.5,0) circle (2pt);
            \filldraw (11.5,0) circle (2pt);
            \filldraw (10.5,0) circle (2pt);
            \filldraw (9.5,0) circle (2pt);
            \filldraw (8.5,0) circle (2pt);
            \filldraw (7.5,0) circle (2pt);
            \filldraw (6.5,0) circle (2pt);
            \draw (0,0) to (1,1) to (2,2) to (3,1) to (4,1) to (5,2);
            \draw (0,0) to (5,0);
            \draw (6.5,0) to (12.5,0);
            \draw (6.5,3) to (7.5,2) to (8.5,2) to (9.5,1) to (10.5,2) to (11.5,1) to (12.5,0);
            \node (a) at (5.75,0) {\scalebox{.5}{$\ldots$}};
            \node (b) at (5.75,2) {\scalebox{.5}{$\ldots$}};
            \node (c) at (5.75,3) {\scalebox{.5}{$\ldots$}};
            \filldraw (13.5,-1) circle (2pt);
            \filldraw (14.5,0) circle (2pt);
            \filldraw (15.5,0) circle (2pt);
            \filldraw (16.5,-1) circle (2pt);
            \filldraw (17.5,-2) circle (2pt);
            \filldraw (18.5,-3) circle (2pt);
            \filldraw (14.5,-1) circle (2pt);
            \filldraw (15.5,-1) circle (2pt);
            \filldraw (16.5,-1) circle (2pt);
            \filldraw (17.5,-1) circle (2pt);
            \filldraw (18.5,-1) circle (2pt);
            \filldraw (20,-4) circle (2pt);
            \filldraw (21,-3) circle (2pt);
            \filldraw (22,-3) circle (2pt);
            \filldraw (23,-4) circle (2pt);
            \filldraw (24,-4) circle (2pt);
            \filldraw (25,-3) circle (2pt);
            \filldraw (26,-3) circle (2pt);
            \filldraw (26,-1) circle (2pt);
            \filldraw (25,-1) circle (2pt);
            \filldraw (23,-1) circle (2pt);
            \filldraw (22,-1) circle (2pt);
            \filldraw (21,-1) circle (2pt);
            \filldraw (20,-1) circle (2pt);
            \draw (13.5,-1) to (14.5,0) to (15.5,0) to (16.5,-1) to (17.5,-2) to (18.5,-3);
            \draw (13.5,-1) to (18.5,-1);
            \draw (20,-1) to (26,-1);
            \draw (20,-4) to (21,-3) to (22,-3) to (23,-4) to (24,-4) to (25,-3) to (26,-3);
            \node (a) at (19.25,-1) {\scalebox{.5}{$\ldots$}};
            \node (b) at (19.25,-3) {\scalebox{.5}{$\ldots$}};
            \node (c) at (19.25,-4) {\scalebox{.5}{$\ldots$}};
      \end{tikzpicture}} \\ 
      & \bigsqcup & \raisebox{-.5\height}{
      \begin{tikzpicture}[scale=.125]
            \filldraw[lightgray,fill=lightgray] (-1,-1) to (0,0) to (12.5,0) to (13.5,-1) to (-1,-1);
            \draw[thick] (-1,-1) to (0,0);
            \draw[thick] (12.5,0) to (13.5,-1);
            \node (1) at (5.5,-4) {\scalebox{0.8}{$\partial_x(x\M)$}};
            \node (3) at (-.5,-3) {$x$};
            \node (4) at (12.5,-3) {$x$};
            \node (2) at (19,1) {$\mathcal V$};
            \filldraw (0,0) circle (2pt);
            \filldraw (1,1) circle (2pt);
            \filldraw (2,2) circle (2pt);
            \filldraw (3,1) circle (2pt);
            \filldraw (4,1) circle (2pt);
            \filldraw (5,2) circle (2pt);
            \filldraw (1,0) circle (2pt);
            \filldraw (2,0) circle (2pt);
            \filldraw (3,0) circle (2pt);
            \filldraw (4,0) circle (2pt);
            \filldraw (5,0) circle (2pt);
            \filldraw (6.5,3) circle (2pt);
            \filldraw (7.5,2) circle (2pt);
            \filldraw (8.5,2) circle (2pt);
            \filldraw (9.5,1) circle (2pt);
            \filldraw (10.5,2) circle (2pt);
            \filldraw (11.5,1) circle (2pt);
            \filldraw (12.5,0) circle (2pt);
            \filldraw (11.5,0) circle (2pt);
            \filldraw (10.5,0) circle (2pt);
            \filldraw (9.5,0) circle (2pt);
            \filldraw (8.5,0) circle (2pt);
            \filldraw (7.5,0) circle (2pt);
            \filldraw (6.5,0) circle (2pt);
            \draw (0,0) to (1,1) to (2,2) to (3,1) to (4,1) to (5,2);
            \draw (0,0) to (5,0);
            \draw (6.5,0) to (12.5,0);
            \draw (6.5,3) to (7.5,2) to (8.5,2) to (9.5,1) to (10.5,2) to (11.5,1) to (12.5,0);
            \node (a) at (5.75,0) {\scalebox{.5}{$\ldots$}};
            \node (b) at (5.75,2) {\scalebox{.5}{$\ldots$}};
            \node (c) at (5.75,3) {\scalebox{.5}{$\ldots$}};
            \filldraw (13.5,-1) circle (2pt);
            \filldraw (14.5,0) circle (2pt);
            \filldraw (15.5,0) circle (2pt);
            \filldraw (16.5,-1) circle (2pt);
            \filldraw (17.5,-2) circle (2pt);
            \filldraw (18.5,-3) circle (2pt);
            \filldraw (14.5,-1) circle (2pt);
            \filldraw (15.5,-1) circle (2pt);
            \filldraw (16.5,-1) circle (2pt);
            \filldraw (17.5,-1) circle (2pt);
            \filldraw (18.5,-1) circle (2pt);
            \filldraw (20,-4) circle (2pt);
            \filldraw (21,-3) circle (2pt);
            \filldraw (22,-3) circle (2pt);
            \filldraw (23,-4) circle (2pt);
            \filldraw (24,-4) circle (2pt);
            \filldraw (25,-3) circle (2pt);
            \filldraw (26,-3) circle (2pt);
            \filldraw (24,-1) circle (2pt);
            \filldraw (25,-1) circle (2pt);
            \filldraw (23,-1) circle (2pt);
            \filldraw (22,-1) circle (2pt);
            \filldraw (21,-1) circle (2pt);
            \filldraw (20,-1) circle (2pt);
            \draw (13.5,-1) to (14.5,0) to (15.5,0) to (16.5,-1) to (17.5,-2) to (18.5,-3);
            \draw (13.5,-1) to (18.5,-1);
            \draw (20,-1) to (26,-1);
            \draw (20,-4) to (21,-3) to (22,-3) to (23,-4) to (24,-4) to (25,-3) to (26,-3);
            \node (a) at (19.25,-1) {\scalebox{.5}{$\ldots$}};
            \node (b) at (19.25,-3) {\scalebox{.5}{$\ldots$}};
            \node (c) at (19.25,-4) {\scalebox{.5}{$\ldots$}};
      \end{tikzpicture}} & \bigsqcup & \raisebox{-.5\height}{\begin{tikzpicture}[scale=.25]
             \filldraw[lightgray,fill=lightgray] (1,1) to (2,2) to (3,1) to (4,1) to (5,2) to (5,1) to (1,1);
            \filldraw[lightgray,fill=lightgray] (6.5,3) to (7.5,2) to (8.5,2) to (9.5,1) to (10.5,2) to (10.5,1) to (6.5,1) to (6.5,3);
            \node (1) at (5.5,-1) {$\Nn$};
            \node (3) at (.5,-1) {$x$};
            \filldraw (0,0) circle (2pt);
            \filldraw (1,1) circle (2pt);
            \filldraw (2,2) circle (2pt);
            \filldraw (3,1) circle (2pt);
            \filldraw (4,1) circle (2pt);
            \filldraw (5,2) circle (2pt);
            \filldraw (1,1) circle (2pt);
            \filldraw (2,1) circle (2pt);
            \filldraw (3,1) circle (2pt);
            \filldraw (4,1) circle (2pt);
            \filldraw (5,1) circle (2pt);
            \filldraw (6.5,3) circle (2pt);
            \filldraw (7.5,2) circle (2pt);
            \filldraw (8.5,2) circle (2pt);
            \filldraw (9.5,1) circle (2pt);
            \filldraw (10.5,2) circle (2pt);
            \filldraw (10.5,1) circle (2pt);
            \filldraw (9.5,1) circle (2pt);
            \filldraw (8.5,1) circle (2pt);
            \filldraw (7.5,1) circle (2pt);
            \filldraw (6.5,1) circle (2pt);
            \draw (0,0) to (1,1) to (2,2) to (3,1) to (4,1) to (5,2);
            \draw (1,1) to (5,1);
            \draw (6.5,1) to (10.5,1);
            \draw (6.5,3) to (7.5,2) to (8.5,2) to (9.5,1) to (10.5,2);
            \node (a) at (5.75,0) {\scalebox{.5}{$\ldots$}};
            \node (b) at (5.75,2) {\scalebox{.5}{$\ldots$}};
            \node (c) at (5.75,3) {\scalebox{.5}{$\ldots$}};
      \end{tikzpicture}} & \bigsqcup &  \raisebox{-.5\height}{\begin{tikzpicture}[scale=.25]
             \filldraw[lightgray,fill=lightgray] (0,0) to (1,1) to (10.5,1) to (10.5,0) to (0,0);
            \node (1) at (5.5,-1) {$\partial_x(x\mathcal N)$};
            \node (3) at (.5,-1) {$x$};
            \filldraw (1,1) circle (2pt);
            \filldraw (2,2) circle (2pt);
            \filldraw (3,1) circle (2pt);
            \filldraw (4,1) circle (2pt);
            \filldraw (5,2) circle (2pt);
            \filldraw (1,1) circle (2pt);
            \filldraw (2,1) circle (2pt);
            \filldraw (3,1) circle (2pt);
            \filldraw (4,1) circle (2pt);
            \filldraw (5,1) circle (2pt);
            \filldraw (6.5,3) circle (2pt);
            \filldraw (7.5,2) circle (2pt);
            \filldraw (8.5,2) circle (2pt);
            \filldraw (9.5,1) circle (2pt);
            \filldraw (10.5,2) circle (2pt);
            \filldraw (10.5,1) circle (2pt);
            \filldraw (9.5,1) circle (2pt);
            \filldraw (8.5,1) circle (2pt);
            \filldraw (7.5,1) circle (2pt);
            \filldraw (6.5,1) circle (2pt);
            \draw (0,0) to (1,1) to (2,2) to (3,1) to (4,1) to (5,2);
            \draw (1,1) to (5,1);
            \draw (6.5,1) to (10.5,1);
            \draw (6.5,3) to (7.5,2) to (8.5,2) to (9.5,1) to (10.5,2);
            \node (a) at (5.75,0) {\scalebox{.5}{$\ldots$}};
            \node (b) at (5.75,2) {\scalebox{.5}{$\ldots$}};
            \node (c) at (5.75,3) {\scalebox{.5}{$\ldots$}};
      \end{tikzpicture}}  \Bigg)
\end{array}$}

      \caption{The decompositions of lattice paths in $\mathcal N$, $\Nn$, $\mathcal V$, and $\Vv$ from~\Cref{prop:N},~\Cref{prop:Nn},~\Cref{prop:V}, and~\Cref{prop:Vv}.}
      \label{fig:n_decomp}
\end{figure}

Define the generating function \[\Nn(x):=\sum_{n\geq 0}x^n\sum_{p\in\mathcal N_n}\overline\d(p).\]

\begin{proposition}\label{prop:Nn}
The generating function $\Nn(x)$ satisfies the functional equation \begin{align*}
\Nn(x)&=2x\Nn(x)+x^2\Mm(x,1)\mathcal N(x)+x^2\mathcal M(x,1)\mathcal N(x)+x^2\mathcal N(x)\frac{\partial}{\partial x}(x\mathcal M(x,1)) \\
&+x\Nn(x)+x\frac{\partial}{\partial x}(x\mathcal N(x)).
\end{align*}
Thus, \[\Nn(x)=\frac{4x\left(1+\sqrt{1-4x}-x(1-\sqrt{1-4x})\right)}{\sqrt{1-4x}(1-4x+\sqrt{1-4x})^3}.\]
\end{proposition}

\begin{proof}
Following the same ideas used in the proof of \Cref{prop:mm_decomp}, we find that \[2x\Nn(x)+x^2\Mm(x,1)\mathcal N(x)+x^2\mathcal M(x,1)\mathcal N(x)+x^2\mathcal N(x)\frac{\partial}{\partial x}(x\mathcal M(x,1))\] counts bicolored Motzkin prefixes that touch the $x$-axis at some point other than $(0,0)$, with each path $p$ weighted by $\overline d(p)$. The generating function for bicolored Motzkin prefixes that only touch the $x$-axis at $(0,0)$ (with each path $p$ weighted by $\overline d(p)$) is \[x\Nn(x)+x\frac{\partial}{\partial x}(x\mathcal N(x)).\]  This is illustrated on the second line of~\Cref{fig:n_decomp}.
This yields the functional equation, from which the explicit solution is straightforward to obtain via \Cref{prop:m_decomp,prop:mm_decomp,prop:N}. 
\end{proof}

Let \[\mathcal V(x):=\sum_{n\geq 0}|\mathcal V_n|x^n.\] 

\begin{proposition}\label{prop:V}
The generating function $\mathcal V(x)$ satisfies the functional equation \[\mathcal V(x)=1+2x\mathcal N(x)+2x\mathcal V(x)+2x^2\mathcal M(x,1)\mathcal V(x).\] Thus, \[\mathcal V(x)=\frac{1+\sqrt{1-4x}}{\sqrt{1-4x}(1-4x+\sqrt{1-4x})}.\]
\end{proposition}
\begin{proof}
The expression $1+2x\mathcal N(x)$ counts (possibly empty) bilateral Motzkin prefixes that only touch the $x$-axis at $(0,0)$, while the expression $2x\mathcal V(x)+2x^2\mathcal M(x,1)\mathcal V(x)$ counts bilateral Motzkin prefixes that touch the $x$-axis at some point other than $(0,0)$.  This is illustrated on the third line of~\Cref{fig:n_decomp}. The explicit solution can be derived from the functional equation using \Cref{prop:m_decomp,prop:N}. 
\end{proof}

Finally, consider the generating function \[\Vv(x):=\sum_{n\geq 0}x^n\sum_{p\in\mathcal V_n}\overline\d(p).\]
\begin{proposition}\label{prop:Vv}
The generating function $\Vv(x)$ satisfies the functional equation \begin{align*}
\Vv(x)&=2x\Vv(x)+2x^2\Mm(x,1)\mathcal V(x)+2x^2\mathcal M(x,1)\Vv(x)+2x^2\mathcal V(x)\frac{\partial}{\partial x}(x\mathcal M(x,1)) \\
&+2x\Nn(x)+2x\frac{\partial}{\partial x}(x\mathcal N(x)).
\end{align*}
Thus, \begin{equation}\label{eq:Vv}\Vv(x)=\frac{8x\left(1+\sqrt{1-4x}-x(3+\sqrt{1-4x})\right)}{(1-4x)(1-4x+\sqrt{1-4x})^3}.\end{equation}
\end{proposition}

\begin{proof}
Following the same ideas used in the proof of \Cref{prop:Nn}, we find that \[2x\Vv(x)+2x^2\Mm(x,1)\mathcal V(x)+2x^2\mathcal M(x,1)\Vv(x)+2x^2\mathcal V(x)\frac{\partial}{\partial x}(x\mathcal M(x,1))\] counts bilateral Motzkin prefixes that touch the $x$-axis at some point other than $(0,0)$, with each path $p$ weighted by $\overline d(p)$. Furthermore, the generating function for bilateral Motzkin prefixes that only touch the $x$-axis at $(0,0)$ (with each path $p$ weighted by $\overline d(p)$) is \[2x\Nn(x)+2x\frac{\partial}{\partial x}(x\mathcal N(x)).\]  This is illustrated on the fourth line of~\Cref{fig:n_decomp}.
This yields the functional equation, from which one can derive the explicit solution using \Cref{prop:m_decomp,prop:mm_decomp,,prop:N,,prop:V}. 
\end{proof}

We conclude \Cref{cor:Wiener_JSn} by expanding the explicit generating function for $\Vv(x)$ given in \Cref{eq:Vv} as \[\frac{8x\left(1+\sqrt{1-4x}-x(3+\sqrt{1-4x})\right)}{(1-4x)(1-4x+\sqrt{1-4x})^3} = \sum_{n\geq 0} a_n x^n.\]
 Using a computer algebra system, the coefficients $a_n$ satisfy the difference equation
\[(2 n+3)2^{-2n-1} a_n - (4 n +5)2^{-2n-3} a_{n + 1} + (2 + 2 n)2^{-2n-5} a_{n + 2} = 0\]
with initial conditions $a_0=0$ and $a_1=2$.  It is easily checked that $\frac{2 n (2 n + 1)}{3}\binom{2 n - 1}{n}$ satisfies this equation and initial conditions.

\section{Asymptotic distributions}\label{sec:asymptotics}

In this section, we prove \cref{prop:cv_law_distance,prop:cv_law_distance_shifted}, which describe the asymptotic distribution
of the distance between 2 random points (also called $2$-point distance) in $\P_{\alpha n,n}$ and $\SS_n$, respectively.
We start with the case of shifted staircases, which is easier.

\subsection{2-point distance in $\SS_n$}
Recall that the elements in $\SS_n$
are exactly the lattice paths starting at (0,0),
ending somewhere on the line $x=n$, and using steps 
of the form $U=(1,1)$ and $D=(1,-1)$. 
Let $p^n$ and $q^n$ be independent uniform random elements in $\SS_n$.
Seeing $p^n$ and $q^n$ as lattice paths, we write $p^n_i$
and $q^n_i$
for their heights after $i$ steps.
Clearly, for all $n \ge 1$ and $i \le n$,
one has $p^n_i=X_1+\dots+X_i$ and
$q^n_i=Y_1+\dots+Y_i$,
where $(X_j)_{j \ge 1}$ and $(Y_j)_{j \ge 1}$ are
independent sequences of i.i.d.~Rademacher random variables of parameter $1/2$.
Using \Cref{eq:distance_shifted}, we write
\[d(p^n,q^n)= \frac12 \sum_{i=1}^n |p^n_i-q^n_i|
= \frac{n}2 \int_0^1 \lvert\, p^n_{\lceil nt \rceil}-q^n_{\lceil nt \rceil}\rvert dt.\]

By Donsker's theorem, the processes
\[\left(\tfrac1{\sqrt n} p^n_{\lceil nt \rceil}\right)_{t \le 1} \text{ and }\left(\tfrac1{\sqrt n} q^n_{\lceil nt \rceil}\right)_{t \le 1}  \]
converge in distribution to independent Brownian motions $(B_t)_{t \le 1}$ and $(B'_t)_{t \le 1}$
in Skorokhod space $D[0,1]$
(see \cite[Chapter 3]{billingsley_convergence} for background on Skorokhod space). Since integration is a continuous functional on $D[0,1]$,
we have
\[n^{-3/2} d(p^n,q^n)= \frac12
\int_0^1 \left\lvert n^{-1/2} p^n_{\lceil nt \rceil}-n^{-1/2} q^n_{\lceil nt \rceil}\right\rvert dt \stackrel{d}{\longrightarrow} \frac12
\int_0^1 |B_t-B'_t| dt,\]
where $\stackrel{d}{\longrightarrow}$ means convergence in distribution.
But $B_t-B'_t \stackrel{d}= \sqrt 2 \, B_t$,
proving that $n^{-3/2} d(p^n,q^n)$
converges in distribution to $\frac1{\sqrt 2} \int_0^1 |B_t| dt$, as claimed in \cref{prop:cv_law_distance_shifted}.

It remains to prove moment convergence. By \cite[Corollary of Theorem 25.12]{billingsley_probability}, it suffices to show that for each $s>1$, the sequence of $s$th moments of $n^{-3/2} d(p^n,q^n)$ is bounded as $n$ tends to $+\infty$.
We have
\[n^{-3/2} d(p^n,q^n) \le n^{-1/2} \max_{i \le n} |p^n_i|
+ n^{-1/2} \max_{i \le n} |q^n_i|.\]
Both terms in the upper bound are identically distributed, so we only consider the first one.
By Doob's maximal inequality,
we have
\[ \E\left[\left(\max_{i \le n} |p^n_i|\right)^s\right] \le \left(\frac{s}{s-1}\right)^s
\E\Big[|p^n_n|^s\Big].\]
Since $p^n_n$ is a sum of $n$ i.i.d.~{\em centered}
random variables, we have the following classical bound on its moments
(see, e.g., \cite{petrov1989moments}):
\[ \E\Big[|p^n_n|^s\Big]
\le C(s)\, n^{s/2}\, \E\big[ |X_1|^s \big],\]
where $C(s)$ is a constant depending only on $s$.
In particular the $s$th moment of $n^{-1/2} p^n_n$ is bounded (as $n$ tends to $+\infty$).
Consequently, the $s$th moment of 
$n^{-1/2} \max_{i \le n} p^n_i$
is bounded, and that of $n^{-3/2} d(p^n,q^n)$
is as well. This proves that the convergence
of $n^{-3/2} d(p^n,q^n)$ to $\frac1{\sqrt 2} \int_0^1 |B_t| dt$ holds also in moments, concluding the proof of \cref{prop:cv_law_distance_shifted}. \qed

\subsection{2-point distance in $P_{\alpha n,n}$}
We now turn to the case of rectangles.
Let $p^n$ and $q^n$ be independent uniform random elements in $P_{\alpha n,n}$, seen as lattice paths
from $(0,0)$ to ${((\alpha+1)n, (\alpha-1)n)}$.
These paths  $p^n$ and $q^n$ can be constructed as partial sums of sequences of i.i.d.~random variables {\em under some conditioning}.
To this end, let $(X_j)_{j \ge 1}$ and $(Y_j)_{j \ge 1}$ be
independent sequences of i.i.d.~Rademacher random variables of parameter $\alpha/(\alpha+1)$.
We also let $(\tilde X^n_j)_{j \ge 1}$ have the distribution
of $(X_j)_{j \ge 1}$ conditioned to the event $\sum_{j \le (\alpha+1)n} X_j=(\alpha-1)n$.
Then one has the equality in distribution
\[\big(p^n_i\big)_{i\le (\alpha+1)n} \stackrel{d}=
\left( \sum_{j \le i} \tilde X^n_j \right)_{i\le (\alpha+1)n}.\]

Recall that we are interested in the quantity
\begin{equation}\label{eq:distance_shifted_integral}
n^{-3/2} D_{\alpha,n} = n^{-3/2} d(p^n,q^n) =
n\frac1{2n^{3/2}} \sum_{i=1}^n |p^n_i-q^n_i|
= \frac{\alpha+1}{2 \sqrt n} \int_0^1 |p^n_{\lceil (\alpha+1)nt \rceil}-q^n_{\lceil (\alpha+1)nt \rceil}| dt.
\end{equation}
A version of Donsker's theorem for conditioned partial sums has been proved by Liggett \cite{liggett1968invariance} (see in particular the corollary of Theorem 4 there).
In our case, the centered process
\[\left(\tfrac1{\sigma \, \sqrt {(1+\alpha)n}} \big(p^n_{\lceil (\alpha+1)nt \rceil} - \lceil nt \rceil (\alpha-1) \big)\right)_{0 \le t \le 1} \]
converges in distribution to $B_0(t)$
in Skorokhod space $D[0,1]$,
where $\sigma^2=\Var(X_1)$ and $B_0(t)$ is a Brownian bridge.
A similar convergence result holds for $q^n$
with an independent Brownian bridge $B'_0(t)$.

Using the continuity of taking integrals on $D[0,1]$, the quantity in \eqref{eq:distance_shifted_integral}
converges in distribution to
\[ \frac12 (\alpha+1) \, \sigma\, \sqrt{\alpha+1} \int_0^1 |B_0(t) -B'_0(t)| dt.\]
An easy computation gives $\sigma=2\sqrt{\alpha}/(\alpha+1)$, while $B_0(t) -B'_0(t)\stackrel{d}=\sqrt{2} B_0(t)$ in distribution.
Consequently, $n^{-3/2} D_{\alpha,n}$ converges in distribution
to $\sqrt{2\alpha(\alpha+1)} \int_0^1 |B_0(t)| dt$, as claimed in \cref{prop:cv_law_distance}.

It remains to prove moment convergence.
As above, we shall prove that for any $s>1$, the random variable $n^{-3/2} D_{\alpha,n}$
has a bounded $s$th moment as $n$ tends to $+\infty$.
Using the convexity of the map $t \mapsto |t|^s$, we obtain
\begin{equation}
\label{eq:bounding_Ens}
n^{-s} D_{\alpha,n}^s = 2^{-s} \left(\frac1{n} \sum_{i \le (\alpha+1)n} |p^n_i-q^n_i|\right)^s
\le \frac{2^{-s}}n \sum_{i \le (\alpha+1)n} 
|p^n_i-q^n_i|^s \le \frac{1}{n} \sum_{i \le (\alpha+1)n}
\frac{|\bar p^n_i|^s+|\bar q^n_i|^s}{2},
\end{equation}
where $\bar p^n_i=p^n_i -i \frac{\alpha-1}{\alpha+1}$
is the centered version of $p^n_i$ (and idem for $q$). Writing $\bar X_i=X_i - \frac{\alpha-1}{\alpha+1}$, we have
\[\E\big[ |\bar p^n_i|^s \bar]
= \E\left[ \bigg| \sum_{j \le i} \bar X_j \bigg|^s  \Bigg| \sum_{j \le (\alpha+1)n} \bar X_j=0 \right]
= \sum_{k} |k|^s\, \mathbb P\left[ \sum_{j \le i} \bar X_j =k  \Bigg| \sum_{j \le (\alpha+1)n} \bar X_j=0 \right],
\]
where the sum runs over possible values $k$ for $\sum_{j \le i} \bar X_j$.
Using the independence of the $\bar X_j$, we have
\begin{align*}
    \mathbb P\left[ \sum_{j \le i} \bar X_j =k  \Bigg| \sum_{j \le (\alpha+1)n} \bar X_j=0 \right] &=\frac{\mathbb P\left[ \sum_{j \le i} \bar X_j =k \ \wedge \ \sum_{j \le (\alpha+1)n} \bar X_j=0\right]}
    {\mathbb P \left[\sum_{j \le (\alpha+1)n} \bar X_j=0 \right]} \\
& = \mathbb P\left[ \textstyle \sum_{j \le i} \bar X_j =k\right] \cdot 
\frac{\mathbb P\left[ \sum_{i< j \le (\alpha+1)n} \bar X_j=-k\right]}{\mathbb P \left[\sum_{j \le (\alpha+1)n} \bar X_j=0 \right]}.
\end{align*}
Take $i \le (\alpha+1)n/2$.
The probabilities in the fraction can be evaluated asymptotically---uniformly in $k$---through the local limit theorem (see, e.g., \cite[Theorem 3.5.2]{durrett2019probability}), which yields
\begin{align*}
   \mathbb P \left[\sum_{j \le (\alpha+1)n} \bar X_j=0 \right] &\sim \frac{2}{\sigma \sqrt{2 \pi (\alpha+1) n}};\\
\mathbb P\left[ \sum_{i< j \le (\alpha+1)n} \bar X_j=-k\right] &=
\frac{2 e^{-k^2/2((\alpha+1)n-i) \sigma^2}}{\sigma \sqrt{2 \pi ((\alpha+1) n-i)}} + o(n^{-1/2}) \le \frac{2+o(1)}{\sigma \sqrt{\pi (\alpha+1) n}}.
\end{align*}
In particular, the quotient is bounded by $2$ for $n$ large enough, uniformly in $k$.
Bringing everything together, we obtain that for $n$ large enough and $i \le (\alpha+1)n/2$,
\[\E\big[ |\bar p^n_i|^s \bar]
\le \sum_{k} |k|^s \cdot 2 \mathbb P\left[ \textstyle \sum_{j \le i} \bar X_j =k\right]
= 2 \E\left[ \bigg| \sum_{j \le i} \bar X_j \bigg|^s
\right].\]
Since the $\bar X_j$ are i.i.d.~{\em centered} random variables with finite moments, we have (see, e.g., \cite{petrov1989moments})
\[ \E\left[ \bigg| \sum_{j \le i} \bar X_j \bigg|^s
\right] \le C(s) i^{s/2} \E\big[ |X_1|^s \big],\]
where $C(s)$ is a constant depending only on $s$
(and $\alpha$ in the sequel),
which may change from line to line.
Therefore, for $n$ large enough and $i \le (\alpha+1)n/2$, we
have 
\[\E\big[ |\bar p^n_i|^s \bar] \le  C(s)\,  n^{s/2}. \]
By symmetry, this holds also for $i \ge (\alpha+1)n/2$ (we have $p^n_i \stackrel{d}= p^n_{(\alpha+1)n-i}$ for all $i \le (\alpha+1)n$).
Going back to \eqref{eq:bounding_Ens}, we get
\[n^{-s} \E \big[ D_{\alpha,n}^s \big]
\le (\alpha+1)\, C(s)\, n^{s/2}.\]
Thus $n^{-3/2} D_{\alpha,n}$ has bounded moments,
and the convergence to 
$\sqrt{2\alpha(\alpha+1)} \int_0^1 |B_0(t)| dt$
holds also in moments.
\cref{prop:cv_law_distance_shifted} is proved. \qed

\section{Higher moments}
\label{sec:higher_moments}

Given a finite graph $G=(V,E)$ and a positive integer $r$, let $d^r(G)$ denote the moment $d^r(G)=\sum_{(p,q)\in V\times V}d(p,q)^r$. 
The convergence of the distance between two random elements in distribution and in moments established in the previous section
yield some asymptotic estimates for $d^r(P_{\alpha n,n})$ and $d^r(Q_n)$.
In this section, we give an exact expression of $d^2(P_{k,n-k})$. The same method can, in principle, be used to compute the moments $d^r(P_{k,n-k})$ one by one. Similarly, one could use a similar method, drawing from the ideas in \Cref{sec:bij_stair,sec:proof_stair}, to compute the moments $d^r(Q_n)$. For the sake of brevity, we merely state the explicit formula for $d^2(Q_n)$ and omit the computation.

\begin{proposition} We have
    \[\d^2(P_{m,k}) = \frac{1}{30} \frac{m+k+1}{m+k} \binom{m+k}{m-1}\binom{m+k}{k-1} \left( 7mk^2 + 7m^2k + 3m^2 + 10mk+ 3k^2  + 3m + 3k + 4 \right)\] and \[\d^2(Q_{n}) = 2^{2n-4} n \left(20 + 15 (n - 2) + 3 (n - 2)^2\right).\]
\end{proposition}

\begin{proof}
As mentioned above, we will only prove the first formula. Given a bilateral Motzkin path $p$, let $\len(p)$ denote the length of $p$, and let $\mathcal U(p)$ be the number of steps in $p$ of the form $U$ or $O_1$. Recall that \[\W(x,u)=\sum_{p\in\W}x^{\len(p)}u^{\mathcal U(p)},\quad \Ww(x,u)=\sum_{p\in \W}x^{\len(p)}u^{U(p)}d(p),\] \[\M(x,u)=\sum_{p\in\M}x^{\len(p)}u^{\mathcal U(p)}, \quad\Mm(x,u)=\sum_{p\in \M}x^{\len(p)}u^{U(p)}d(p).\] Let \[\WWW(x,u)=\sum_{p\in \W}x^{\len(p)}u^{U(p)}d(p)^2\quad\text{and}\quad\MMM(x,u)=\sum_{p\in \M}x^{\len(p)}u^{U(p)}d(p)^2.\] It follows from \Cref{prop:bija} that $\WWW(x,u)=\sum_{n\geq 0}\sum_{k\geq 0}d^2(P_{k,n-k})x^nu^k$. 

The contribution to $\MMM(x,u)$ coming from paths that start with $O_1$ or $O_2$ is $x(u+1)\MMM(x,u)$. The other paths that contribute to $\MMM(x,u)$ begin with $U$ and have the form $UpDq$ for some $p,q\in\M$. We find that \begin{equation}\label{eq:Sigma}
\MMM(x,u)=x(u+1)\MMM(x,u)+ux^2\Sigma,
\end{equation}
where \[\Sigma=\sum_{p,q\in\M}x^{\len(p)+\len(q)}u^{\mathcal U(p)+\mathcal U(q)}(d(p)+d(q)+\len(p)+1)^2.\] We can write 
\begin{align*}
(d(p)+d(q)+\len(p)+1)^2&=(d(p)^2+d(q)^2)+(\len(p)+1)^2+2d(p)(\len(p)+1) \\ 
&+2d(q)(\len(p)+1)+2d(p)d(q)
\end{align*}
to find that 
\begin{align}\label{eq:Sigma2}
\Sigma&=2\MMM(x,u)\M(x,u)+\M(x,u)\frac{\partial}{\partial x}\left(x\frac{\partial}{\partial x}(x\M(x,u))\right)+2\M(x,u)\frac{\partial}{\partial x}(x\Mm(x,u)) \nonumber\\ 
&+2\Mm(x,u)\frac{\partial}{\partial x}(x\M(x,u))+2\Mm(x,u)^2.
\end{align}

A similar argument yields the functional equation 
\begin{equation}\label{eq:Sigma'}
\WWW(x,u)=x(u+1)\WWW(x,u)+2ux^2\Sigma', 
\end{equation}
where 
\begin{align}\label{eq:Sigma'2}
\Sigma'&=\sum_{\substack{p\in\M \\ q\in\W}}x^{\len(p)+\len(q)}u^{\mathcal U(p)+\mathcal U(q)}(d(p)+d(q)+\len(p)+1)^2 \nonumber\\ 
&=\MMM(x,u)\W(x,u)+\WWW(x,u)\M(x,u)+\W(x,u)\frac{\partial}{\partial x}\left(x\frac{\partial}{\partial x}(x\M(x,u))\right) \nonumber \\ 
&+2\W(x,u)\frac{\partial}{\partial x}(x\Mm(x,u))+2\Ww(x,u)\frac{\partial}{\partial x}(x\M(x,u))+2\Mm(x,u)\Ww(x,u).
\end{align}

We already computed explicit formulas for $\M(x,u)$, $\Mm(x,u)$, $\W(x,u)$, and $\Ww(x,u)$ in \Cref{prop:m_decomp,prop:mb_decomp,prop:mm_decomp,prop:mbm_decomp}. Combining those formulas with \Cref{eq:Sigma,eq:Sigma2,eq:Sigma',eq:Sigma'2}, we can derive the explicit formula
\[\WWW(x,u)=\frac{2 u x^2 \left((u-1)^2 (u+1) x^3-((u-8) u+1) x^2-(u+1) x+1\right)}{\left((u x+x-1)^2-4 u x^2\right)^{7/2}}.\] Setting $u=y/x$ and extracting coefficients yields the desired explicit formula for $d^2(P_{m,k})$. 
\end{proof}

\section{Open problems}

Comparing the results of Proposition~\ref{prop:mb_decomp} and~\ref{prop:mbm_decomp}, we get the intriguing equation
\begin{equation}
\label{eq:interesting_identity}
\Mbm=2ux^2\;\Mb^4.
\end{equation} 

A direct proof of this would be interesting in itself and could lead to a bijective proof of~\Cref{cor:coeff} via the explicit formula \eqref{eq:expansion_bilateral}.  Recall that $\Mbm$ counts pairs of paths where a cell in the symmetric difference is marked.  The connected component where this cell occurs corresponds to a part where the two paths only meet at their beginning and end (this forms a \emph{parallelogram polyomino}), and the generating function $\Mb^2$ naturally enumerates the rest of the paths. It follows that a bijective proof of \eqref{eq:interesting_identity} reduces to a bijective proof that $\Mb^2$ enumerates the total area of parallelogram polyominoes. The specialization $u=1$ is known \cite{dellungo2004bijection}.

As minuscule lattices arise as the weak order on certain maximal parabolic quotients of finite Coxeter groups, it would be interesting to extend our results to other parabolic quotients.  Minuscule lattices also appear as certain crystal graphs; one could also ask about the Wiener indices of more general crystals.

\section*{Acknowledgements}
V.F., P.N., and N.W.~thank the organizers of the \href{https://www.mfo.de/occasion/2250/www_view}{2022 Oberwolfach conference in Enumerative Combinatorics} for inviting them to a wonderful conference.   P.N.~was partially supported by the project ANR19-CE48-011-01 (COMBIN\'E).  N.W.~was partially supported by Simons Foundation Collaboration Grant No.~585380.
V.F.~is partially supported by the Program ``Future Leader''
of the Initiative ``Lorraine Universit\'e d'Excellence'' (LUE). C.D. was supported by the National Science Foundation under Award No. 2201907 and by a Benjamin Peirce Fellowship at Harvard University.

Computer exploration to discover the simple exact formulas established in this paper have been carried out
with the computer algebra software SageMath \cite{sagemath}, and we are grateful to its developers for their work.

\bibliographystyle{bibli_perso}
\renewcommand*{\bibliofont}{\normalsize}
\bibliography{distance}

\end{document}